\documentclass[reqno,10pt]{amsart}
\usepackage{amscd,amssymb,amsmath,amsthm}
\usepackage[dvips]{graphicx}
\usepackage[dvips]{color}
\usepackage[all]{xy}
\theoremstyle{plain}
\newtheorem{proposition}{Proposition}
\newtheorem{theorem}[proposition]{Theorem}

\newtheorem{conjecture}[proposition]{Conjecture}
\newtheorem*{conjecture*}{Conjecture}
\newtheorem{definition}[proposition]{Definition}
\newtheorem{corollary}[proposition]{Corollary}
\newtheorem{lemma}[proposition]{Lemma}
\newtheorem{remark}[proposition]{Remark}

\newtheorem{example}[proposition]{Example}

\newtheorem{proposition-definition}[proposition]{Proposition/Definition}
\newtheorem*{proposition*}{Proposition}
\newtheorem*{theorem*}{Theorem}
\newtheorem*{maintheorem*}{Main Theorem}
\newtheorem*{maincorollary*}{Main Corollary}
\newtheorem*{corollary*}{Corollary}
\newtheorem*{lemma*}{Lemma}
\newtheorem*{remark*}{Remark}
\newtheorem*{definition*}{Definition}
\newtheorem*{example*}{Example}
\newtheorem*{examples*}{Examples}
\newtheorem*{criterion*}{Generation Criterion}
\newtheorem*{explanation*}{Explanation}

\numberwithin{proposition}{section}
\numberwithin{equation}{section}

\def\co{\colon\thinspace}

\newcommand{\N}{\mathbb{N}}
\newcommand{\Z}{\mathbb{Z}}

\newcommand{\R}{\mathbb{R}}
\newcommand{\C}{\mathbb{C}}
\newcommand{\K}{\mathbb{K}}
\newcommand{\D}{\mathbb{D}}

\newcommand{\CP}{\C\mathbb{P}}

\newcommand{\coker}{\mathrm{coker}\,}

\renewcommand{\P}{\mathbb{P}}

\newcommand{\Log}{\mathrm{Log}\,}

\begin{document}

\title[The Fukaya category of Fano toric varieties]
{Circle-actions, quantum cohomology, and the Fukaya category of Fano toric varieties}

\author{Alexander F. Ritter}
\address{A. F. Ritter, Mathematical Institute, University of Oxford, England.}
\email{ritter@maths.ox.ac.uk}

\date{version: \today}

\begin{abstract}
We define a class of non-compact Fano toric manifolds, called \emph{admissible toric manifolds}, for which Floer theory and quantum cohomology are defined. The class includes Fano toric negative line bundles, and it allows blow-ups along fixed point sets.

We prove closed-string mirror symmetry for this class of manifolds: the Jacobian ring of the superpotential is the symplectic cohomology (not the quantum cohomology). Moreover, $SH^*(M)$ is obtained from $QH^*(M)$ by localizing at the toric divisors. We give explicit presentations of $SH^*(M)$ and $QH^*(M)$, using ideas of Batyrev, McDuff and Tolman. 

Assuming that the superpotential is Morse (or a milder semisimplicity assumption), we prove that the wrapped Fukaya category for this class of manifolds satisfies the toric generation criterion, i.e. is split-generated by the natural Lagrangian torus fibres of the moment map taken with suitable holonomies. In particular, the wrapped category is compactly generated and cohomologically finite.

We prove a generic generation theorem: a generic deformation of the monotone toric symplectic form defines a local system for which the twisted wrapped Fukaya category satisfies the toric generation criterion. This theorem, together with a limiting argument about continuity of eigenspaces, are used to prove the untwisted generation results.

We prove that for any closed Fano toric manifold, and a generic local system, the twisted Fukaya category satisfies the toric generation criterion. If the superpotential is Morse (or assuming semisimplicity), also the untwisted Fukaya category satisfies the criterion.

The key ingredients are non-vanishing results for the open-closed string map, using tools from the paper by Ritter-Smith (we also prove a conjecture from that paper that any monotone toric negative line bundle contains a non-displaceable monotone Lagrangian torus). The above presentation results require foundational work: we extend the class of Hamiltonians for which the maximum principle holds for symplectic manifolds conical at infinity, thus extending the class of Hamiltonian circle actions for which invertible elements can be constructed in $SH^*(M)$. Computing $SH^*(M)$ is notoriously hard and there are very few known examples beyond the case of cotangent bundles and subcritical Stein manifolds. So this computation is significant in itself, as well as being the key ingredient in proving the above results in homological mirror symmetry.
\end{abstract}
\maketitle
\setcounter{tocdepth}{1}
\tableofcontents
%
\section{Introduction}
%
\subsection{Circle actions, generators and relations in Floer cohomology}
The goal of this paper is to develop techniques to produce generators and relations in the quantum cohomology $QH^*(M)$ and in the symplectic cohomology $SH^*(M)$ of monotone symplectic manifolds conical at infinity, by exploiting Hamiltonian circle-actions, and to deduce generation results for the Fukaya category $\mathcal{F}(M)$ and the wrapped Fukaya category $\mathcal{W}(M)$. To clarify the big picture, we recall that by foundational work of Ritter-Smith \cite{RitterSmith}, there is a commutative diagram, called the \emph{acceleration diagram}, relating all these invariants via the open-closed string map $\mathcal{OC}$ on Hochschild homology:
\begin{equation}\label{Equation Intro Comm Diagram Ritter Smith}
\xymatrix{ 
\mathrm{HH}_*(\mathcal{F}(M)) \ar@{->}[r]^-{} \ar@{->}[d]_-{\mathcal{OC}} & \mathrm{HH}_*(\mathcal{W}(M)) 
 \ar@{->}^{\mathcal{OC}}[d] \\
QH^*(M)
\ar@{->}[r]^-{c^*}  & SH^*(M)}
\end{equation}
Recall that in the compact Fukaya category $\mathcal{F}(M)$ one only allows \emph{closed} monotone orientable Lagrangian submanifolds, and one works with pseudo-holomorphic maps, whereas for the wrapped Fukaya category $\mathcal{W}(M)$ one allows also \emph{non-compact} Lagrangians which are Legendrian at infinity, and the definition of the morphism spaces involve wrapping one Lagrangian around the other by using a Hamiltonian flow (more precisely, a direct limit construction is needed just like for $SH^*(M)$, where one increases the growth of the Hamiltonian at infinity).

The tool for producing generators and relations in $QH^*(M)$ and $SH^*(M)$ comes from our foundational paper \cite{Ritter4}, which constructs a commutative diagram of ring homomorphisms
\begin{equation}\label{Equation Intro Comm Diagram Ham SH and QH}
\xymatrix{ 
 & SH^*(M)
 \ar@{<-}^{c^*}[d] \\
\widetilde{\pi_1}(\mathrm{Ham}_{\ell\geq 0}(M,\omega))
\ar@{->}[r]^-{r}  \ar@{->}^{\mathcal{R}}[ur] & QH^*(M)}
\end{equation}
Namely, from loops of Hamiltonian diffeomorphisms on $M$ we construct elements in $QH^*(M)$ and $SH^*(M)$, by generalizing the ideas of the Seidel representation \cite{SeidelGAFA} to the non-compact setting. The non-compact setting brings with it technical difficulties, such as the necessity of a \emph{maximum principle} which prevents holomorphic curves and Floer solutions from escaping to infinity. These seemingly technical constraints are responsible for interesting algebro-homological consequences: for example, the elements above are invertible in $SH^*(M)$ but not in $QH^*(M)$ (whereas in the compact setting, Seidel elements were always invertible in quantum cohomology). Further down the line, this phenomenon will also imply that for many non-compact monotone toric manifolds, 
\begin{equation}\label{Equation Introduction SH=Jac}
SH^*(M)\cong \mathrm{Jac}(W_M)
\end{equation}
recovers the Jacobian ring of the Landau-Ginzburg superpotential $W_M$ (a purely combinatorial invariant associated to the moment polytope of $M$) whereas one may have first imagined that the isomorphism $QH^*(C)\cong \mathrm{Jac}(W_C)$ for closed monotone toric manifolds $C$ (Batyrev \cite{Batyrev}, Givental \cite{Givental,Givental2}) should persist also in the non-compact setting (see Example \ref{Example failure of QH=Jac}). 

The result \eqref{Equation Introduction SH=Jac} is also noteworthy because it is a proof of closed-string mirror symmetry.
\\[1mm]
\emph{{\bf Convention.} In this paper, we always restrict $SH^*(M)\equiv SH^*_0(M)$ to the component of contractible loops. For simply connected $M$, e.g. any toric manifold, this is all of $SH^*(M)$. Recall also that $SH^*_0(M)$ contains the unit $c^*(1)$, so  if $SH^*_0(M)=0$ then all of $SH^*(M)=0$.}
\\[1mm]
\indent As an illustration of diagram \eqref{Equation Intro Comm Diagram Ham SH and QH}, we recall the following application from \cite{Ritter4}, which exploited the circle action given by rotation in the fibres of a line bundle. 

\begin{theorem}\label{Introduction theorem SH of E is QH modulo ker}
Let $\pi:E \to B$ be a complex line bundle over a closed symplectic manifold $(B,\omega_B)$. Assume $E$ is \emph{negative}, meaning $c_1(E)=-k[\omega_B]$ for some $k>0$. Then
\begin{equation}\label{Equation Introduction localization}
SH^*(E)\cong QH^*(E)_{\mathrm{PD}[B]}
\end{equation}
are isomorphic algebras, where $QH^*(E)$ has been localized at the zero section.
\end{theorem}

The negativity assumption is needed to ensure that the total space is a symplectic manifold conical at infinity, so Floer theory is defined. The theorem was not phrased in terms of localization in \cite{Ritter4} so we comment on this in Section \ref{Subsection Some remarks about localizations of rings}. More generally, by \cite{Ritter4}, 

\begin{theorem} $SH^*(M)$ is a localization of $QH^*(M)$ whenever there is a Hamiltonian $S^1$-action on $M$ which corresponds to the Reeb flow at infinity.
\end{theorem}

As a vector space, $QH^*(M)$ is just the ordinary cohomology, however it is a non-trivial task to determine its ring structure, especially in the non-compact setting. For example, the non-triviality of $SH^*(E)$ above is equivalent to $\pi^* c_1(E)=\mathrm{PD}[B]$ being non-nilpotent in $QH^*(E)$, and it is not known whether this holds in general (we will prove it for toric $E$).

Even in the deceptively simple case of negative line bundles, $QH^*(E)\cong QH^*(B)$ is false. The moduli spaces of holomorphic
spheres are trapped inside $B$ by the maximum principle, but the dimensions are all wrong as they don't take into account the new fibre direction of $E$. Even for $\mathcal{O}(-k)\to \C P^m$, although we succeeded in \cite{Ritter4} to 
compute $QH^*(E)$ for $1\leq k\leq m/2$ by difficult virtual localization techniques, these calculations became intractable for $k>m/2$.

This reality-check, that determining $QH^*(M)$ is a difficult problem, and that it is an important first step towards understanding the Fukaya category $\mathcal{F}(M)$, motivates our work. Although the high-profile route would be to express the Gromov-Witten invariants in terms of obstruction bundle invariants, packaged up inside a generating function, such approaches would still leave us with the task of actually calculating the obstruction bundle invariants. So in practice, strictly speaking, one hasn't computed the invarians and, for example, this route would not help in understanding how $QH^*(M)$ decomposes into eigenspaces of the $c_1(TM)$-action, which is a key step towards determining generators of $\mathcal{F}(M)$. 

A key issue in constructing the representations $r,\mathcal{R}$ from diagram \eqref{Equation Intro Comm Diagram Ham SH and QH} is to ensure that the Hamiltonian $S^1$-actions are generated by Hamiltonians $H$ which satisfy a certain maximum principle. In \cite{Ritter4}, we proved this holds provided $H=\textrm{constant}\cdot R$ at infinity, where $R$ is the radial coordinate on the conical end. This is the same strong constraint on the class of Hamiltonians that arises in the direct limit construction of $SH^*(M)$. The existence of such an $S^1$-action implies that the Reeb flow at infinity needs to arise from a circle action on $M$. Unfortunately, already for toric negative line bundles, we prove in Section \ref{Subsection The Hamiltonians generating the rotations around the toric divisors} that the natural rotations around the toric divisors are generated by Hamiltonians
\begin{equation}\label{Equation H=fR Hamiltonians}
H(y,R) = f(y)R
\end{equation}
which depend on the coordinate $y$ of the contact hypersurface $\Sigma$ which defines the conical end  (e.g. for line bundles, $\Sigma$ is a sphere subbundle).
In Section \ref{Subsection Invertibles in the symplectic cohomology for a larger class of Hamiltonians} via Appendix C we prove:
\begin{lemma}[Extended Maximum Principle]\label{Introduction Lemma Extended Max principle}
The maximum principle applies to Hamiltonians of the form \eqref{Equation H=fR Hamiltonians} for which $f:\Sigma\to \R$ is invariant under the Reeb flow.
\end{lemma}
This also enlarges the class of Hamiltonians for which $SH^*(M)$ can be defined. Although it is a mild generalization, this result makes the crucial difference between whether or not the Floer-machinery can jump-start.

\subsection{Floer theory for non-compact toric varieties}
We now describe the non-compact Fano toric varities $M$ for which we can define and compute $SH^*(M)$ and $QH^*(M)$.
We omit from this Introduction the discussion of NEF toric varieties (Section \ref{Subsection The Calabi-Yau case}) and the discussion of how our machinery applies to blow-ups of non-compact K\"{a}hler manifolds
(Sections \ref{Subsection Blow-up at a point}-\ref{Subsection Blow-up at a subvariety}). The applicability is limited only by our knowledge of whether a maximum principle applies.

We will describe a family of non-compact monotone toric manifolds $M$ for which \eqref{Equation Introduction SH=Jac} will hold, and for which the analogue of \eqref{Equation Introduction localization} is
$$
SH^*(M)\cong QH^*(M)_{\mathrm{PD}[D_1],\, \ldots\, ,\, \mathrm{PD}[D_r]}
$$
where we localize at the generators corresponding to the toric divisors $D_i\subset M$. Moreover we obtain a combinatorial description of $QH^*(M)$ whose localization is the Jacobian ring.

\noindent
{\bf Conjecture.} More speculatively, our work suggests that in situations where a toric divisor is removed from a toric variety $M$, and one allows the Hamiltonians which define Floer cohomology to grow to infinity near the divisor, then this should correspond to localizing the ring at the $r$-element corresponding to rotation about the divisor.

\begin{definition}\label{Definition admissible toric manifolds}
 A non-compact toric manifold $(X,\omega_X)$ is \emph{\bf{admissible}} if
\begin{enumerate}
 \item\label{Item Admissible conical} $X$ is conical at infinity (see Section \ref{Subsection Symplectic manifolds conical at infinity});
 \item\label{Item Admissible monotone} $X$ is monotone, so $c_1(TX)=\lambda_X\, [\omega_X]$ for some $\lambda_X>0$;
 \item\label{Item Admissible max principle} the rotations $g_i$ about the toric divisors $D_i$ (Section \ref{Subsection Review of the McDuff-Tolman proof of the presentation of QH}) are generated by Hamiltonians $H_i$ with non-negative slope satisfying the maximum principle (Theorem \ref{Theorem maximum principle}, $f\geq 0$);
 \item \label{Item Admissible positive slope} at least one of the $H_i$ has strictly positive slope at infinity ($f>0$ in Theorem \ref{Theorem maximum principle}).
\end{enumerate}
\end{definition}

The conditions required are, essentially, what is necessary to make sense of Floer theory with current technology. The prototypical examples of admissible manifolds are Fano toric negative line bundles, which we discuss in detail later. In Sections \ref{Subsection Blow-up at a point} and \ref{Subsection Blow-up at a subvariety} we prove that the above class of manifolds allows for blow-ups along fixed point sets.

\begin{theorem}\label{Theorem admissible toric manifolds}
Let $(X,\omega_X)$ be any admissible toric manifold, and let $\mathcal{J}$ denote the ideal generated by the linear relations \eqref{EqnLinearRelations} and the quantum Stanley-Reisner relations \eqref{EqnQSRrelation}. Then:
\begin{enumerate}
 \item There is an algebra isomorphism
 $$QH^*(X,\omega_X) \cong \Lambda [x_1,\ldots,x_r] / \mathcal{J}, \quad \mathrm{PD}[D_i] \mapsto x_i.$$
 \item The symplectic cohomology is the localization of the above ring at all $x_i=\mathrm{PD}[D_i]$,
$$
\quad\! SH^*(X,\omega_X) \cong \Lambda [x_1^{\pm 1},\ldots,x_r^{\pm 1}] / \mathcal{J}, \quad c^*(\mathrm{PD}[D_i]) \mapsto x_i,
$$
and $c^*: QH^*(X,\omega_X) \to SH^*(X,\omega_X)$ equals the canonical localization map sending $x_i \mapsto x_i$;

\item The moment polytope of $X$, $\Delta = \{ y\in \R^n: \langle y,e_i \rangle \geq \lambda_i \}$, determines an isomorphism
$$
\psi: QH^*(X,\omega) \to 
R_X/(\partial_{z_1} W,\ldots,\partial_{z_n} W), \quad x_i \mapsto t^{-\lambda_i} z^{e_i},
$$
which sends $c_1(TX)=\sum x_i \mapsto W= \sum t^{-\lambda_i} z^{e_i}$, where $W$ is the superpotential (Section \ref{Subsection Batyrev's argument: from the presentation of QH to JacW}) and the ring $R_X$ is described in Definition \ref{Definition RX algebra};

\item and it determines the following algebra isomorphism $($see Definition \ref{Definition Jacobian ring} for $\mathrm{Jac}(W))$,
$$
SH^*(X,\omega_X) \cong \mathrm{Jac}(W),\quad x_i \mapsto t^{-\lambda_i} z^{e_i},\quad c^*(c_1(TX)) \mapsto W.
$$

\item The canonical map $c^*:QH^*(X)\to SH^*(X)$ corresponds to the localization map 
$$
R_X/(\partial_{z_1} W,\ldots,\partial_{z_n} W) \to \mathrm{Jac}(W),\quad z_i \mapsto z_i,
$$
which is the quotient homomorphism by the ideal generated by the generalized $0$-eigenvectors of multiplication by all $z^{e_i}$ (equivalently of all $z_i$, since $X$ is smooth).
\end{enumerate}
\end{theorem}

Part of the proof involves adapting to the non-compact setting the McDuff-Tolman proof \cite[Sec.5]{McDuffTolman} of the Batyrev presentation of $QH^*(C)$ for closed monotone toric manifolds $C$ (reviewed in Sections \ref{Subsection Review of the Batyrev-Givental presentation of QH}-\ref{Subsection Review of the McDuff-Tolman proof of the presentation of QH}). Recall that the toric divisors $D_i$ are the $\textrm{codim}_{\C}=1$ complex torus orbits and the linear relations are the classical relations they satisfy in $H_{2(\dim_{\C}C-1)}(C,\Z)$. The key observation of McDuff and Tolman is that the remaining necessary relations for the algebra $QH^*(C)$ come from applying the Seidel representation $\pi_1 \textrm{Ham}(C)\to QH^*(C)^{\times}$ to the relations satisfied by the natural circle rotation actions $g_i$ about the toric divisors.
In our non-compact setting, the Seidel representation is replaced by the representations $r,\mathcal{R}$ from diagram \eqref{Equation Intro Comm Diagram Ham SH and QH}, and some care needs to be taken in picking ``lifts'' $g_i^{\wedge}$ of the rotations $g_i$ because one now works with an extension $\widetilde{\pi_1} \textrm{Ham}$ of $\pi_1 \textrm{Ham}$ (Sections \ref{Subsection Invertibles in the symplectic cohomology}, \ref{Subsection The problem with relating the lifted rotations}, \ref{Subsection Rotations and the lifting problem} discuss this technical issue). A key step (Lemma \ref{Lemma R of gi standard lift}) is to compute the $r$-element for the rotations $g_i$, 
$$
r(g_i^{\wedge})=x_i=\mathrm{PD}[D_i]\in QH^2(X,\omega_X)
\quad \textrm{ and } \quad 
\mathcal{R}(g_i^{\wedge}) = c^*\mathrm{PD}[D_i]\in SH^2(X,\omega_X)^{\times}.
$$
This is a consequence of a more general computation which we prove in Section \ref{Subsection Invertibles in the symplectic cohomology for a larger class of Hamiltonians}, as follows.

\begin{lemma}\label{Lemma rg1}\!\!\!\footnote{The Lemma as written assumes that the lift $g^{\wedge}$ exists (preserving constant discs): this holds for example if $D$ is connected. In general, $\mathrm{Fix}(g)$ will have several connected components (see the footnote to Section \ref{Subsection Review of the McDuff-Tolman proof of the presentation of QH}). More generally, one should consider the connected component $D\subset \mathrm{Fix}(g)$ arising as the minimum of the Hamiltonian $H$ generating $g$ (when $H$ is proper, $D:=\min H$ is connected by a classical argument due to Atiyah).
By convention, we choose $H$ to take value zero on $D$, which ensures (subject to the assumptions on $dg_t$ as written in the Lemma) that $$r(g^{\wedge})=\mathrm{PD}[D]+\textrm{(higher order } t \textrm{ terms})$$ with leading coefficient $t^0$.
We now explain why any other components $D'$ of $\mathrm{Fix}(g)$ can only contribute with a strictly positive power of $t$. A $\nabla H$-trajectory from $D$ to $D'$ gives rise, via the $S^1$-action, to a pseudo-holomorphic sphere $S$, in particular $\omega[S]>0$. The lift $g^{\wedge}$ will then send the constant disc $(c_x,x)$ at $x\in D'$ to a constant disc with $S$ attached. This causes the contributions coming from $D'$ (if they exist, as regularity is not guaranteed) to be counted with weight $t^{\omega[S]}$. We refer to McDuff-Tolman \cite[Theorem 1.10]{McDuffTolman} for a detailed discussion of this, keeping in mind they use slightly different conventions (see the footnote to Sec.\ref{Subsection Review of the McDuff-Tolman proof of the presentation of QH}).}
Let $(M,\omega)$ be a monotone K\"{a}hler manifold of dimension $\dim_{\C}M=m$. Suppose that $g\in \pi_1\mathrm{Ham}_{\ell\geq 0}(M)$ acts holomorphically on $M$, and that the fixed point set $$D=\mathrm{Fix}(g)\subset M$$ is a complex submanifold of dimension $d=\dim_{\C}\mathrm{Fix}(g)$. Hence $dg_t$ acts as a unitary map on a unitary splitting $T_xM = T_x D \oplus \C^{m-d}$, for $x\in D$. Suppose further that the eigenvalues of $dg_t: \C^{m-d} \to \C^{m-d}$ define degree $1$ paths $S^1 \to S^1$ (more generally, the clutching construction using $dg_t$ defines a rank $m-d$ holomorphic vector bundle over $\P^1$, which by a theorem of Grothendieck splits as $\oplus \mathcal{O}(n_i)$, and for our Lemma it is enough to assume that all $n_i=-1$).

Let $g^{\wedge}$ be the lift of $g$ which maps the constant disc $(c_x,x)$ to itself, for $x\in D$. Then the Maslov index $I(g^{\wedge})=m-d$, and
\begin{equation}\label{Equation rg1}
r(g^{\wedge}) = \mathrm{PD}[\mathrm{Fix}(g)] + (\textrm{higher order }t\textrm{ terms}) \in QH^{2(m-d)}(M).
\end{equation}
If the fixed point set has codimension $m-d=1$, then
\begin{equation}\label{Equation rg1 if Ig is 1}
r(g^{\wedge}) = \mathrm{PD}[\mathrm{Fix}(g)] \in QH^2(M).
\end{equation}
\end{lemma}

\subsection{Toric negative line bundles}

As an illustration of the general theory of Theorem \ref{Theorem admissible toric manifolds}, we will now determine $QH^*(E)$, $SH^*(E)$ and the wrapped Fukaya category $\mathcal{W}(E)$ for monotone (i.e. Fano) toric negative line bundles $\pi:E\to B$,
in terms of the analogous invariants for the base $B$ (in Section \ref{Subsection The Calabi-Yau case} we also comment on the case where the line bundle is Calabi-Yau). Heuristically this is the symplectic analogue of the \emph{quantum Lefschetz hyperplane theorem} \cite{Lee}: the invariants of the hyperplane section $B\subset E$ are recovered from invariants of the ambient $E$ and a quantum multiplication operation by an Euler class.

To clarify, $E$ is any complex line bundle over a monotone toric manifold $(B,\omega_B)$ with 
$$c_1(E)=-k[\omega_B] \quad \textrm{ where }\quad 1\leq k < \lambda_B.$$ Recall $B$ is monotone if $c_1(TB)=\lambda_B [\omega_B]$ for $\lambda_B>0$, and we may assume $[\omega_B]\in H^2(B,\Z)$ is primitive. The case $k=\lambda_B$ would make the total space Calabi-Yau, but in that case (and more generally for large $k$) it turns out that $SH^*(E)=0$ by \cite{Ritter4}, also $\mathcal{W}(E)$ is not so interesting since it is homologically trivial being a module over $SH^*(E)$.

We appreciate that some may view such $E$ as a rather basic geometric setup. We hope those readers will nevertheless appreciate that these are the first steps in the study of non-compact symplectic manifolds which are not exact, so we are trying to move beyond the case of cotangent bundles and Stein manifolds which are well-studied in the symplectic literature.

\begin{theorem}\label{Theorem Introduction SH is Jac}
$QH^*(E)$ is generated by the toric divisors $x_i=\mathrm{PD}[\pi^{-1}(D_i^B)]$, they satisfy the same linear relations and quantum Stanley-Reisner relations that the generators $x_i=\mathrm{PD}[D_i^B]$ of $QH^*(B)$ satisfy, except for a change in the Novikov variable (defined in Section \ref{Subsection Novikov ring}):
$$\textstyle
\qquad \qquad \qquad \qquad \qquad \qquad t \mapsto t\, (\pi^*c_1(E))^k \qquad (\textrm{where }\pi^*c_1(E)= -(k/\lambda_B) \sum x_i). 
$$
Moreover, $SH^*(E)\cong \mathrm{Jac}(W_E)$ and it is obtained by localizing $QH^*(E)$ at $\pi^*c_1(E)$.
\end{theorem}

Although the above follows from Theorem \ref{Theorem admissible toric manifolds}, some substantial work is involved:
\begin{itemize}
\item We explicitly compute the Hamiltonians generating the rotations about the toric divisors (Theorem \ref{Theorem Hamiltonians for rotations about divisors}) to verify that the maximum principle in Lemma \ref{Introduction Lemma Extended Max principle} applies.

\item We verify that the Hamiltonians involved in the quantum Stanley-Reisner relations have positive slope at infinity (a condition required for $r$-elements in $QH^*$ to be defined, whereas $\mathcal{R}$-elements in $SH^*$ do not require this), see Lemma \ref{Lemma Quantum SR-relations involve positive slope Hamiltonians}.

\item We compute the moment polytope of $E$ in terms of the moment polytope of $B$ (see Appendix A) to be able to compare the presentations of $QH^*(B)$ and $QH^*(E)$. We compare the linear and quantum SR-relations for $B$ and $E$ in Corollary \ref{Corollary SR relns of E in terms of B}.

\item The fact that the quantum SR-relations of $B$ and $E$ differ by the above change of Novikov variable is proved in Theorem \ref{Theorem change of Novikov param}.
\end{itemize}

Given that the presentations of $QH^*(B)$ and $QH^*(E)$ are the same up to a change in Novikov parameter, one might hope there is a natural homomorphism $QH^*(B)\to QH^*(E)$. However this is neither natural nor entirely correct. In the monotone (Fano) regime, using a formal Novikov variable $t$ is not strictly necessary as there are no convergence issues in the definition of the quantum product. However, in the absence of $t$, the Novikov change of variables becomes meaningless. Secondly, over the usual Novikov ring $\Lambda$ from Section \ref{Subsection Novikov ring}, $t$ is invertible, but under the change of variables the image $t\, (\pi^*c_1(E))^k \in QH^*(E)$ is never invertible \cite{Ritter4} (this is an instance of the general fact that $\mathcal{R}$-elements are invertible in $SH^*$ but $r$-elements may not be invertible in $QH^*$). Provided we restrict the Novikov ring, we show in Theorem \ref{Theorem Minimal poly and char poly from B to E} that there are ring homomorphisms 
\begin{equation}\label{Equation Introduction varphi}
\begin{array}{l}
\varphi: QH^*(B;\K[t]) \to QH^*(E;\K[t]) \to SH^*(E;\K[t])\\
\varphi: QH^*(B;\K[t,t^{-1}]) \to SH^*(E;\K[t,t^{-1}]),
\end{array}
\end{equation}
defined by $x_i\mapsto x_i$ and $t\mapsto t (\pi^*c_1(E))^k$. These $\varphi$ are useful in deducing relations in $E$ from relations in $B$. For example, the characteristic polynomial of $[\omega_B]$ maps via $\varphi$ to the characteristic polynomial of $[\omega_E]=\pi^*[\omega_B]$.  

In Section \ref{Subsection Compuation of QH and SH of O(-1,-1)} we compute $QH^*$ and $SH^*$ for $\mathcal{O}(-k)\to \P^m$ and $\mathcal{O}(-1,-1)\to \P^1 \times \P^1$. Recall $QH^*(\P^m) = \Lambda[x]/(x^{1+m}-t)$, where $x^{1+m}-t$ is the characteristic polynomial of $x=[\omega_{\P^m}]$; $\Lambda$ is the Novikov ring (Section \ref{Subsection Novikov ring}). Applying $\varphi$ to $x^{1+m}-t$ we effortlessly deduce:
\begin{equation}\label{Equation Introduction SH of O-k over CPm}
\begin{array}{rcl}
QH^*(\mathcal{O}_{\P^m}(-k)) = \Lambda[x]/(x^{1+m} -t (-kx)^k)\\
SH^*(\mathcal{O}_{\P^m}(-k)) = \Lambda[x]/(x^{1+m-k} -t (-k)^k)
\end{array}
\end{equation}
in the monotone regime, so $1\leq k \leq m$. A vast improvement over the difficulty of the virtual localization methods of \cite{Ritter4} which were only feasible for $k\leq m/2$.

The $\varphi$ are also crucial in Section \ref{Subsection The eigenvalues of c_1(TX)}  to determine the splitting of $QH^*(E)$ into eigensummands of $c_1(TE)=\lambda_E [\omega_E]$ in terms of the splitting of $QH^*(B)$ into eigensummands of $c_1(TB)=\lambda_B [\omega_B]$ (assuming now that the underlying field we work over is algebraically closed). This decomposition plays a key role later in decomposing the Fukaya category of $E$.

\begin{theorem}[Assuming $\mathrm{char}(\K)$ does not divide $k$]\label{Introduction Theorem Eigenvalues of c1 from B to E} The eigensummand decomposition for $QH^*(B)$, as a $\Lambda[x]$-module with $x$ acting by multiplication by $[\omega_B]$, has the form
\begin{equation}\label{Equation primary decomposition of QH(B)}
QH^*(B) \cong \frac{\Lambda[x]}{(x^{\lambda_B}-\mu_1^{\lambda_B}t)^{d_1}} \oplus \cdots \oplus\frac{\Lambda[x]}{(x^{\lambda_B}-\mu_p^{\lambda_B}t)^{d_p}} \oplus \frac{\Lambda[x]}{x^{d_{p+1}}} \oplus \cdots \oplus \frac{\Lambda[x]}{x^{d_q}}.
\end{equation}
It determines the following $\Lambda[x]$-module isomorphisms where $x$ acts as $[\omega_E]=\pi^*[\omega_B]$,
\begin{equation}\label{Equation primary decomposition of QH(E)}
\begin{array}{rcl}
SH^*(E) &\cong & \displaystyle \frac{\Lambda[x]}{[x^{\lambda_B-k}-(-k)^k\mu_1^{\lambda_B}t]^{d_1}} \oplus \cdots \oplus\frac{\Lambda[x]}{[x^{\lambda_B-k}-(-k)^k\mu_p^{\lambda_B}t]^{d_p}}\\[4mm]
QH^*(E) &\cong & SH^*(E)\oplus \ker(x^{\mathrm{large}}).
\end{array}
\end{equation}
\end{theorem}

In Ritter-Smith \cite{RitterSmithVersion1} (Version 1) we conjectured that $SH^*(E)$ was non-zero for all monotone toric negative line bundles $E$, and we proved this for $\mathcal{O}_{\P^m}(-k)$, $1\leq k\leq m/2$. For more general $E$, this is not an immediate consequence of the explicit presentation in Theorem \ref{Theorem Introduction SH is Jac} (presenting a ring does not mean it is easy to determine whether an element vanishes or not). By Theorem \ref{Introduction theorem SH of E is QH modulo ker}, $SH^*(E)\neq 0$ is equivalent to the condition that $\pi^*c_1(E)$ is not nilpotent in $QH^*(E)$. By Theorem \ref{Introduction Theorem Eigenvalues of c1 from B to E}, $\pi^*c_1(E)=-k\pi^*[\omega_B]$ is non-nilpotent in $QH^*(E)$ precisely if $c_1(B)$ is non-nilpotent in $QH^*(B)$. In \cite{RitterSmithVersion1} (Version 1) we conjectured that for closed Fano toric varieties, $c_1(TB)\in QH^*(B)$ is always non-nilpotent, but despite the explicit Batyrev presentation of $QH^*(B)$ this was not known. In Section \ref{Subsection Galkin's result} we explain how a recent result of Galkin \cite{Galkin} implies this Conjecture. Thus:

\begin{corollary}\label{Corollary SH(E) nonzero for mon toric nlb}
Let $E$ be any monotone toric negative line bundle. Then $c_1(TE)\in QH^*(E)$ is non-nilpotent and $SH^*(E)\neq 0$. 
\end{corollary}

This result has significant geometric implications: we will see below that it implies $E$ always contains a non-displaceable monotone Lagrangian torus (Section \ref{Subsection Galkin's result}).

\subsection{Generation results for the Fukaya category and the wrapped category}
We now come back to the discussion of diagram \eqref{Equation Intro Comm Diagram Ritter Smith}. Recall that the direct limit construction which defines the morphism spaces of the wrapped category $\mathcal{W}(M)$ is typically responsible for making the Hochschild homology $\mathrm{HH}_*(\mathcal{W}(M))$ infinite dimensional, just like $SH^*(M)$ is typically expected to be infinite dimensional. This is in contrast to the surprising outcome in Theorem \ref{Theorem admissible toric manifolds} that
$SH^*(M)$ is finite-dimensional for admissible toric manifolds. The goal of the following section is to prove the open-string analogue of this (at least under a genericity assumption on the superpotential), namely cohomological-finiteness for the wrapped category $\mathcal{W}(M)$. We will show that $\mathcal{W}(M)$ is split-generated by the unique compact monotone Lagrangian toric fibre $L$ of the moment map (taken together with finitely many choices of holonomy data, determined by the superpotential), and that $L$ is non-displaceable.\\[1mm]
{\bf Remark.} \emph{Before proceeding, it may be good to clarify what was legitimately to be expected and what was not. Applications to the existence of non-displaceable Lagrangian toric fibres were to be expected, given the substantial work in the case of closed toric manifolds by Fukaya-Oh-Ohta-Ono \cite{FOOOtoric}. In fact for negative line bundles $E\to B$, although messy, one can even explicitly compute the critical points of the superpotential of $E$ in terms of those of $B$ (carried out in Ritter-Smith \cite{RitterSmith}), and then general machinery due to Cho-Oh \cite{Cho-Oh} implies the non-displaceability of the Lagrangian fibres corresponding to the critical points. The novelty of this paper is another: we prove that these toric fibres actually split-generate the whole wrapped category (and not just for negative line bundles). This is surprising in the non-compact setting because the wrapped category admits non-compact Lagrangians and in general it is expected to be either cohomologically infinite or trivial, just like $SH^*(M)$. Secondly, our approach bypasses the messy computations involving the superpotential thanks to our closed-string mirror symmetry result $SH^*(M)\cong \mathrm{Jac}(W_M)$ (Theorem \ref{Theorem admissible toric manifolds}).}
\\[1mm]
\indent We now discuss the generation theorems. Diagram \eqref{Equation Intro Comm Diagram Ritter Smith}, in the monotone setting, is not entirely precise. The Lagrangians are typically \emph{obstructed}: for example, in the toric case, there are always holomorphic Maslov 2 discs bounding toric Lagrangians, and these cause $\partial^2\neq 0$ at the Floer chain level, thus Floer theory becomes problematic \cite{FOOO}. In fact, with current technology, one must restrict the categories $\mathcal{F}_{\lambda}(M),\mathcal{W}_{\lambda}(M)$: they are labeled by the eigenvalues $\lambda$ of the action of $c_1(TM)$ on $QH^*(M)$ by quantum multiplication, and one only allows objects $L$ with a fixed $m_0$-value $\lambda=m_0(L)$. Following \cite{FOOO} (see also \cite{Auroux,RitterSmith}), $m_0(L)$ is the count of holomorphic Maslov 2 discs bounding $L$ passing through a generic point of $L$:
$$\mathfrak{m}_0(L) = m_0(L)\, [L] = \lambda\, [L] \quad \textrm{ for }L\in \mathrm{Ob}(\mathcal{F}_{\lambda}(M))\textrm{ or }\mathrm{Ob}(\mathcal{W}_{\lambda}(M)).$$ 
{\bf Remark.} \emph{We emphasize that we will not undertake the onerous task of generalizing these constructions to a curved $A_{\infty}$-category $\mathcal{F}(M),\mathcal{W}(M)$ which mixes different $m_0$-values, although this task will perhaps be within reach after forthcoming work of \cite{AFOOO}.}\\[1mm]
\indent The generation criterion of Abouzaid \cite{Abouzaid}, adapted to the monotone setup by Ritter-Smith \cite{RitterSmith}, loosely speaking states that Lagrangians $L_1,\ldots,L_m$ split-generate the category if $\mathcal{OC}:\mathrm{HH}_*(\mathcal{W}_{\lambda}(M))\to SH^*(M)$ hits the unit when restricted to the subcategory generated by those $L_j$ (the analogous statement holds for $\mathcal{OC}:\mathrm{HH}_*(\mathcal{F}_{\lambda}(M))\to QH^*(M)$). 

In Ritter-Smith \cite{RitterSmith}, we showed that $\mathcal{OC}$ is an $SH^*(M)$-module map (a property independently observed in the exact case by Ganatra \cite{Ganatra}). Analogously, $\mathcal{OC}:\mathrm{HH}_*(\mathcal{F}_{\lambda}(M))\to QH^*(M)$ is a $QH^*(M)$-module map. In \cite{RitterSmith} we also proved that $\mathcal{OC}$ preserves eigensummands, so the acceleration diagram becomes:
\begin{equation}
\xymatrix{ 
\mathrm{HH}_*(\mathcal{F}_{\lambda}(M)) \ar@{->}[r]^-{} \ar@{->}[d]_-{\mathcal{OC}} & \mathrm{HH}_*(\mathcal{W}_{\lambda}(M)) 
 \ar@{->}^{\mathcal{OC}}[d] \\
QH^*(M)_{\lambda}
\ar@{->}[r]^-{c^*}  & SH^*(M)_{\lambda}}
\end{equation}
where $QH^*(M)_{\lambda}$, $SH^*(M)_{\lambda}$ are the generalized $\lambda$-eigensummands for multiplication by, respectively, $c_1(TM)$ and $c^*(c_1(TM))$. We proved \cite{RitterSmith} that the generation criterion for the $\lambda$-piece of the category holds provided that $\mathcal{OC}$ hits an invertible element in the $\lambda$-eigensummand. 

In the toric setup, since toric Lagrangians are compact and since $c^*(1)=1$ in the acceleration diagram, $\mathcal{W}_{\lambda}(M)$ is split-generated by toric Lagrangians if one can show that 
$$\mathcal{OC}:\mathrm{HH}_*(\mathcal{F}_{\lambda}(M)) \to QH^*(M)_{\lambda}$$
hits an invertible, so we reduce to studying holomorphic discs bounding Lagrangians. We will say that the \emph{toric generation criterion holds for }$\lambda$ if the composite 
$$\mathcal{CO}\circ \mathcal{OC}: \mathrm{HH}_*(\mathcal{F}_{\lambda}^{\mathrm{toric}}(M)) \to QH^*(M) \to HF^*(K,K)$$
hits the unit $[K]\in HF^*(K,K)$ for any Lagrangian $K\in \mathrm{Ob}(\mathcal{F}_{\lambda}(M))$. Here $\mathcal{F}_{\lambda}^{\mathrm{toric}}(M)\subset \mathcal{F}_{\lambda}(M)$ is the subcategory generated by the toric Lagrangians (with holonomy data), and $\mathcal{CO}$ is the closed-open string map. This is in fact enough to prove split-generation \cite{Abouzaid,RitterSmith}.

The practical reasons (see Appendix A) for wanting to work in the Fano toric setup are:
\begin{itemize}
\item There is a machinery due to Cho-Oh \cite{Cho-Oh} for producing Lagrangians: there is a non-displaceable toric Lagrangian $L_p$ with holonomy data for each critical point $p\in \mathrm{Crit}(W)$ of the superpotential $W$, and it satisfies $\lambda=m_0(L)=W(p)$.

\item $p\in \mathrm{Crit}(W)$ if and only if the point class $[\mathrm{pt}]\in HF^*(L_p,L_p)$ is a cycle \cite{Cho-Oh,Auroux}. In fact it then follows by a spectral sequence argument that there is an isomorphism of vector spaces \cite[Theorem 10.1]{Cho-Oh}
\begin{equation}\label{EqnChoOh}
HF^*(L_p,L_p) \cong H^*(L_p;\Lambda).
\end{equation}
%
%
This can also be described explicitly and one can even recover the ring structure \cite{Cho-Products}, a Clifford algebra determined by $W$.

\item The non-constant holomorphic discs bounding toric Lagrangians have Maslov index at least $2$, and those of index $2$ are combinatorially determined by the moment polytope  \cite{Cho-Oh} (in fact, the superpotential $W$ can be interpreted as a count of those discs \cite{Auroux}).

\item By Ostrover-Tyomkin \cite{Ostrover-Tyomkin}, the Jacobian ring $\mathrm{Jac}(W)$ of a closed Fano toric variety has a field summand for each non-degenerate critical point of $W$ (Section \ref{Subsection Generation for 1-dimensional eigensummands}).

\end{itemize}

Our original goal in \cite{RitterSmith} was to prove that toric Lagrangians split-generate $\mathcal{W}(E)$ but we only succeeded to do this when $E$ is $\mathcal{O}(-k)\to \P^m$ for $1\leq k\leq m/2$, as this was the only regime where $QH^*(E)$, $SH^*(E)$ were explicitly known \cite{Ritter4}. The steps in \cite{RitterSmith} were:
\begin{itemize}

\item The ring $SH^*(\mathcal{O}_{\P^m}(-k))$ in \eqref{Equation Introduction SH of O-k over CPm} decomposes into $1$-dimensional eigensummands.

\item A combinatorial exercise with the superpotential $W$ of $\mathcal{O}_{\P^m}(-k)$ shows that all the eigenvalues $\lambda=W(p)$ required by that decomposition do in fact arise.

\item One can check that $L_p$ is the monotone Lagrangian torus in the sphere bundle of  $\mathcal{O}_{\P^m}(-k)$ lying over the Clifford torus in $\P^m$.

\item Observe that the simplest component of $\mathcal{OC}$, namely
\begin{equation}\label{Equation Introduction OC on HFLL}
\mathcal{OC}^{0|0}: HF^*(L,L) \to QH^*(M)
\end{equation}
is non-zero on the point class $[\mathrm{pt}]\in HF^*(L_p,L_p)$, since there is an obvious Maslov 2 disc bounding $L_p$ and living in the fibre determined by the given generic point in $L_p$, and this disc determines the non-zero leading term of $\mathcal{OC}^{0|0}([\mathrm{pt}])$.

\item Since $\mathcal{OC}$ respects eigensummands, $\mathcal{OC}$ hits a non-zero, and thus invertible, element in the $1$-dimensional eigensummand $SH^*(\mathcal{O}_{\P^m}(-k))_{\lambda}$, thus split-generation holds.

\end{itemize}

{\bf Remark.} \emph{Given the importance of \eqref{Equation Introduction OC on HFLL}, we mention that this map was first studied by Albers \cite{Albers}; that this map is a $QH^*(M)$-module map for closed monotone manifolds $M$ first appeared in Biran-Cornea \cite{Biran-Cornea}; more generally string maps first appeared in Seidel \cite{SeidelICM}.}\\[2mm]
\indent In \cite{RitterSmith}, we could not run the same argument above for more general monotone toric negative line bundles $E\to B$, due to the following issues:

\begin{enumerate}

\item\label{Item 1 Intro} We needed an explicit decomposition for $SH^*(E)$, Theorem \ref{Introduction theorem SH of E is QH modulo ker} was not enough.

\item\label{Item 2 Intro} It is computationally unfeasible to check that the critical points of the superpotential $W_E$ recover \emph{all} eigenvalues arising in the decomposition for $SH^*(E)$.

\item\label{Item 3 Intro} Ideally, one wants a condition on the superpotential $W_B$ of base $B$ to ensure generation for $E$, rather than a condition on $E$.

\item\label{Item 4 Intro} Even under the conjectural assumption that $SH^*(E)\cong \mathrm{Jac}(W_E)$, and assuming $W_E$ is Morse (so by Ostrover-Tyomkin \cite{Ostrover-Tyomkin}, $SH^*(E)$ is then a direct sum of fields), it may still happen that an eigensummand $SH^*(E)_{\lambda}$ consists of several field summands. So the above trick ``non-zero element implies invertible'' would fail. This argument would only work if we also assume that all eigenvalues are distinct i.e. simple (Section \ref{Subsection Generation for 1-dimensional eigensummands}).

\end{enumerate}

By Theorem \ref{Theorem Introduction SH is Jac}, we solve issues \eqref{Item 1 Intro} and \eqref{Item 2 Intro}, since $SH^*(E)\cong \mathrm{Jac}(W_E)$, $c^*(c_1(TE))\mapsto W$ ensures that the eigensummand decompositions agree. Issue \eqref{Item 3 Intro} is solved by Theorem \ref{Introduction Theorem Eigenvalues of c1 from B to E} (in particular, if the superpotential $W_B$ is Morse then $W_E$ is Morse). Issue \eqref{Item 4 Intro} requires a general perturbation argument (Appendix B), which we also use in the general argument for admissible toric manifolds, and we explain this later. We first clarify why, in the absence of the condition that $W$ is Morse, the generation problem is out of reach by current technology.

\subsection{The motivation for requiring the superpotential to be Morse}
\label{Subsection The motivation for requiring the superpotential to be Morse} 

In Sections \ref{Subsection What it means for the toric generation criterion to hold} and \ref{Subsection Failure of the generation criterion for some degenerate critical values}, we inspect the generation condition more closely. We show that $\mathcal{CO}:QH^*(M)\to HF^*(K,K)$, with $\lambda=m_0(K)$, will vanish on the eigensummands for eigenvalues different from $\lambda$. Moreover, $\mathcal{CO}$ vanishes on multiples of $c_1(TM)-\lambda$, thus it will vanish on multiples of $x-\lambda$ in a generalized eigensummand $\Lambda[x]/(x-\lambda)^d$ of $QH^*(M)_{\lambda}$. This shows that when the eigensummands of $QH^*(M)_{\lambda}$ are not $1$-dimensional (so $d\geq 2$), then the toric generation criterion will fail if the image of $\mathcal{OC}$ consists of $\lambda$-eigenvectors. Finally, in Theorem \ref{Theorem OC on HF(L,L) lands in eigenspace} by a deformation argument which we explain later, we show that the map in \eqref{Equation Introduction OC on HFLL} always consists of $\lambda$-eigenvectors, thus:

\begin{corollary}\label{Corollary need more than just OC00}
For closed Fano toric manifolds $M$ and for (non-compact) admissible toric manifolds $M$, the map $\mathcal{OC}^{0|0}: HF^*(L,L) \to QH^*(M)$ lands in the $m_0(L)$-eigenspace (not just the generalized eigenspace). So if the superpotential is not Morse, then proving the toric generation criterion requires knowing higher-order components of $\mathcal{OC}$ rather than just $\mathcal{OC}^{0|0}$.
\end{corollary}

It is for this reason that our paper will assume that the superpotential $W$ is Morse (or, more mildly, \emph{$\lambda$-semisimplicity}, meaning the non-degeneracy assumption at critical points $p$ of $W$ with eigenvalue $\lambda=W(p)$). Although this is a generic condition, it does imply the strong consequence that $QH^*(M)$ is semisimple, that is a direct sum of fields \cite{Ostrover-Tyomkin}. For some time, it was conjectured that $QH^*(M)$ was always semisimple for closed toric Fano manifolds, but this was shown to be false by Ostrover-Tyomkin \cite{Ostrover-Tyomkin}. In the non-compact setting, one can see from 
\eqref{Equation Introduction SH of O-k over CPm} that $QH^*(\mathcal{O}_{\P^m}(-k))$ for $k\geq 2$ is not semisimple due to the generalized $0$-eigensummand. In fact more generally, for monotone toric negative line bundles $E$, the zero-eigensummand $QH^*(E)_0$ is still rather mysterious. The toric generation criterion cannot hold for $\mathcal{F}_0(E)$ since there are no toric Lagrangians for $\lambda=0$. This is because by \eqref{Equation Introduction SH=Jac}, the superpotential does not have critical value zero ($W$ corresponds to a non-zero multiple of $c^*(\pi^*c_1(E))\in SH^*(E)$, which is invertible by Theorem \ref{Introduction theorem SH of E is QH modulo ker}, and $SH^*(E)\neq 0$ by Corollary \ref{Corollary SH(E) nonzero for mon toric nlb}).

\begin{corollary}
The toric generation criterion fails for $\mathcal{F}_0(E)$, i.e. $\lambda=0$, for any monotone toric negative line bundle.
\end{corollary}

The perturbation theory in Appendix B (which we explain below) in fact suggests that to generate a generalized eigensummand of type $\Lambda[x]/(x-\lambda)^d$ of $QH^*(M)_{\lambda}$ will require a limit of $d$ eigenvectors of a deformation of $W$ (which makes $W_{\mathrm{deformed}}$ Morse, and $d$ Lagrangians with different $m_0$ values are involved in the limit). Thus knowledge of $\mathcal{OC}$ up to degree $d$ would be required. This is out of reach by current technology, because we cannot make sense of a Fukaya category which mixes different $m_0$-values (the $\partial^2\neq 0$ problem). The semisimplicity assumption on $\lambda$ is not problematic at the linear algebra level: Section \ref{Theorem Gevecs converge}, Theorem \ref{Subsection non-semisimple eigenvalue}, shows that the deformation theory still predicts (conjecturally!) surjectivity of the full map $\mathcal{OC}:\mathrm{HH}_*(\mathcal{F}^{\mathrm{toric}}(M))\to QH^*(M)\to SH^*(M)$, provided one could make sense of a Fukaya category in which Lagrangians with different $m_0$-values can co-exist. 

An explicit example for closed Fano toric manifolds, where $\mathcal{OC}^{0|0}$ is not enough to prove the toric generation crierion, is the smooth Fano $4$-fold called $U_8$, number 116, in Batyrev's classification \cite{BatyrevClassification}. 
This is the example of Ostrover-Tyomkin \cite{Ostrover-Tyomkin} having a non-semisimple quantum cohomology. The eigensummand for eigenvalue $W(p)=-6$ has the form $\Lambda[x]/(x+6)^d$ for $d\geq 2$, but the only torus $L_p$ available for split-generation for that summand has $\mathcal{OC}^{0|0}$ landing in the eigenspace by Corollary \ref{Corollary need more than just OC00}. Finally, as remarked above, $\mathcal{CO}$ vanishes on that eigenspace as these vectors are multiples of $x+6$.

\subsection{Twisting the Fukaya category, generation via a deformation theory argument}
\label{Subsection Twisting the Fukaya category, generation via a deformation theory argument}

We recall (Section \ref{Subsection Toric symplectic forms}) that for any Fano toric variety $X$, near the preferred \emph{monotone} toric symplectic form $\omega_X$, there is a family of non-monotone toric symplectic forms $\omega_F$ depending on the parameters ($\lambda_i=F(e_i)$) which define the moment polytope $\Delta_X=\{ y\in \R^n: \langle y,e_i \rangle \geq \lambda_i \}$ (having fixed the edges $e_i$ of the fan of $X$). We show in Sections \ref{Subsection Using a non-monotone toric form is the same as twisting} and \ref{Subsection Fukaya category for non-monotone toric} that using $\omega_F$ in Floer theory is the same as using $\omega_X$ together with a local system of coefficients determined by $\omega_F$ (a key observation is that $\omega_X,\omega_F$ are both K\"{a}hler for the integrable complex structure $J$, so we can still use $\omega_X$ to control energies in the Floer theory for $\omega_F$).

It is well-established that for generic $\omega_F$ (i.e. for generic perturbation parameters $\lambda_i=F(e_i)$), the twisted superpotential $W_F$ is Morse (\cite{Ostrover-Tyomkin}, \cite[Cor.5.12]{Iritani}, \cite[Prop.8.8]{FOOOtoric}). We need to strengthen this to ensure also simplicity (i.e. distinct eigenvalues), in Section \ref{Subsection Generation for generic toric forms}:

\begin{lemma}\label{Lemma separate critical values}
 For a generic choice of toric symplectic form $\omega_F$ (meaning, after a generic small perturbation of the values $F(e_i)\in \R$), $W_F$ is Morse and the critical values are all distinct and non-zero.
\end{lemma}

\begin{example*} The twisting can have non-trivial consequences in Floer cohomology.
An interesting example is $E=\mathcal{O}(-1,-1)$ over $B=\P^1\times \P^1$ (Section \ref{Subsection Compuation of QH and SH of O(-1,-1)}). The $0$-eigenvectors in $QH^*(B)_0$ yield $0$-eigenvectors in $QH^*(E)_0$ which get quotiented out in $SH^*(E)$. But, for generic twistings, $B$ will not have a $0$-eigensummand, so the twisted $SH^*(E)$ will have larger rank than in the untwisted case (contrary to expectations from classical Novikov homology, where twisting can only make the rank drop). See Remark \ref{Remark rank SH can jump}.
\end{example*}

In Sections \ref{Subsection Presentation of QHB, QHE, SHE}-\ref{Subsection F-twisted superpotential and Jac WF} we prove the twisted analogue of Theorem \ref{Theorem Introduction SH is Jac}. Running the generation argument outlined above, implies (see Section \ref{Subsection Generation for 1-dimensional eigensummands}):

\begin{theorem}[Generic Generation Theorem]\label{Theorem generic generation}
Let $E\to B$ be a monotone toric negative line bundle. Perturb $\omega_B$ to a nearby generic toric symplectic form, this induces a (non-monotone) toric symplectic form $\omega_F$ on $E$ having a simple Morse twisted-superpotential. Then the $\omega_F$-twisted wrapped Fukaya category $\mathcal{W}_{\lambda}(E)_{\omega_F}$ is split-generated by the unique monotone Lagrangian torus taken with suitable holonomies. The same holds for $\mathcal{F}_{\lambda}(E)_{\omega_F}$ provided $\lambda\neq 0$.
\end{theorem}

The aim now, is to use this generic generation result to obtain, in the limit as we undo the deformation, a generation result in the undeformed case. We only sketch the argument here (see Section \ref{Subsection Generation for semisimple critical values} for details). For a small deformation of the parameters $F$, the critical points of $W$ move continuously. Each critical point $p$ of the undeformed $W$ gives rise to the unique monotone Lagrangian torus fibre $L=L_p$ of the moment map together with a choice of holonomy determined by $p$. When we deform to $W_F$, the new $L_p'$ is a Lagrangian torus fibre of the moment map close to the original monotone $L$, and the holonomy is also deformed. The key observation, is that the holomorphic torus action of the toric variety identifies all the moduli spaces involved in the Floer theory of $L_p'$ with those involved for $L_p$. So analytically, we are counting the same PDE solutions, except we count them with different holonomy weights. Since $p\in \mathrm{Crit}(W)$, the vector subspaces $\Lambda\cdot [\mathrm{pt}]\subset HF^*(L_p',L_p')_{\omega_F}$ can be identified even as we vary $F$, so in this sense the vector space map $\mathcal{OC}:\Lambda\cdot [\mathrm{pt}]\subset HF^*(L_p',L_p')_{\omega_F} \to QH^*(E)\cong H^*(E)\otimes \Lambda$ can be viewed as a linear map between fixed vector spaces, which depends holomorphically on some parameter data $F$. Appendix B develops the necessary matrix perturbation theory to tackle this problem. The upshot is:

\begin{theorem}\label{Theorem Intro perturbation argument}
Let $W_E$ be Morse, or more generally assume that $\lambda$ is a \emph{semisimple} eigenvalue of $W$ (non-degeneracy of critical points of $W_E$ with critical value $\lambda$). Then the sum of the images of the deformed $\mathcal{OC}^{0|0}$-maps, summing over eigenvalues $\lambda'=W_F(p')$ which converge to $\lambda$, will converge to the image of the undeformed $\mathcal{OC}^{0|0}$-map.
By Theorem \ref{Theorem generic generation}, the sum of the deformed images is $\oplus QH^*(E)_{\lambda'}$ and (by Appendix B) the vector space $\oplus QH^*(E)_{\lambda'}$ converges to $QH^*(E)_{\lambda}$ in the Grassmannian. Thus the undeformed $\mathcal{OC}^{0|0}$ surjects onto $QH^*(E)_{\lambda}$, so it hits an invertible element, so the toric generation criterion holds.
\end{theorem}

\begin{remark}\label{Remark Intro}
One cannot easily apply the above to $\mathcal{OC}: \mathrm{HH}_*(\mathcal{F}_{\lambda}(E))\to QH^*(E)$. The issue in this deformation theory, is that chain-level expressions which are cycles in the undeformed case are typically not cycles in the deformed case, due to the different holonomy weights. In other words, we have no analogue of \eqref{EqnChoOh} for $\mathrm{HH}_*(\textrm{A}_{\infty}\textrm{-algebra for }L_p)$.
\end{remark}

Combining Theorem \ref{Theorem Intro perturbation argument} with Theorem \ref{Introduction Theorem Eigenvalues of c1 from B to E} we deduce (see Theorem \ref{Theorem toric generation for semisimple critical values}):

\begin{corollary}
Let $E\to B$ be a monotone toric negative line bundle. 
If the superpotential $W_B$ of $B$ is Morse then the toric generation criterion holds for $\mathcal{W}(E)$. The split-generator is a non-displaceable monotone Lagrangian torus in the sphere bundle taken with finitely many choices of holonomy data. Thus $\mathcal{W}(E)$ is compactly generated and cohomologically finite.

More generally, if $W_B$ is not Morse, but it has a semisimple eigenvalue $\lambda\neq 0$, then the toric generation criterion holds for $\mathcal{W}_{\lambda}(E)$ and $\mathcal{F}_{\lambda}(E)$.
\end{corollary}

We finally discuss how these generation arguments apply more generally. First for closed Fano toric varieties and then, in the non-compact case, for admissible toric manifolds.

\subsection{Generation theorems for closed Fano toric manifolds}

For closed Fano toric varieties $C$, in equation \eqref{Equation Introduction OC on HFLL} on the point class $[\mathrm{pt}]\in HF^*(L_p,L_p)$ for $p\in \mathrm{Crit}(W)$, the constant disc contributes the leading term in:
\begin{equation}\label{Equation Introduction OC(pt) is non-zero}
\mathcal{OC}^{0|0}([\mathrm{pt}]) = \mathrm{PD}([\mathrm{pt}]) + (\textrm{higher order }t \textrm{ terms})\neq 0 \in QH^*(C)
\end{equation}
(see Lemma \ref{Lemma calculation of OC of point}). This does not help in the non-compact case as $\mathrm{PD}([\mathrm{pt}])=[\mathrm{vol}_M]=0 \in H^{\dim_{\R}M}(M)=0$.
The same arguments outlined in Section \ref{Subsection Twisting the Fukaya category, generation via a deformation theory argument} therefore imply the following.

\begin{theorem} Let $(C,\omega_C)$ be a closed monotone toric manifold.
If we twist quantum cohomology $QH^*(C)_{\eta}$ by a generic $\eta\in H^2(C;\R)$ close to $[\omega_C]$, then $QH^*(C)_{\eta}$ is a direct sum of $1$-dimensional eigensummands and the twisted superpotential is Morse. So the toric generation criterion applies to the twisted Fukaya category $\mathcal{F}_{\lambda}(C)_{\eta}$.
\end{theorem}

\begin{corollary}
For closed monotone toric manifolds $(C,\omega_C)$, if the superpotential is Morse then the toric generation criterion holds. More generally, the criterion holds for $\mathcal{F}(C)_{\lambda}$ for any semisimple critical value $\lambda$ of the superpotential.
\end{corollary}

\noindent {\bf Remark.} Forthcoming work of Abouzaid-Fukaya-Oh-Ohta-Ono \cite{AFOOO} aims to prove mirror symmetry for closed toric manifolds (without the Fano assumption). In particular, they will show that the toric generation criterion always applies for closed toric Fano manifolds (i.e. even if $W$ is not Morse).
We emphasize that our generation results in the closed Fano case are a much more modest project by comparison. Since we assume monotonicity, the technical difficulties in defining the Fukaya category are rather mild, and we have tools, such as eigensummand splittings, which would not be available in the general setup of \cite{AFOOO}.

\subsection{Generation results for (non-compact) admissible toric manifolds}

The arguments outlined above hold more generally for admissible toric manifolds $M$ (Definition \ref{Definition admissible toric manifolds}), since $SH^*(M)\cong \mathrm{Jac}(W_M)$ by Theorem \ref{Theorem admissible toric manifolds}. The only issue is the calculation of $\mathcal{OC}^{0|0}([\mathrm{pt}])$ in \eqref{Equation Introduction OC on HFLL}. The key result from Section \ref{Subsection Calculation of OC on point class} is:

\begin{lemma}
For an admissible toric manifold $M$, $c_1(TM)$ admits a \emph{compact} cycle representative $C=\mathrm{PD}(c_1(TM))$. After a choice of basis of cycles (which affects the other terms in the expansion below), $C$ determines an lf-cycle $C^{\vee}$ which is intersection-dual to $C$. Then
$$\mathcal{OC}^{0|0}([\mathrm{pt}]) = m_0(L)\, \mathrm{PD}(C^{\vee}) + (\textrm{linearly independent terms}) \in QH^*(M).$$
In particular, for $\lambda=m_0(L)\neq 0$, we deduce $\mathcal{OC}^{0|0}([\mathrm{pt}])\neq 0$.
\end{lemma}

At the heart of this, is a fact due to Kontsevich, Seidel and Auroux \cite{Auroux}: for monotone toric manifolds $M$, and a toric Lagrangian $L$ which does not intersect a representative of $c_1(TM)$, the holomorphic Maslov $2$ discs bounding $L$ passing through a generic point of $L$ will hit $c_1(TM)$ once at an interior point. This fact is responsible for the equation $\mathcal{CO}(c_1(TM))=c_1(TM)*[L]=m_0(L)\,[L]$. The count of discs in this equation is in fact the same as the count involved in the coefficient of $\mathrm{PD}(C^{\vee})$ above. It is crucial that $c_1(TM)$ has a compact cycle representative, and not just a locally finite cycle (in the non-compact setting, one struggles to obtain a duality relationship between $\mathcal{OC}$ and $\mathcal{CO}$ precisely because cycles and lf-cycles are generally unrelated). We prove in Lemma \ref{Lemma L does not intersect c1} that for orientable Lagrangian submanifolds $L$, one can always find a representative of $c_1(TM)$ disjoint from $L$ (this fails for non-orientable Lagrangians, such as $\R P^2\subset \CP^2$, but our Fukaya categories only allow orientable Lagrangians, following \cite{RitterSmith}).

The deformation argument, which ensures that the twisted superpotential becomes Morse with simple eigenvalues, combined with the above Lemma implies:

\begin{theorem}
Let $M$ be any (non-compact) admissible toric manifold $M$ (Definition \ref{Definition admissible toric manifolds}), and let $\omega_F$ be a generic toric (non-monotone) symplectic form close to the monotone form $\omega_M$. For any non-zero eigenvalue $\lambda\neq 0$ of $c_1(TM)\in QH^*(M,\omega_M)$, the $\omega_F$-twisted Fukaya categories $\mathcal{W}_{\lambda}(M)_{\omega_F}$ and $\mathcal{F}_{\lambda}(M)_{\omega_F}$ are split-generated by the unique monotone Lagrangian torus taken with suitable holonomies determined by the superpotential $W_F$.
\end{theorem}

The matrix perturbation argument from Section \ref{Subsection Generation for semisimple critical values} then implies:

\begin{corollary}
Let $M$ be any (non-compact) admissible toric manifold. 
If $W_M$ is Morse and $\lambda\neq 0$, then the toric generation criterion holds for $\mathcal{W}_{\lambda}(M)$ and $\mathcal{F}_{\lambda}(M)$. In particular, $\mathcal{W}_{\lambda}(M)$ is compactly generated and cohomologically finite.
This also holds for non-Morse $W_M$ provided $\lambda$ is a semisimple eigenvalue.
\end{corollary}

\subsection{Mirror symmetry}

Finally we comment on \eqref{Equation Introduction SH=Jac} in the light of mirror symmetry. Ganatra \cite{Ganatra} showed that for exact symplectic manifolds $X$ conical at infinity, if $\mathcal{OC}:\mathrm{HH}_*(\mathcal{W}(X))\to SH^*(X)$ hits a unit, then the closed-open string map $\mathcal{CO}:SH^*(X)\to \mathrm{HH}^*(\mathcal{W}(X))$ is an isomorphism of rings. Significant work would be involved in showing that this also holds for monotone $X$,  but let us suppose it does for the sake of argument. Then, for admissible toric manifolds $M$ with Morse superpotential $W_M$, it would follow that 
$$SH^*(M)\cong \mathrm{HH}^*(\mathcal{W}(M)) \cong \mathrm{HH}^*(\mathcal{F}^{\mathrm{toric}}(M)),$$
 where the latter is the subcategory of the (compact) Fukaya category generated by the toric Lagrangians. The assumption that $W_M$ is Morse, allows one to identify (non-canonically) the semisimple categories $\mathcal{F}^{\mathrm{toric}}(M)$ and $\mathrm{Mat}(W_M)$, the category of matrix factorizations of $W_M$ (in fact, more generally, \cite{AFOOO} will construct a canonical $A_{\infty}$-functor $\mathcal{F}(C) \to \mathrm{Mat}(W_C)$ for closed toric varieties $C$). In particular, 
$$\mathrm{HH}^*(\mathcal{F}^{\mathrm{toric}}(M))\cong \mathrm{HH}^*(\mathrm{Mat}(W_M)).$$
For Morse $W_M$, it follows by Dycherhoff \cite{Dycherhoff} that $\mathrm{HH}^*(\mathrm{Mat}(W_M))\cong \mathrm{Jac}(W_M)$. Thus, glossing over these speculative identifications, mirror symmetry confirms that $SH^*(M)\cong \mathrm{Jac}(W_M)$.

\vspace{3mm}
\noindent {\bf Acknowledgements.} \emph{I thank Mohammed Abouzaid for spotting the subtle issue about cycles explained in Remark \ref{Remark Intro}, which fixed a mistake in the original version of this paper.}
%
\section{Construction of invertible elements in the symplectic cohomology}
\label{Section Construction of invertible elements in the symplectic cohomology}
%
\subsection{The Novikov ring}
\label{Subsection Novikov ring}
We will always work over the field
$$
\Lambda = \left\{ \sum_{i=0}^{\infty} a_i t^{n_i}: a_i\in \K, n_i\in \R, \lim n_i = \infty \right\},
$$
called the \emph{Novikov ring}, where $\K$ is some chosen ground field, and $t$ is a formal variable.
%
We will mostly be concerned with \emph{monotone} symplectic manifolds $(M,\omega)$, meaning
$$
c_1(TM)[u]=\lambda_M \omega[u],
$$
for spheres $u\in \pi_2(M)$. For toric manifolds, $\pi_1(M)=1$, so the condition is equivalent to requiring $c_1(TM)=\lambda_M [\omega]$. For monotone manifolds, the Novikov ring is graded by placing $t$ in (real) degree
$
|t|=2\lambda_M.
$
We often also use a Novikov variable $T$ in degree $|T|=2$, so
$$
T = t^{1/\lambda_M}.
$$
In all Floer constructions, solutions are counted with Novikov weights. For example for the Floer differential, the Floer solutions $u$ are counted with weight
$$t^{\omega[u]} = T^{c_1(TM)[u]}$$
lying in degree $2\lambda_M \omega[u] = 2c_1(TM)[u]$.
%
\subsection{Review of the construction of invertibles in the symplectic cohomology}
\label{Subsection Invertibles in the symplectic cohomology}
Let $M$ be a symplectic manifold conical at infinity (see Section \ref{Subsection Symplectic manifolds conical at infinity}), satisfying \emph{weak+ monotonicity}: that is at least one of the following conditions holds:
\begin{enumerate}
 \item $\omega(\pi_2(M))=0$ or $c_1(TM)(\pi_2(M))=0$;
 \item $M$ is monontone: $\exists \lambda>0$ with $\omega(A)=\lambda c_1(TM)(A)$, $\forall A\in \pi_2(M)$;
 \item the minimal Chern number $|N|\geq \mathrm{dim}_{\C} M - 1$.
\end{enumerate}
That one of these conditions holds is equivalent to the condition:
$$
\textrm{ If }\omega(A)> 0 \textrm{ then } 2-\dim_{\C} M \leq c_1(TM)(A) < 0 \textrm{ is false.}
$$
One can further weaken this assumption by only requiring this for \emph{effective} classes $A$, that is those which are represented by a $J$-holomorphic sphere (see \cite{Ritter4}).
Weak+ monotonicity ensures Floer theory is well-behaved by ``classical'' arguments (Hofer-Salamon \cite{Hofer-Salamon}).

In what follows, $SH^*(M)$ will denote the symplectic cohomology of $M$ (see Section \ref{Subsection Remark about the definition of symplectic cohomology}), restricting to only contractible loops (which is everything when $\pi_1(M)=1$, e.g. for toric $M$).

When working in the monotone setting, we work over the Novikov ring from Section \ref{Subsection Novikov ring}. Otherwise, we work over the Novikov ring $\mathfrak{R}$ from Section \ref{Subsection New Novikov ring R}, or over the Novikov ring defined in \cite{Ritter4} (see the Technical Remark in Section \ref{Subsection Using a non-monotone toric form is the same as twisting}).

\begin{theorem}[Ritter \cite{Ritter4}]\label{Theorem Representation of Ham on SH} We can construct a homomorphism
$$
\mathcal{S}:\widetilde{\pi_1} \mathrm{Ham}_{\ell} (M,\omega) \to SH^*(M)^{\times},\, \widetilde{g}\mapsto \mathcal{S}_{\widetilde{g}}(1)
$$
into invertible elements of the symplectic cohomology, where $\mathrm{Ham}_{\ell}$ refers to Hamiltonian diffeomorphisms generated by Hamiltonians $K=K_t$ which are linear in $R$ for large $R$: $$\qquad\qquad\qquad\qquad\qquad\qquad K_t = \kappa_t R \qquad\textrm{ (negative slopes }\kappa_t<0\textrm{ are also allowed).}$$
\end{theorem}

The symbol $\widetilde{\pi_1}$ refers to an extension of the usual $\pi_1$-group by the group $\Gamma = \pi_2(M)/\pi_2(M)_0$, where $\pi_2(M)_0$ is the subgroup of spheres on which both $\omega$ and $c_1(TM)$ vanish. As this will be important in the application in Section \ref{Subsection The Hamiltonians generating the rotations around the toric divisors}, we now explain this (see \cite{SeidelGAFA} for details). A loop $g:S^1 \to \mathrm{Ham}_{\ell}(M,\omega)$ acts on the space $\mathcal{L}_0M$ of contractible free loops by $(g\cdot x)(t)=g_t\cdot x(t)$.\\ [1mm]
{\bf Technical Remark.} \emph{(For $M$ simply connected, e.g. toric $M$, this issue doesn't arise.)\\ The proof that $g\cdot x$ is contractible in Lemma 2.2 of \cite{SeidelGAFA} used the Arnol'd conjecture. In our case, if this failed for a (time-dependent!) Hamiltonian $K$ linear at infinity, then $SH^*(M)=0$ as the unit $c^*(1)$ would vanish ($c^*: QH^*(M)\to HF^*(K)=0\to SH^*(M)$ factorizes). One can check this is consistent with the construction of the $r,\mathcal{R}$-element for $g$ in diagram \ref{Equation Intro Comm Diagram Ritter Smith}, which turns out to be zero (this is because the construction would factorize through a map $CF^*(g^*H_0)\to CF^*(H_0)$ where one restricts to loops in the class of non-contractible loops $g^{-1}\cdot \mathcal{L}_0(M)$, but here $H_0$ is a $C^2$-small Hamiltonian so its $1$-orbits are contractible). In most applications this does not arise: the $S^1$-action $g$ will typically have fixed points, therefore $g$ must preserve the connected component $\mathcal{L}_0M\subset \mathcal{L}M$ of contractible loops.}\\[1mm]
Let $\widetilde{\mathcal{L}}_0 M$ denote the (connected) cover of $\mathcal{L}_0M$ given by pairs $(v,x)$ where $v:\mathbb{D}^2 \to M$ is a smooth disc with boundary $x\in \mathcal{L}_0M$, where we identify two pairs if both $\omega$ and $c_1(TM)$ vanish on the sphere obtained by gluing together the discs. Then the action of $g$ on $\mathcal{L}_0 M$ lifts to a continuous map $$\widetilde{g}:\widetilde{\mathcal{L}}_0 M\to \widetilde{\mathcal{L}}_0 M$$ which is uniquely determined by the image of a constant disc $(c_x,x)$ at a constant 
loop $x\in M$. Any two such lifts differ by an element in the deck group $\Gamma$.
The above $\widetilde{\pi_1}\mathrm{Ham}$ group is $\pi_0\widetilde{G}$ of the group $\widetilde{G}$ of these choices of lifts $\widetilde{g}$.

The $\Lambda$-module automorphism $\mathcal{S}_{\widetilde{g}}\in \mathrm{Aut}(SH^*(M))$, which turns out to be pair-of-pants product by $\mathcal{S}_{\widetilde{g}}(1)\in SH^*(M)^{\times}$, arises from the isomorphism at the chain level
$$
\mathcal{S}_{\widetilde{g}}:HF^*(H,J,\omega) \to HF^{*+2I(\widetilde{g})}(g^*H,g^*J,\omega), \quad x\mapsto \widetilde{g}^{-1} \cdot x
$$
given by pulling-back all the Floer data by $\widetilde{g}$, where $I(\widetilde{g})$ is a Maslov index (see \cite{Ritter4}) and
$$
g^*H(m,t) = H(g_t\cdot m,t) - K(g_t\cdot m,t) \quad \textrm{ and }\quad g^*J = dg_t^{-1} \circ J_t \circ dg_t.
$$
The direct limit of the maps $\mathcal{S}_{\widetilde{g}}$ as the slope of $H$ at infinity is increased yields a $\Lambda$-module automorphism $\mathcal{S}_{\widetilde{g}}:SH^*(M)\to SH^*(M)$ with inverse $\mathcal{S}_{\widetilde{g}^{-1}}$.

The analogous construction for closed symplectic manifolds $(C,\omega)$ is a well-known argument due to Seidel \cite{SeidelGAFA}. For closed $C$, the Floer cohomologies are all isomorphic to $QH^*(C)$, and so the above chain isomorphism turns into quantum product by a quantum invertible. The above homomorphism would then be the well-known \emph{Seidel representation} $$\widetilde{\pi_1} \mathrm{Ham}(C,\omega) \to QH^*(C)^{\times}.$$

In the non-compact setup, however, the situation is quite different. The groups $HF^*(H)$ depend dramatically on the slope of $H$ at infinity. Only when the slope of $K$ at infinity is positive, it is possible by \cite{Ritter4} to obtain a typically \emph{non-invertible} element in $QH^*(M)$ as follows.
Write $\mathrm{Ham}_{\ell\geq 0}$ to mean that we impose additionally that the slope $\kappa_t\geq 0$ above. Any $g\co S^1 \to \mathrm{Ham}_{\ell \geq 0}(M,\omega)$ gives rise to endomorphisms
\begin{equation}\label{Eqn varphiHS}
\mathcal{R}_{{\widetilde{g}}}=\varphi_H\circ \mathcal{S}_{{\widetilde{g}}}\co HF^*(H,J) \to HF^{*+2I(\widetilde{g})}(H,J),
\end{equation}
where $\varphi_H$ is the continuation map
$$
\varphi_H\co HF^{*}(g^*H,g^*J,\omega) \to HF^{*}(H,J,\omega).
$$
In the direct limit, as the slope of $H$ grows to infinity, this yields the $\Lambda$-module automorphism $$\mathcal{R}_{{\widetilde{g}}}=\mathcal{S}_{{\widetilde{g}}}:SH^*(M)\to SH^{*+2I(\widetilde{g})}(M).$$

The reason for the positive slope condition on $K$ is that $g^*H$ will have smaller slope than $H$, which ensures that the continuation map $\varphi_H$ exists, in particular the continuation map for $H=H_0$ of very small slope at infinity will exist.

Recall that for such small-slope $H_0$, $QH^*(M)\cong HF^*(H_0)$ and there is a canonical $\Lambda$-algebra homomorphism
$$
c^*\co QH^*(M)\cong HF^*(H_0) \to \varinjlim HF^*(H) = SH^*(M).
$$
%
%
Taking $H=H_0$ in \eqref{Eqn varphiHS} determines a $\Lambda$-module homomorphism 
$$r_{\widetilde{g}}: QH^*(M) \to QH^{*+2I(\widetilde{g})}(M)$$ 
which turns out to be quantum cup product by $r_{\widetilde{g}}(1)$. 
\begin{theorem}[Ritter \cite{Ritter4}]\label{Theorem Representation of Ham on QH} 
The homomorphism
$
r\co \widetilde{\pi_1}(\mathrm{Ham}_{\ell\geq 0}(M,\omega)) \to QH^*(M)$, $\widetilde{g} \mapsto r_{\widetilde{g}}(1)$
fits into the commutative diagram \eqref{Equation Intro Comm Diagram Ham SH and QH}.
%
%
%
%
\end{theorem}
%
%
The key observation is that when the slope $\kappa_t>0$, the map $r_{\widetilde{g}}$ in fact determines $SH^*(M)$.
\begin{theorem}[Ritter \cite{Ritter4}]\label{Theorem Circle action gives SH} 
Given $g: S^1 \to \mathrm{Ham}_{\ell>0}(M,\omega)$, the canonical map $c^*:QH^*(M)\to SH^*(M)$ induces an isomorphism of $\Lambda$-algebras
$$
SH^*(M) \cong QH^*(M)/(\textrm{generalized }0\textrm{-eigenspace of }r_{\widetilde{g}})
$$
$($recall the generalized $0$-eigenspace is $\ker r_{\widetilde{g}}^d$ for any $d\geq \mathrm{rank}\, H^*(M))$.
\end{theorem}

This follows from the previous Theorem as follows. Given a Hamiltonian $H_0$ with very small slope at infinity, one defines 
$$
H_k = (g^{-k})^*H_0,
$$
which has slope proportional to $k$ at infinity: $\mathrm{slope}(H_0) + k\cdot \mathrm{slope}(K)$ (here we use that we work in $\mathrm{Ham}_{\ell>0}$ so that $\mathrm{slope}(K)>0$). Notice that $g^*H_k=H_{k-1}$.

The maps $\mathcal{S}_{\widetilde{g}}^k$ identify $HF^*(H_k)\cong HF^{*+2kI(\widetilde{g})}(H_0)$, more precisely
$$S_{\widetilde{g}}^k=S_{\widetilde{g}^k}:HF^*((g^{-k})^*H_0, (g^{-k})^*J) \cong HF^{*+2kI(\widetilde{g})}(H_0,J)$$
and recall that $HF^*(H,J)$ only depends on the slope of $H$ and it does not depend on the choice of $\omega$-compatible almost complex structure $J$ of contact type at infinity.

By a naturality argument, this implies that the continuation map $HF^*(H_{k-1}) \to HF^*(H_k)$ can be identified with $\mathcal{S}_{\widetilde{g}}^{-k} \circ \varphi_{H_0} \circ \mathcal{S}_{\widetilde{g}}^k$. So, after identifying $HF^{*-2kI(\widetilde{g})}(H_k)\cong HF^*(H_0)\cong QH^*(M)$, all continuation maps $HF^*(H_k)\to HF^*(H_{k+1})$ are identified with the map $r_{\widetilde{g}}:QH^*(M)\to QH^{*+2I(\widetilde{g})}(M)$. The Theorem then follows by linear algebra.
%
\subsection{Invertibles in the symplectic cohomology for a larger class of Hamiltonians}
\label{Subsection Invertibles in the symplectic cohomology for a larger class of Hamiltonians}
In Appendix C, Section \ref{Section The Maximum principle revisited}, we will strengthen the maximum principle:

\begin{theorem}[Extended Maximum Principle]
 Let $H: M \to \R$ have the form
$$
 H(y,R) = f(y) R
$$
for large $R$, where $(y,R)\in \Sigma\times (1,\infty)$ are the collar coordinates, and $f:\Sigma\to \R$ satisfies
\begin{itemize}
 \item $f$ is invariant under the Reeb flow (that is $df(Y)=0$ for the Reeb vector field $Y$);
 \item $f \geq 0$.\;\;  \emph{[this condition can be omitted if $d\beta=0$]}
\end{itemize}

Let $J$ be an $\omega$-compatible almost complex structure of contact type at infinity. Then the maximum principle holds for the $R$-coordinate of any Floer solution.
\end{theorem}

The above also holds for the parametrized maximum principle (Theorem \ref{Theorem Continuation maps max}) so it allows us to compute $SH^*(M)$ for a larger class of Hamiltonians (Corollary \ref{Corollary SH for pairs gH gJ}):

\begin{theorem}
Let $H_j:M\to \R$ be Hamiltonians of the form $H_j=m_j R$ for large $R$, whose slopes $m_j\to \infty$, and let $J$ be of contact type at infinity. Suppose $g_t$ is the flow for a Hamiltonian of the form $K=f(y)R$ for large $R$, with $f:\Sigma \to \R$ invariant under the Reeb flow. Then the pairs $(g^*H_j,g^*J)$ compute symplectic cohomology: $SH^*(M)={\displaystyle \lim_{j\to \infty}} HF^*(g^*H_j,g^*J)$.
\end{theorem}

\begin{theorem}\label{Theorem Representation for lager class of Hams}
\label{Theorem Representation into SH for new Hamiltonians}
Theorems \ref{Theorem Representation of Ham on SH}, \ref{Theorem Representation of Ham on QH} and \ref{Theorem Circle action gives SH} also hold when we enlarge $\mathrm{Ham}_{\ell}(M)$, $\mathrm{Ham}_{\ell\geq 0}(M)$, $\mathrm{Ham}_{\ell>0}(M)$ to allow for generating Hamiltonians $K$ of the form
$$
K_t(y,R) = f_t(y)R,
$$
for large $R$, where $(y,R)\in \Sigma\times [1,\infty)$ are the collar coordinates, and where
$f_t: \Sigma \to \R$ is invariant under the Reeb flow. In the cases $\ell\geq 0$, $\ell>0$ we require $f_t\geq 0$, $f_t>0$ respectively.
\end{theorem}
\begin{proof}
%
The construction of $\mathcal{S}_{\widetilde{g}}:HF^*(H_k,J,\omega) \to HF^{*+2I(\widetilde{g})}(g^*H_k,g^*J,\omega)$ is never an issue, since it is an identification at the chain level. However, to obtain a map $\mathcal{S}_{\widetilde{g}}:SH^*(M)\to SH^{*+2I(\widetilde{g})}(M)$ from this, by a direct limit argument, it is necessary that the pairs $(g^*H_k,g^*J)$ can be used to compute $SH^*(M)$ when we let the slopes $k\to \infty$. The previous two Theorems enable us to do this. The proof of Theorem \ref{Theorem Representation of Ham on SH} then follows by the same arguments as in \cite{Ritter4}.

Similarly, Theorem \ref{Theorem Representation of Ham on QH} follows: for $f\geq 0$ the slope $m-f(y)$ of $g^*H_m=(m-f(y))R$ is smaller than $m$, so the maps $\varphi_H$ still exist (see Section \ref{Subsection Remark about the definition of symplectic cohomology}). 
Finally, the argument described under Theorem \ref{Theorem Circle action gives SH} can still be carried out, since the Hamiltonians $H_k = (g^{-k})^*H_0$ have slope proportional to $k$ at infinity: $\mathrm{slope}(H_0) + k\cdot \mathrm{slope}(K)$, where $\mathrm{slope}(K)=f(y)>0$.
\end{proof}

\begin{proof}[{\bf Proof of Lemma \ref{Lemma rg1}}]
We mimick the argument in \cite[Section 10]{Ritter4}. By \cite{Ritter4}, $r_{g^{\wedge}}(1)\in QH^*(M)$ is a count of $(j,\hat{J})$-holomorphic sections of the bundle $E_g \to \P^1$ whose fibre is $M$ and whose transition map over the equator of $\P^1$ is $g$, where $\hat{J}$ is an almost complex structure on $E_g$ related to $J$. The moduli spaces $\mathcal{S}(j,\hat{J},\gamma+S_{g^{\wedge}})$ of such sections break up according to certain equivalence classes $\gamma + S_{g^{\wedge}}$ indexed by $\gamma \in \pi_2(M)/\pi_2(M)_0$.\\
{\bf Technical remark.} \emph{The class $\mathcal{S}_{g^{\wedge}}$ is represented by the constant section at $x\in \mathrm{Fix}(g)$, as it is obtained by gluing the following two sections on the upper/lower hemispheres \cite{Ritter4}: $s_{g^{\wedge}}^+(z)=c_x(z)=x$, $s_{g^{\wedge}}^-(z)=(g^{\wedge}\cdot c_x)(z)=c_x(z)=x$, using the hypothesis $g^{\wedge}\cdot c_x=c_x$.
}\\
The virtual dimension of the moduli space is
\begin{equation}\label{Equation Dim of moduli space of sections}
 \dim_{\C} \mathcal{S}(j,\hat{J},\gamma+S_{g^{\wedge}}) = \dim_{\C} M - I(g^{\wedge}) + c_1(TM)(\gamma).
\end{equation}
As $M$ is monotone, $c_1(TM)(\gamma)\geq 0$, so the first term in Equation \eqref{Equation rg1} arises as the locally-finite cycle swept by the constant sections (so $\gamma=0$). Indeed, the constant sections must lie in $\mathrm{Fix}(g)$ otherwise the transition map $g$ would make them non-constant. If we can show that the constant sections are regular, then Equation \eqref{Equation rg1} follows.\\

To prove regularity of the constant sections we mimick \cite[Section 10]{Ritter4}. For a constant section $u: \P^1 \to E_g$ the vertical tangent space 
$$(u^*T^{v}E_g)_z = TM_{u(z)} \cong T_{u(z)}D \oplus \C^{m-d}.
$$
Here we use that $dg_t$ is complex linear and $g_t$ is symplectic to deduce that $dg_t$ is unitary, so it preserves the unitary complement $\C^{m-d}$ of the fixed subspace $TD$. 
%
%
%
As we vary $z\in \P^1$, the transition $g$ along the equator acts identically on $TD$ and it acts by $dg_t$ on $\C^{m-d}$. By the assumption on the eigenvalues of $dg_t$, the $\C^{m-d}$ summand gives rise to the bundle $\mathcal{O}(-1)^{\oplus d}\to \P^1$. Essentially by definition, $I(g^{\wedge})=\mathrm{deg}(t\mapsto \det (dg_t: T_x M \to T_x M))$ (using that $g^{\wedge}$ fixes $(c_x,x)$) which is the sum $m-d$ of the degrees of the eigenvalues. 
Thus $u^*T^{v}E_g \cong \mathcal{O}^d \oplus \mathcal{O}(-1)^{\oplus m-d}$. So the obstruction bundle is, using Dolbeaut's theorem and Serre duality,
$$
\coker \overline{\partial} = H^1(\P^1,\mathcal{O}^{d} \oplus \mathcal{O}(-1)^{\oplus m-d}) \cong 
H^0(\P^1,\mathcal{O}(-2)^{\oplus d} \oplus \mathcal{O}(-1)^{\oplus m-d})^{\vee} =0.$$

For $I(g^{\wedge})=m-d=1$, $\dim_{\C} \mathcal{S}(j,\hat{J},\gamma+S_{g^{\wedge}}) \geq \dim_{\C} M-1$, but by the maximum principle the locally finite pseudo-cycle $\mathcal{S}(j,\hat{J},\gamma+S_{g^{\wedge}})$ cannot sweep the unbounded manifold $M$. This forces $c_1(TM)(\gamma)=0$ and so, by monotonicity, that $\gamma=0$. So Equation \eqref{Equation rg1 if Ig is 1} follows.

We remark that in \cite{Ritter4} a larger Novikov ring was used than here, which kept track of the class $\gamma$ in the count of sections. But in the monotone case, it suffices to keep track of $\omega(\gamma)$ and we count sections with weight $t^{\omega(\gamma)}$. For constant sections, this weight is $1$.
\end{proof}

\section{The quantum cohomology of non-compact monotone toric manifolds}
\label{Section The quantum cohomology of non-compact toric varieties}
\subsection{Review of the Batyrev-Givental presentation of $\mathbf{QH^*(M)}$ for closed toric monotone manifolds $\mathbf{M}$}
\label{Subsection Review of the Batyrev-Givental presentation of QH}
In this section, $M$ will be a \emph{closed} monotone toric manifold.
Let $\Delta$ be a moment polytope for $M$, and let $F$ be the corresponding fan (we recall some of this terminology in Appendix A, Section \ref{Section The moment polytope of a toric negative line bundle}). Then there are correspondences:
$$
\{ \textrm{edges } e_i \textrm{ of } F \} 
=
 \left\{\begin{smallmatrix}\textrm{inward normals }e_i\\ \textrm{to facets of } \Delta\end{smallmatrix}\right\}
\leftrightarrow
 \left\{\begin{smallmatrix}\textrm{facets } F_i\\ \textrm{of } \Delta\end{smallmatrix}\right\}
 \leftrightarrow
 \left\{\begin{smallmatrix}\textrm{toric}\\ \textrm{divisors } D_i\end{smallmatrix}\right\}
\leftrightarrow
 \left\{\begin{smallmatrix}\textrm{homogeneous}\\ \textrm{coordinates } x_i\end{smallmatrix}\right\}
$$
in particular $D_i = \{ x_i=0\}$. Recall \emph{facet} means $\mathrm{codim}_{\R}=1$ face.

\emph{Example. For $\CP^2$: $\Delta = \{y\in \R^2: y_1\geq 0, y_2 \geq 0, -y_1-y_2\geq -1\}$, $F$ has edges $e_1=(1,0)$, $e_2=(0,1)$, $e_3=(-1,-1)$, and the cones are the $\R_{\geq 0}$-span of proper subsets of the edges.
}

We work over a field $\K$ of characteristic zero. Then the classical cohomology $H^*(M)$ is
\begin{equation}\label{EqnHtoric}
H^*(M) = \K[x_1,\ldots,x_r]/(\textrm{linear relations}, \textrm{classical Stanley-Reisner relations}).
\end{equation}
The linear relations are 
\begin{equation}\label{EqnLinearRelations}
\sum \langle \xi,e_i\rangle \, x_i = 0,
\end{equation}
 where $\xi$
runs over the standard basis of $\R^n$ and $\langle \cdot,\cdot\rangle$ is the standard inner product on $\R^n$.

\emph{Example: For $\CP^2$: $x_1 - x_3=0$, $x_2-x_3=0$, so this identifies all $x_i$ with a variable $x$.}

\begin{definition}[Primitive subset]\label{Definition Primitive}
 A subset of indices $I=\{ i_1,\ldots,i_a\}$ for the edges is called \emph{primitive} if the subset $\{e_{i_1},\ldots,e_{i_a}\}$ does not define a cone of $F$ but any proper subset does.\\
Equivalently, if $F_{i_1}\cap \ldots \cap F_{i_a} = \emptyset$ but for any proper subset of $I$ the intersection is non-empty.\\
Equivalently, if $D_{i_1}\cap \ldots \cap D_{i_a}=\emptyset$ but for any proper subset of $I$ the intersection is non-empty.
\end{definition}

The classical SR-relations are $x_{i_1}\cdot x_{i_2} \cdot\cdots\cdot x_{i_a}=0$ for primitive $I$.

\emph{Example: For $\CP^2$: $e_1,e_2,e_3$ is primitive, so $x_1x_2x_3=0$. So $H^*(\CP^2)=\K[x]/(x^3)$.}

The quantum cohomology is 
\begin{equation}\label{EqnQHtoric}
QH^*(M) = \K[x_1,\ldots,x_r]/(\textrm{linear relations}, \textrm{quantum Stanley-Reisner relations}).
\end{equation}
Batyrev \cite{Batyrev} showed that to each primitive subset $I=\{i_1,\ldots,i_a\}$ there corresponds a unique disjoint subset $j_1,\ldots,j_b$ of indices, yielding a linear dependence relation
$$
e_{i_1} + \ldots + e_{i_a} = c_{1} e_{j_1} + \cdots + c_{b} e_{j_b}
$$
for (non-zero) positive $c_{1},\ldots,c_{b}\in \Z$.\\
\emph{Remark. This uses the compactness assumption, that the cones cover $\R^n$, so $\sum e_{i_k}$ lies in a cone and so is uniquely a non-negative linear combination of the edges $e_{j_{\ell}}$ defining that cone.\\
}
The quantum version of the SR-relation is then
\begin{equation}\label{EqnQSRrelation}
x_{i_1} x_{i_2}  \cdots  x_{i_a} = t^{\omega(\beta_I)} x_{j_1}^{c_{1}} \cdots x_{j_b}^{c_{b}}
\end{equation}
where the $\beta_I\in H_2(M)$ arising in the exponent of the Novikov variable $t$ is the class obtained from $e_{i_1}+\cdots + e_{i_a} - c_{1} e_{j_1} - \cdots -c_{b} e_{j_b}=0$ under the identification
$$
 (\Z\textrm{-linear relations among the }e_i) \longleftrightarrow H_2(M,\Z),
$$
where explicitly: $\sum n_i e_i = 0$ corresponds to the $\beta\in H_2(M,\Z)$ whose intersection products with the toric divisors are $\beta\cdot D_i = n_i$. In particular, Batyrev showed that $\omega(\beta_I)>0$.

\emph{Example: For $\CP^2$: $e_1+e_2+e_3 = 0$, $\beta_{1,2,3}=[\CP^1]$, so $x_1x_2x_3=t$. So $QH^*(\CP^2)=\K[x]/(x^3-t)$.}

Recall the polytope is defined by the edges $e_i$ of the fan and by real parameters $\lambda_i$,
\begin{equation}\label{Eqn moment polytope}
\Delta = \{ y\in \R^n: \langle y,e_i \rangle \geq \lambda_i \}.
\end{equation}
Since $c_1(TM) = \sum \mathrm{PD}[D_i]$ and $[\omega] = - \sum \lambda_i \mathrm{PD}[D_i]$, via \eqref{EqnQHtoric} we have 
$$
c_1(TM) = \sum x_i \quad \textrm{and} \quad  [\omega] = -\sum \lambda_i x_i,
$$
from which it follows that
\begin{equation}\label{EqnC1andOmega}
 c_1(TM)(\beta_I)= a- c_{1}-\cdots-c_{b} \quad \textrm{and} \quad [\omega](\beta_I) = - \lambda_{i_1}-\cdots -\lambda_{i_a}+c_{1}\lambda_{j_1} + \cdots + c_{b} \lambda_{j_b}.
\end{equation}
\emph{Example: For $\CP^2$: $(\lambda_1,\lambda_2,\lambda_3)=(0,0,-1)$, $c_1(TM)=3x$, $[\omega]=x$ in $H^*(\CP^2)=\K[x]/(x^3)$.}
\subsection{Review of the McDuff-Tolman proof of the presentation of $\mathbf{QH^*(M)}$}
\label{Subsection Review of the McDuff-Tolman proof of the presentation of QH}
For each tuple $(n_1,\ldots,n_r)\in \Z^r$ there is an associated Hamiltonian $S^1$-action given in homogeneous coordinates by $x_i \mapsto e^{2\pi \sqrt{-1} n_i t} x_i$. The generators, corresponding to the standard basis of $\Z^r$, are the natural rotations $g_i$ about the toric divisors $D_i$, yielding correspondences:
$$
\{ \textrm{edges } e_i\} 
 \leftrightarrow
 \left\{\begin{smallmatrix}\textrm{toric}\\ \textrm{divisors } D_i\end{smallmatrix}\right\}
\leftrightarrow
 \left\{\begin{smallmatrix}\textrm{homogeneous}\\ \textrm{coordinates } x_i\end{smallmatrix}\right\}
\leftrightarrow
\{ \textrm{rotations } g_i\}
$$
in particular\footnote{This statement is meant before making identifications via the group action. Globally in $M$, the fixed locus of $g_i$ will typically have other components beyond $D_i$. However, for monotone toric manifolds, only $D_i$ will contribute to the Seidel element \cite[Theorem 1.10(iii)]{McDuffTolman}. For example, for $\C \mathbb{P}^1$ and $g_1([x_0:x_1])=[x_0: e^{2\pi i t}x_1]$, both $[0:1]$ and $D_1=[1:0]$ are fixed points, but only $D_1$ will contribute (see \cite[Example 2.6]{McDuffTolman}).
Our conventions differ from \cite{McDuffTolman}: we choose $e_i$ to be inward normals to $\Delta$, so $D_i$ arises as the minimum of $H_i$, whereas in their conventions it is the maximum \cite[Footnote (9) p.65]{McDuffTolman}.} $D_i = (x_i=0) = \textrm{Fix}(g_i)$ are the fixed points of $g_i$.

The Hamiltonians $H_i$ for the $g_i$ are determined in terms of the moment map $\mu: M \to \R^n$ (tacitly identifying $\R^n=\mathfrak{u}(1)^n$) and the data which defines $\Delta$:
\begin{equation}\label{EqHamToric}
H_i(x) =\langle \mu(x), e_i \rangle - \lambda_i = \tfrac{1}{2}|x_i|^2
\end{equation}
where the constant $\lambda_i$ ensures that $H_i\geq 0$ and $D_i = H_i^{-1}(0)$. We emphasize that the final equality in \eqref{EqHamToric} holds only for $x\in f^{-1}(0)$, using the notation of Section \ref{Subsection The moment map}.

\emph{Example: For $\CP^2$, $[x_1:x_2:x_3]$ are the usual coordinates, $g_i$ is the standard rotation in the $i$-th entry, the moment map is $\mu = \tfrac{1}{\sum |x_j|^2}(|x_1|^2,|x_2|^2)$, and $H_i=\tfrac{|x_i|^2}{\sum |x_j|^2}$. In the notation of Section \ref{Subsection The moment map}, $x\in f^{-1}(0)$ imposes $f(x)=\tfrac{1}{2}\sum |x_j|^2 -1=0$, so then $H_i=\tfrac{1}{2}|x_i|^2$.
}

For a closed monotone toric manifold $(M,\omega)$, the Seidel representation \cite{SeidelGAFA}
$$
\mathcal{S}:\pi_1 \textrm{Ham}(M,\omega) \to QH^*(M)^{\times}
$$
is a homomorphism into the even degree invertible elements of $QH^*(M)$. 

\begin{theorem}[McDuff-Tolman \cite{McDuffTolman}]
 For a closed monotone toric manifold $M$, the rotations $g_i$ yield Seidel elements
$$
\mathcal{S}(g_i)=t^{\lambda_i} x_i.
$$
\end{theorem}

Via the correspondences
$$
\begin{array}{rcccl}
(\textrm{\small primitive } I) &\leftrightarrow &
 \left\{\begin{smallmatrix}\textrm{relations among edges}\\ 
e_{i_1}+\cdots+e_{i_a}=c_{1}e_{j_1}+\cdots+c_{b} e_{j_b}
\end{smallmatrix}\right\}
 &\leftrightarrow &
 \left\{\begin{smallmatrix}\textrm{relations among Hamiltonians}\\ 
H_{i_1}+\cdots+H_{i_a}=c_{1}H_{j_1}+\cdots+c_{b}H_{j_b}+\textrm{const.}
\end{smallmatrix}\right\} \\[2mm]
&\leftrightarrow &
 \left\{\begin{smallmatrix}\textrm{relations among rotations}\\ 
g_{i_1}\cdots g_{i_a}=g_{j_1}^{c_{1}}\cdots g_{j_b}^{c_{b}}
\end{smallmatrix}\right\}
 &\leftrightarrow &
 \left\{\begin{smallmatrix}\textrm{quantum SR-relations}\\ 
x_{i_1} x_{i_2}  \cdots  x_{i_a} = t^{\omega(\beta_I)} x_{j_1}^{c_{1}} \cdots x_{j_b}^{c_{b}}
\end{smallmatrix}\right\}
\end{array}
$$
the Seidel homomorphism $\mathcal{S}$ therefore recovers the quantum SR-relations
\eqref{EqnQSRrelation} using \eqref{EqnC1andOmega}.

The final ingredient of the argument of McDuff-Tolman, is the following useful algebraic trick. The Lemma in fact does not involve toric geometry and is stated in greater generality in \cite{McDuffTolman}. It holds for any closed monotone symplectic manifold $(M,\omega)$ taking $\varphi(x_i)$ to be a choice of algebraic generators for $H^*(M)$, and it also holds in the non-compact monotone setup if one can construct quantum cohomology. The argument only relies on the $t$-adic valuation for the Novikov ring and the fact that quantum corrections occur with strictly positive $t$-power.

\begin{lemma}[McDuff-Tolman \cite{McDuffTolman}] \label{Lemma McDuffTolman Algebra trick}
Consider the algebra homomorphisms
$$\varphi: \K[x_1,\ldots,x_r] \to H^*(M) \quad \textrm{ and } \quad \psi: \K[x_1,\ldots,x_r] \otimes \Lambda \to QH^*(M)$$
determined by $\varphi(x_i)=\mathrm{PD}[D_i]$ and $\psi(x_i)=\mathrm{PD}[D_i]\otimes 1$ (using quantum multiplication to define $\psi$). By construction $\varphi$ is surjective.

Then $\psi$ is surjective. Moreover, suppose $p_1,\ldots,p_r$ generate $\ker \varphi$ and suppose there exist
$$P_j = p_j + q_j \in \ker \psi \subset \K[x_1,\ldots,x_r] \otimes \Lambda
$$
with $t$-valuations $\mathrm{val}_t(q_j)>0$. Then the $P_j$ generate $\ker \psi$, thus
$$
H^*(M) = \K[x_1,\ldots,x_r]/(p_1,\ldots,p_r) \quad \textrm{ and } \quad QH^*(M) = \K[x_1,\ldots,x_r]/(P_1,\ldots,P_r).
$$
\end{lemma}
%
\subsection{Batyrev's argument: from the presentation of $\mathbf{QH^*}$ to $\mathbf{Jac(W)}$}
\label{Subsection The quantum cohomology of non-compact toric varieties}
\label{Subsection Batyrev's argument: from the presentation of QH to JacW}

Let $X$ be a monotone toric manifold of dimension $\dim_{\C}X=n$.
We defined the superpotential $W=\sum t^{-\lambda_i} z^{e_i}$ of $X$ in Definition \ref{Definition Superpotential}, with $e_i,\lambda_i$ as in \eqref{Eqn moment polytope}.

\begin{definition}\label{Definition Jacobian ring}
 The Jacobian ring of $W$ is 
$$
\mathrm{Jac}(W) = \Lambda[z_1^{\pm 1},\ldots,z_n^{\pm 1}]/(\partial_{z_1}W,\ldots,\partial_{z_n}W).
$$
\end{definition}

In the definition above, $\Lambda[z_1^{\pm 1},\ldots,z_n^{\pm 1}]$ should be thought of as the coordinate ring of the complex torus $(\C^*)^n\subset X$ whose compactification is $X$.

Following Batyrev \cite[Theorem 8.4]{Batyrev}, consider the homomorphism
$$
\psi: \Lambda[x_1,\ldots,x_r] \to \mathrm{Jac}(W), \quad x_i \mapsto t^{-\lambda_i} z^{e_i},
$$
which sends the $i$-th homogeneous coordinate to the $i$-th summand of the superpotential $W$. In $\mathrm{Jac}(W)$ the quotient by the derivatives of $W$ corresponds to the linear relations \eqref{EqnLinearRelations}. Indeed, imposing $z_j \tfrac{\partial W}{\partial z_j} =0$ is equivalent to 
$$
\sum_i e_i^j t^{-\lambda_i} z^{e_i} =0
$$
where $e_i^j$ is the $j$-th entry of $e_i$, which corresponds to the linear relation
$
\sum e_i^j x_i = \sum \langle \xi,e_i \rangle x_i = 0
$
when $\xi$ is the $j$-th standard basis vector of $\R^{n}$. 

When $X$ is compact, the union of the cones of the fan for $X$ cover all of $\R^n$ \cite[Definition 2.3(iii)]{Batyrev}, which immediately implies that the above map $\psi$ is surjective. This fails in the non-compact case, for example:
\begin{example}\label{Example failure of QH=Jac} 
In Section \ref{Section The moment polytope of a toric negative line bundle} we show that for a monotone toric negative line bundle $E$ the fan does not cover $-e_f$, where $e_f$ is the edge corresponding to the fibre direction. Multiplication by $x_f=\pi^*c_1(E)=\mathrm{PD}[B] \in QH^*(E)$ is never invertible by \cite{Ritter4}, even though its image $\psi(x_f)= z_f$ is tautologically invertible since $z_f^{-1}$ is a generator of $\mathrm{Jac}(W)$. Thus $QH^*(E)\cong \mathrm{Jac}(W)$ fails, in fact we will prove $SH^*(E)\cong \mathrm{Jac}(W)$.

From a Floer-theoretic perspective, the reason why $x_f^{-1}$ exists
in $SH^*(E)$ but not in $QH^*(E)$, is that the representation $\mathcal{S}_{\widetilde{g}_f^{-1}}(1)=x_f^{-1}$ is defined in $SH^*(E)$ (see Section \ref{Subsection Invertibles in the symplectic cohomology}), since negative slope Hamiltonians are not problematic for $SH^*(E)$. Whereas the representation $r_{\widetilde{g}_f^{-1}}(1)=x_f^{-1}\in QH^*(E)$ cannot be defined due to a failure of the maximum principle.
\end{example}

Now consider in general the kernel of $\psi$. By construction, the kernel is generated by the relations $\prod_p x_p^{a_p} = T^{\sum a_p-\sum c_q}\prod_q x_q^{c_q}$ for each relation amongst edges of the form
$$
\sum a_p e_p = \sum c_q e_q
$$
where $a_p,c_q\geq 0$ are integers (the power of $T$ is explained in Lemma \ref{Lemma relations in SH}). Batyrev proved \cite[Theorem 5.3]{Batyrev} that the ideal generated by these relations has a much smaller set of generators, namely those arising from relations $\sum e_{i_p} = \sum c_q e_{j_q}$ for primitive subsets $I=\{i_1,i_2,\ldots\}$ (Definition \ref{Definition Primitive}). These give rise to the quantum SR-relations.

\begin{remark}We remark that the notation $\mathrm{exp}(\varphi(v_i))$ of \cite{Batyrev} corresponds to our $t^{-\lambda_i}$, so
\begin{equation}\label{EqnPowerOfT}
T^{c_1(TX)[\beta_I]}=T^{|I| - \sum c_q} = t^{-\sum \lambda_i + \sum c_q \lambda_{j_q}} = t^{[\omega_X](\beta_I)}
\end{equation}
since $c_1(TX)=\sum \mathrm{PD}[D_i]=\lambda_X[\omega_X]$ by monotonicity;  $[\omega_X]=\sum -\lambda_i \mathrm{PD}[D_i]$ corresponds to the $\varphi$ of \cite[Section 5]{Batyrev} (see \cite[Definition 5.8]{Batyrev}); and $T=t^{1/\lambda_X}$ (see Section \ref{Subsection Novikov ring}).
\end{remark}

From now on, let $(X^{2n},\omega)$ be a (possibly non-compact) toric manifold with moment polytope $\Delta=\{ y\in \R^n: \langle y, e_i \rangle \geq \lambda_i, i=1,\ldots,r\}$ and superpotential $W=\sum t^{-\lambda_i} z^{e_i}$. 

\begin{definition}\label{Definition RX algebra}
The above data for $X$ determines an algebra over $\Lambda$, 
$$
R_X = \left\{ \sum {\lambda}_e z^e: \textrm{finite sum}, \lambda_e\in \Lambda, e\in \mathrm{span}_{\Z_{\geq 0}}(e_i) \right\} \subset \Lambda[z_1^{\pm 1},\ldots,z_n^{\pm 1}].
$$
When $X$ is compact, the fan of $X$ is complete, so $R_X = \Lambda[z_1^{\pm 1},\ldots,z_n^{\pm 1}]$.
\end{definition}

\begin{corollary}\label{Corollary QH of noncompact toric}
Suppose $QH^*(X,\omega)\cong \Lambda[x_1,\ldots,x_r]/\mathcal{J}$, where $\mathcal{J}$ is the ideal generated by the linear relations and the quantum SR-relations determined by $\Delta$. Then there is an isomorphism
$$
\psi: QH^*(X,\omega) \to 
R_X/(\partial_{z_1} W,\ldots,\partial_{z_n} W), \quad x_i \mapsto t^{-\lambda_i} z^{e_i},
$$
which sends $c_1(TX)=\sum x_i \mapsto W= \sum t^{-\lambda_i} z^{e_i}$. When $X$ is closed, this is the well-known isomorphism $QH^*(X,\omega)\cong \mathrm{Jac}(W) = \Lambda[z_1^{\pm 1},\ldots,z_n^{\pm 1}]/(\partial_{z_1} W,\ldots,\partial_{z_n} W)$.
\end{corollary}
\begin{proof}
Observe that the image of $\psi$ is precisely $R_X$, by construction. Moreover, $\psi$ descends to the following quotient:
$$\overline{\psi}: \Lambda[x_1,\ldots,x_r]/(\textrm{quantum SR-relations}) \to R_X.$$
By Batyrev's argument, $\overline{\psi}$ is injective (the quotient on the left is the same as quotienting by all positive integral relations amongst edges, as mentioned above, and then it is clear that the map is injective). The image under $\overline{\psi}$ of the linear relations is precisely the ideal $(\partial_{z_1}W,\ldots,\partial_{z_r}W)$. The claim follows by quotienting both sides by that ideal.
\end{proof}

\begin{corollary}\label{Corollary QH of noncompact toric 2}
Continuing with the assumption of Corollary \ref{Corollary QH of noncompact toric}, there is a quotient map
$$
QH^*(X,\omega)
 \to \mathrm{Jac}(W), \quad x_i \mapsto t^{-\lambda_i} z^{e_i},
$$
sending $c_1(TX) \mapsto W$, and corresponding to the natural map
$$
c^*:R_X/(\partial_{z_1} W,\ldots,\partial_{z_n} W) \to R_X[w_1,\ldots,w_r]/(\partial_i W,w_i z^{e_i}-1: i=1,\ldots,r) \cong \mathrm{Jac}(W).
$$
which is the canonical localization at the variables $z_1,\ldots,z_r$.
\end{corollary}
\begin{proof}
 The only thing left prove, is the very last isomorphism. This follows by Section \ref{Subsection Some remarks about localizations of rings}: 
the $w_i$ formally invert the $z^{e_i}$, which corresponds to quotienting $R_X/(\partial_{z_1} W,\ldots,\partial_{z_n} W)$ by the generalized $0$-eigenspace of $z^{e_i}$ acting on $R_X/(\partial_{z_1} W,\ldots,\partial_{z_n} W)$, and after quotienting there is a multiplicative inverse for $z^{e_i}$. The $w_i$ correspond to $z^{-e_i}$ in $\mathrm{Jac}(W)$, and the smoothness of $X$ ensures that the $\Z$-span of the $e_i$ is $\Z^n$ (see the comment in Section \ref{Subsection The fan of a line bundle over a toric variety} about smoothness). This implies the surjectivity of the map into $\mathrm{Jac}(W)$. 
\end{proof}

%

\begin{conjecture}\label{Conjecture SH for monotone toric noncompact X} 
For non-compact monotone toric manifolds $X$ for which $SH^*(X)$ can be defined (e.g. $X$ conical at infinity), we expect that
$$
c^*:QH^*(X)\cong R_X/(\partial_{z_1}W,\ldots,\partial_{z_n}W) \to \mathrm{Jac}(W)\cong SH^*(X)
$$
is the canonical map $c^*:QH^*(X)\to SH^*(X)$. In particular, $c^*$ is the canonical localization map, localizing at the variables $x_i=\mathrm{PD}[D_i]$ corresponding to the toric divisors $D_i$.
\end{conjecture}

We now prove the conjecture for admissible toric manifolds (Definition \ref{Definition admissible toric manifolds}), and in particular in Section \ref{Subsection The Jacobian ring for nlb} we prove that this applies to monotone toric negative line bundles. 
%
\subsection{Admissible non-compact toric manifolds}
\label{Subsection Admissible non-compact toric manifolds}

\begin{proof}[{\bf Proof of Theorem \ref{Theorem admissible toric manifolds}}]
 Definition \ref{Definition admissible toric manifolds} part \eqref{Item Admissible conical} and part \eqref{Item Admissible monotone} ensure that $QH^*(M)$ and $SH^*(M)$ are well-defined (see \cite{Ritter4}).  Part \eqref{Item Admissible monotone} means we can work over the Novikov ring $\Lambda$ (in particular, Lemma \ref{Lemma relations in SH} will hold). Part \eqref{Item Admissible max principle} is the analogue of Corollary \ref{Corollary Hams in toric case are OK} for negative line bundles: it ensures by Theorem \ref{Theorem Representation into SH for new Hamiltonians} that the representation from \cite{Ritter4},
$$
r: \widetilde{\pi_1}(\mathrm{Ham}_{\ell\geq 0}(X,\omega_X)) \to QH^*(X,\omega_X),\;\; \widetilde{g} \mapsto r_{\widetilde{g}}(1),
$$
is defined on the $S^1$-actions $g_i$ given by rotation around the toric divisors, lifting to $g_i^{\wedge}$ as in Lemma \ref{Lemma rg1} so that $r_{g_i^{\wedge}}(1)=x_i=\mathrm{PD}[D_i]$ (using Lemma \ref{Lemma rg1}, and $\mathrm{Fix}(g_i)=[D_i]$). The algebraic trick in Lemma \ref{Lemma McDuffTolman Algebra trick} then ensures that the generators $x_i$ of $QH^*(X,\omega_X)$ satisfy precisely the relations in $\mathcal{J}$ -- in particular the quantum SR-relations arise from the relations amongst the rotations $g_i$, the fact that the representation $r$ is a homomorphism, and the fact that these relations only involve positive slope Hamiltonians (see the proof of Lemma
\ref{Lemma Quantum SR-relations involve positive slope Hamiltonians}).

By Theorem \ref{Theorem Representation into SH for new Hamiltonians}, the representation from \cite{Ritter4},
$$
\mathcal{S}:\widetilde{\pi_1}(\mathrm{Ham}_{\ell}(X,\omega_X)) \to SH^*(X,\omega_X)^{\times},\;\; \widetilde{g} \mapsto \mathcal{S}_{\widetilde{g}}(1)
$$
yields invertible elements $\mathcal{S}(g_i^{\wedge})(1)=x_i \in SH^*(E)^{\times}$ (the inverse $x_i^{-1}=\mathcal{S}((g_i^{-1})^{\wedge})(1)$ exists since the Hamiltonian is allowed to have negative slope for the representation $\mathcal{S}$).

The fact that $SH^*$ is the quotient of $QH^*$, and hence is the localization at all $x_i$, follows by Theorem \ref{Theorem When SH is quotient of QH} and Section \ref{Subsection Some remarks about localizations of rings}, using the assumption in Definition \ref{Definition admissible toric manifolds} part \eqref{Item Admissible positive slope}. The fact that $c^*(r_{g_i^{\wedge}}(1))=\mathcal{S}_{g_i^{\wedge}}(1)$ follows by Theorem \ref{Theorem Representation of Ham on QH}.
The rest follows by Section \ref{Subsection Batyrev's argument: from the presentation of QH to JacW}.
\end{proof}

\subsection{Blow-up at a point}
\label{Subsection Blow-up at a point}
We briefly recall some generalities about blow-ups, following Guillemin \cite[Chapter 1]{Guillemin}. For a complex manifold $X$ of dimension $\dim_{\C}X=n$, the blow-up $\pi:\widetilde{X}\to X$ at a point $p\in X$ is a holomorphic map, which is a biholomorphism outside of the exceptional divisor
$
E= \pi^{-1}(p),
$
and $E$ can naturally be identified with $\CP^{n-1}$ viewed as the projectivization of the normal bundle of $p\in X$.
Suppose, in addition, that $(X,\omega_X)$ is symplectic. Then $\widetilde{X}$ carries a symplectic form $\omega_{\widetilde{X}}$ such that $\omega_{\widetilde{X}}-\pi^*\omega_{X}$ is compactly supported near $E$ and restricts to $\varepsilon \omega_{FS}$ on $E\cong \CP^{n-1}$. Here $\omega_{FS}$ is the normalized Fubini-Study form, and $\varepsilon>0$ is called the blow-up parameter (namely, the symplectic area of degree one spheres in $E$).
When there is a Hamiltonian $G$-action on $X$, for a compact group $G$ (such as a torus in the case of toric $X$), and $p$ is a fixed point of $G$, then $\omega_{\widetilde{X}}$ can be chosen to be $G$-equivariant and the action will lift to a $G$-Hamiltonian action on $\widetilde{X}$.

By Griffiths-Harris \cite[Chp.4 Sec.7]{Griffiths-Harris} the cohomology of the blow-up is 
\begin{equation}\label{Equation cohomology of blow-up}
H^*(\widetilde{X}) = \pi^*H^*(X) \oplus \overline{H}^*(E)
\end{equation}
where the pull-back $\pi^*$ is injective, and $\overline{H}^*(E)=\overline{H}^*(\CP^{n-1})$ is the reduced cohomology (i.e. quotiented by $H^*(\mathrm{point})$). The generator $-\omega_{FS}\in H^*(E)$ corresponds to $\mathrm{PD}[E]$ since the normal bundle to $E\subset \widetilde{X}$ has Chern class $-\omega_{FS}$. Moreover,
$$
c_1(T\widetilde{X}) = \pi^*c_1(TX) - (n-1)\mathrm{PD}[E].
$$
Thus, if $X$ is monotone, so $c_1(TX)=\lambda_X [\omega_X]$ with $\lambda_X>0$, and we want $(\widetilde{X},\omega_{\widetilde{X}})$ to be monotone, then we require:
$$
\varepsilon = (n-1)/\lambda_X.
$$
This ensures that $c_1(T\widetilde{X}) = \lambda_X [\omega_{\widetilde{X}}]$. The same condition is required if by monotonicity we just require $c_1(TX)=\lambda_X [\omega_X]$ to hold when integrating over spheres, by observing that $c_1(T\widetilde{X})$, $\omega_{\widetilde{X}}$ integrate to $n-1,\varepsilon\lambda_X$ respectively on degree one spheres in $E$.

The above condition on $\varepsilon$ cannot always be achieved, as there may not be a sufficiently large Darboux neighbourhood around $p$ to apply the local model of a blow-up of $\C^n$ at $0$ with area parameter $\varepsilon$ (see \cite[Thm 1.10]{Guillemin}).

More generally, if $X$ satisfies weak+ monotonicity (Section \ref{Subsection Invertibles in the symplectic cohomology}) then so does $\widetilde{X}$ since the holomorphic spheres $A\subset E$ satisfy $c_1(T\widetilde{X})(A)=(n-1)\omega_{FS}(A)\geq 0$.

Let $(X,\omega_X)$ be a non-compact K\"ahler manifold (a complex manifold with a compatible symplectic form $\omega_X$) with a Hamiltonian $S^1$-action. Suppose also that $X$ is conical at infinity, satisfies weak+ monotonicity, and the Hamiltonian generating the $S^1$-action on $X$ has the form prescribed by the extended maximum principle in Theorem \ref{Theorem maximum principle} with $f>0$. 

\begin{theorem}\label{Theorem Kahler manifold with S1 action}
 Under the above assumptions, the blow-up $(\widetilde{X},\omega_{\widetilde{X}})$ of $X$ at a fixed point of the action is a K\"ahler manifold conical at infinity, satisfying weak+ monotonicity, whose lifted Hamiltonian $S^1$-action satisfies Theorem \ref{Theorem maximum principle} with $f>0$. In particular, $SH^*(\widetilde{X})$ is determined as a quotient of $QH^*(\widetilde{X})$ by Theorem \ref{Theorem When SH is quotient of QH}.
\end{theorem}
\begin{proof}
 Away from the exceptional divisor, $X,\widetilde{X}$ are biholomorphic so the Hamiltonians generating the $S^1$-actions agree. The claim follows.
\end{proof}
%
In applications, via Lemma \ref{Lemma rg1}, $r_{g^{\wedge}}(1)\in QH^*(\widetilde{X})$ will typically be $\mathrm{PD}[E]$.

For a toric variety $X$, with a choice of toric symplectic form $\omega_X$, a fixed point $p\in X$ of the torus action corresponds to a vertex $p\in \Delta$ of the moment polytope \eqref{Eqn moment polytope}. The moment polytope $\widetilde{\Delta}$ of the one-point blow-up is then obtained by chopping off a standard simplex from $\Delta$ at $p$ corresponding to $\varepsilon \Delta_{\CP^n}$ \cite[Thm 1.12]{Guillemin}. Thus $p$ is replaced by the opposite facet in this simplex, which is a copy of $\varepsilon \Delta_{\CP^{n-1}}$. Explicitly,
$$
\widetilde{\Delta} = \Delta \cap \{ y\in \R^n: \langle y,e_0  \rangle \geq \lambda_0 \}
$$
where $e_0 = e_{i_1} + \cdots + e_{i_n}$ is the sum of those edges of the fan for $X$ which are the inward normals to the facets of $\Delta$ which meet at $p$, and $\lambda_0= \varepsilon+ \lambda_{i_1} + \cdots + \lambda_{i_n}$.

When $(X,\omega_X)$ is monotone, we can normalize so that $[\omega_X]=c_1(TX)$ (so the monotonicity constant $\lambda_X=1$), and by Section \ref{Subsection The polytope of a Fano variety} one can pick all $\lambda_i = -1$. So if we require $(\widetilde{X},\omega_{\widetilde{X}})$ to be monotone, we need $\lambda_0 = -1$ which agrees with the above monotonicity constraint for $\widetilde{X}$:
$$
\varepsilon = n-1.
$$
Since the Euclidean length of the edges in $\Delta$ through the vertex $p$ corresponds to the symplectic areas of the spheres corresponding to those edges \cite[Chp. 2]{Guillemin}, that condition on $\varepsilon$ can be satisfied if the edges have length strictly greater than $n-1$.
For example, for $X=\CP^2$ with $\omega_X=c_1(TX)=3\omega_{FS}$ the edges have length $3$ and the condition is $\varepsilon=1$. So $\CP^2$ can be blown-up in a monotone way at a vertex (indeed even at all three vertices) of $\Delta_{\CP^2}$.

\begin{theorem}\label{Theorem Blow-up at a point is admissible}
Let $(X,\omega_X)$ be an admissible toric manifold (Definition \ref{Definition admissible toric manifolds}). Suppose $X$ admits a monotone blow-up $(\widetilde{X},\omega_{\widetilde{X}})$ at a toric fixed point (or several such points). Then $(\widetilde{X},\omega_{\widetilde{X}})$ is admissible, in particular Theorem \ref{Theorem admissible toric manifolds} holds.
\end{theorem}
\begin{proof}
Part \eqref{Item Admissible conical} of Definition \ref{Definition admissible toric manifolds} follows since $X$ is conical and since $\omega_{\widetilde{X}}$ agrees with $\omega_X$ away from $E$ (where $X$ and $\widetilde{X}$ can be identified).
Part \eqref{Item Admissible max principle} of Definition \ref{Definition admissible toric manifolds} follows as in the proof of Theorem \ref{Theorem Kahler manifold with S1 action} for the edges $e_i$, and for the new edge $e_0 = e_{i_1} + \cdots + e_{i_n}$ the rotation $g_0=g_{i_1} g_{i_2} \cdots g_{i_n}$ is generated by the sum of the Hamiltonians for the $g_i$, so again it has non-negative slope.
\end{proof}

The passage from $QH^*(X,\omega_X)$ to $QH^*(\widetilde{X},\omega_{\widetilde{X}})$ at the level of vector spaces is already known by Equation \eqref{Equation cohomology of blow-up}, although in terms of the linear relations some care is needed: $\sum \langle \xi, e_i\rangle x_i =0$ now contains the additional term $\langle \xi, e_0 \rangle x_0 \equiv \langle \xi, e_{i_1} + \cdots + e_{i_n} \rangle x_0$. The new edge $e_0$ gives rise to the new generator 
$$r_{g_0^{\wedge}}(1)=\mathrm{PD}[E]$$
 of $\overline{H}^*(E)$ in \eqref{Equation cohomology of blow-up}. The only new quantum SR-relation required is
$$
x_{i_1} x_{i_2} \cdots x_{i_n} = T^{n-1} x_0,
$$
corresponding to the relation $e_{i_1} + \cdots + e_{i_n} = e_0$ among edges. Indeed, suppose there was another ``new'' relation among edges involving $e_0$. Substituting $e_0=e_{i_1} + \cdots + e_{i_n}$ shows that this relation was detected and thus generated by the original relations among edges $e_i$, $i\neq 0$.

\textbf{Example.} We conclude with a basic example: $X=\C^{n+1}$ blown up at the origin. Then $\widetilde{X}$ is the total space of $\mathcal{O}(-1)\to \CP^n$, where $E=\CP^n$ is the base. The edges for $X$ are $e_1=(1,0,\ldots,0)$, $e_2=(0,1,0,\ldots,0)$, $\ldots$, $e_{n+1}=(0,\ldots,0,1)$. 
The new edge for $\widetilde{X}$ is $e_0=(1,1,\ldots,1)$.
The linear relations $x_i=0$ for $X$ become $x_i = -x_0$.
The new (and only) quantum SR-relation is $x_1\cdots x_{n+1} = T^n x_0$.
Letting $x=-x_0$, we deduce
$$
QH^*(\widetilde{X}) \cong \Lambda[x]/(x^{n+1} + T^n x),
$$
which agrees with the computation of $QH^*(\mathcal{O}_{\CP^n}(-1))$, indeed $x$ represents $-\mathrm{PD}[E]$ which is the pull-back of $\omega_{\CP^n}$ via $\mathcal{O}(-1)\to \CP^n$. 

\subsection{Blow-up along a closed complex submanifold}
\label{Subsection Blow-up at a subvariety}
Similarly, one can blow-up a complex manifold $X$ along a closed complex submanifold $S\subset X$. The exceptional divisor of the blow-up $\pi:\widetilde{X}\to X$ is $E=\pi^{-1}(S)$, which can be identified with the projectivization of the normal bundle $\nu_{S\subset X}$. A neighbourhood of $E\subset\widetilde{X}$ is modeled on the tautological bundle over $\mathbb{P}(\nu_{S\subset X})$. The cohomology is \cite[Chp.4 Sec.7]{Griffiths-Harris}
\begin{equation}\label{Equation cohomology of blow-up 2}
H^*(\widetilde{X}) = \pi^*H^*(X) \oplus \overline{H}^*(E)
\end{equation}
where $\pi^*$ is injective, and now $\overline{H}^*(E)$ denotes $H^*(E)$ quotiented by the pull-back of $H^*(S)$ via $E\equiv \mathbb{P}(\nu_{S\subset X}) \to S$. One can explicitly describe $\overline{H}^*(E)$ \cite[Chp.4 Sec.7]{Griffiths-Harris}, in particular it is generated as an $H^*(X)$-algebra by $\mathrm{PD}[E]$. Moreover,
$$
c_1(T\widetilde{X}) = \pi^*c_1(TX) - (\mathrm{codim}_{\C} X - 1)\, \mathrm{PD}[E] 
$$

If, in addition, $(X,\omega_X)$ is symplectic, then one can construct $\omega_{\widetilde{X}}$ with analogous properties as in Section \ref{Subsection Blow-up at a point}. In particular $\omega_{\widetilde{X}}-\pi^*\omega_X$ is compactly supported near $E$, and restricts to $\varepsilon\omega_E$ on $E$ where $\omega_E$ is a symplectic form such that $[\omega_E] = -\mathrm{PD}[E]$. If $X$ is monotone, then monotonicity of $\widetilde{X}$ holds if 
$$\varepsilon = (\mathrm{codim}_{\C}X -1)/\lambda_X.$$
Just as for one-point blow-ups, it is not always possible to pick such an $\varepsilon$, but if we only require weak+ monotonicity then that will hold for $\widetilde{X}$ if it holds for $X$. Theorem \ref{Theorem Kahler manifold with S1 action} becomes:

\begin{theorem}\label{Theorem Kahler manifold with S1 action 2}
 Under the assumptions above Theorem \ref{Theorem Kahler manifold with S1 action}, the blow-up $(\widetilde{X},\omega_{\widetilde{X}})$ of $X$ along a closed complex submanifold $S\subset X$ fixed by the action is a K\"ahler manifold conical at infinity, satisfying weak+ monotonicity, whose lifted Hamiltonian $S^1$-action satisfies Theorem \ref{Theorem maximum principle} with $f>0$. Also, $SH^*(\widetilde{X})$ is determined as a quotient of $QH^*(\widetilde{X})$ by Theorem \ref{Theorem When SH is quotient of QH}.
\end{theorem}
%
In applications, via Lemma \ref{Lemma rg1}, $r_{g^{\wedge}}(1)\in QH^*(\widetilde{X})$ will typically be $\mathrm{PD}[E]$.

For $X$ toric, McDuff-Tolman \cite{McDuffTolmanMassLin} describe the blow-up of $X$ along the complex submanifold corresponding to a face $F_I$ of the moment polytope $\Delta$ of $X$, of complex codimension at least 2. The face $F_I=F_{i_1}\cap\cdots \cap F_{i_a}$ is an intersection of the facets $F_j=\{y\in \Delta: \langle y,e_j \rangle=\lambda_j\}$ for a collection $I=\{i_1,\ldots,i_a\}$, and the codimension is $|I|$. The new polytope is
$$
\widetilde{\Delta} = \Delta \cap \{ y\in \R^n: \langle y,e_0  \rangle \geq \lambda_0 \}
$$
where $e_0 = \sum_{i\in I} e_i$ and $\lambda_0= \varepsilon+ \sum_{i\in I}\lambda_i$. This holds for all small parameters $\varepsilon>0$ as long as the new facet $F_0=\{ y\in \Delta: \langle y, e_0\rangle = \lambda_0 \}$ stays bounded away from the vertices of $\Delta \setminus F_I$. 
For monotone $X$, taking $\lambda_i=-1$, $\widetilde{X}$ is monotone provided that $\lambda_0=-1$, which agrees with the above condition:
$
\varepsilon = |I| - 1.
$

\begin{theorem}
Let $(X,\omega_X)$ be an admissible toric manifold (Definition \ref{Definition admissible toric manifolds}). Suppose $X$ admits a monotone blow-up $(\widetilde{X},\omega_{\widetilde{X}})$ along a face $F_I$ of complex codimension $|I|\geq 2$ (or several such). Then $(\widetilde{X},\omega_{\widetilde{X}})$ is admissible, in particular Theorem \ref{Theorem admissible toric manifolds} holds.
\end{theorem}
\begin{proof}
The same proof as in Theorem \ref{Theorem Blow-up at a point is admissible} applies.
In particular, for the new edge $e_0 = \sum_{i\in I} e_i$ the rotation $g_0=\prod_{i\in I} g_i$ is generated by the sum of the Hamiltonians for the $g_i$, so again it has non-negative slope.
\end{proof}

The passage from $QH^*(X,\omega_X)$ to $QH^*(\widetilde{X},\omega_{\widetilde{X}})$ at the level of vector spaces is already known by Equation \eqref{Equation cohomology of blow-up 2}. Just as at the end of Section \ref{Subsection Blow-up at a point}, the linear relations $\sum \langle \xi, e_i\rangle x_i =0$ contain the new term $\langle \xi, e_0 \rangle x_0 \equiv \langle \xi,\sum_{i\in I} e_i \rangle x_0$. The new edge $e_0$ gives rise to the new generator 
$r_{g_0^{\wedge}}(1)=\mathrm{PD}[E]$
of $\overline{H}^*(E)$ in \eqref{Equation cohomology of blow-up 2}. The only new quantum SR-relation required is
$$
\prod_{i\in I} x_i = T^{|I|-1} x_0.
$$
%
%
\section{Presentation of $QH^*$ and $SH^*$ for toric negative line bundles}
\label{Section Presentation of QH and SH for toric nlb}
\subsection{Definition}
\label{Subsection Negative line bundles}
A complex line bundle $\pi: E \to B$ over a closed symplectic manifold $(B,\omega_B)$ is called \emph{negative} if $c_1(E) = -k[\omega_B]$ for $k>0$. As shown for example in \cite[Section 7]{Ritter4}, there is a Hermitian metric on $E$ which determines a norm $r:E \to \R_{\geq 0}$ and which determines a symplectic form
$$
\omega = \pi^*\omega_B + \boldsymbol{\pi}\Omega,
$$
such that $[\omega]=\pi^*[\omega_B]\in H^2(E)$, where we temporarily use the bold symbol $\boldsymbol{\pi}$ for the mathematical constant to avoid confusion with the pullback map $\pi^*$. There is a Hermitian connection whose curvature $\mathcal{F}$ satisfies $\tfrac{1}{2\boldsymbol{\pi} i} \pi^*\mathcal{F} = k\pi^*\omega_B$. Outside of the zero section of $E$, there is an angular $1$-form $\theta$ on $E$ (which vanishes on horizontal vectors and satisfies $\theta_w(w)=0$ and $\theta_w(iw)=1/2\boldsymbol{\pi}$ in the fibre directions), such that 
$$d\theta = k\pi^*\omega_B \quad \textrm{and} \quad d(r^2\theta) = \Omega.$$
So outside of the zero section, $\omega$ is exact: $\omega = d(\tfrac{1}{k}\theta + \boldsymbol{\pi} r^2 \theta)$. Moreover $\omega$ has the block form
$$
\left(\begin{matrix}
       (1+ k\boldsymbol{\pi} r^2)\pi^*\omega_B & 0 \\
       0 & \omega_{\textrm{standard}}
      \end{matrix}
\right)
$$
in the decomposition $TE \cong T^{\textrm{horiz}}E \oplus E$; it is the standard form in the vertical $\C$-fibres.

\emph{Technical Remark. In \cite[Section 7]{Ritter4} we considered $\omega=d\theta + \varepsilon \Omega = k \pi^*\omega_B + \varepsilon\Omega$, so $[\omega]=k\pi^*[\omega_B]$. It turns out that for toric $E$ a more natural choice is to rescale that form by $1/k$ and to choose $\varepsilon=k\boldsymbol{\pi}$, as we did above. The Floer theory is not affected by this rescaling.}

In this decomposition, we define an $\omega$-compatible almost complex structure $J = \left(\begin{smallmatrix}
       J_B & 0 \\
       0 & i
      \end{smallmatrix}
\right)
$
in terms of an $\omega_B$-compatible almost complex structure $J_B$ for $(B,\omega_B)$, in particular fiberwise it is just multiplication by $i=\sqrt{-1}$.

It is shown for example in \cite{Ritter4}, that the radial coordinate is 
$$R=\tfrac{1+k\boldsymbol{\pi} r^2}{1+k\boldsymbol{\pi}}.$$
%
%
%
%
The Liouville vector field is 
$Z=\tfrac{1+k\boldsymbol{\pi} r^2}{k\boldsymbol{\pi} r^2} \tfrac{w}{2}$,
 and the Reeb vector field is 
  $Y=X_R=\tfrac{2k\boldsymbol{\pi}}{1+k\boldsymbol{\pi}} iw$ 
 where $w\in \C$ is the local fibre coordinate in a unitary frame (the extra $k$ in $Y$ compared to \cite{Ritter4} is due to the rescaling mentioned in the above Technical Remark). Since $iw$ is the angular vector field whose flow wraps around the fibre circle in time $2\boldsymbol{\pi}$, the $S^1$-action which rotates the fibre circle in time $1$ is generated by the vector field $\tfrac{1+k\boldsymbol{\pi}}{k}Y=X_{(1+k\boldsymbol{\pi})R/k}$.
 
The contact type condition has the form $JZ = c(R) Y$, as in Remark \ref{Remark tweaking contact condition Max}.
%
%
%
%
%
%

For negative line bundles $\pi:E\to B$, assuming weak+ monotonicity (see \ref{Subsection Invertibles in the symplectic cohomology}), Theorem \ref{Theorem Circle action gives SH}  applies to the circle action $g_t=e^{2\boldsymbol{\pi} i t}$ acting by rotation on the fibres of $E$. In this case $r_{\widetilde{g}}$ is quantum multiplication by $\pi^*c_1(E)=-k\pi^*[\omega_B]$.

\begin{theorem}[Ritter \cite{Ritter4}]\label{Theorem SH of negative line bundle} 
The canonical map $QH^*(E)\to SH^*(E)$ induces an isomorphism of $\Lambda$-algebras
$
SH^*(E) = QH^*(E)/(\textrm{generalized }0\textrm{-eigenspace of }\pi^*c_1(E)).
$

For $E=\mathrm{Tot}(\mathcal{O}(-k)\to \P^m)$ and $1\leq k \leq m/2$, this becomes explicitly via $x=\pi^*[\omega_B]$:
$$
\begin{array}{rcl}
QH^*(\mathcal{O}_{\P^m}(-k)) &=& \Lambda[x]/(x^{1+m}-(-k)^k\,T^{1+m-k}x^k)\\
SH^*(\mathcal{O}_{\P^m}(-k)) &=& \Lambda[x]/(x^{1+m-k}-(-k)^k\,T^{1+m-k}). 
\end{array}
$$
\end{theorem}

\subsection{Some remarks about localizations of rings}
\label{Subsection Some remarks about localizations of rings}
Recall that for a ring $R$ and an element $f\in R$, the \emph{localization $R_f$ of $R$ at $f$} consists of equivalence classes $(r,f^n)$ for $n\in \N$ (informally thought of as fractions $\tfrac{r}{f^n}$) under the equivalence relation $(r,f^n)\simeq (r',f^m)$ whenever $f^k(rf^m-r'f^n)=0$ for some $k\in \N$. In particular, if $f$ is nilpotent then $R_f=0$.

\begin{lemma}
Localizing at $f$ is the same as formally introducing an inverse $z$ of $f$:
$$R_f \cong R[z]/(1-fz).\qquad\qed$$
\end{lemma}

There is a \emph{canonical localization map} $c: R\to R_f$, $r\mapsto (r,1)$. Given an ideal $I\subset R$, the \emph{saturation} $(I:f^{\infty}) \subset R$ is the preimage under $c$ of the localized ideal $R_f I\subset R_f$:
$$c^{-1}(I)=(I:f^{\infty})=\{ r\in R: f^k r\in I \textrm{ for some }k\in\N \}.$$
\begin{lemma}
 In general, $(I:f^{\infty}) = I' \cap R$ where $I'\subset R[z]$ is the ideal generated by $I$ and $1-fz$. Recall $(I:f^{\infty})\mapsto R_f I$ via the surjection $c: R\to R_f$. Thus the canonical localization map $c$ determines the natural isomorphism of rings
$$
R/(I:f^{\infty}) \to R_f/R_fI. \qedhere
$$ 
In particular, for $I=0$, $R/(0:f^{\infty})\cong R_f$. $\qquad\qquad\qquad\qed$
\end{lemma}

\begin{corollary}\label{Corollary SH is a localization of QH}
 For negative line bundles $E\to B$ (satisfying weak+ monotonicity), $SH^*(E)\cong QH^*(E)_{\pi^*c_1(E)}$ is the localization of $QH^*(E)$ at $f=\pi^*c_1(E)$ and the canonical map $c^*:QH^*(E) \to SH^*(E)$ corresponds to the canonical localization map.

 In particular, given a presentation $QH^*(E)\cong \K[x_1,\ldots,x_r]/\mathcal{J}$ for an ideal of relations $\mathcal{J}$, the corresponding presentation for $SH^*(E)$ is:
$$
SH^*(E)\cong QH^*(E)_f \cong \K[x_1,\ldots,x_r,z]/\langle \mathcal{J}, 1-fz \rangle.
$$
\end{corollary}
\begin{proof}
 In Theorem \ref{Theorem SH of negative line bundle}, the generalized $0$-eigenspace of $f=\pi^*c_1(E)$ is precisely the saturation $(0:f^{\infty})$ of the trivial ideal $0\subset QH^*(E)$. Thus, by the previous Lemma, $SH^*(E)\cong QH^*(E)/(0:f^{\infty})\cong QH^*(E)_f$.
\end{proof}

\subsection{The Hamiltonians generating the rotations around the toric divisors}
\label{Subsection The Hamiltonians generating the rotations around the toric divisors}
Let $\pi:E \to B$ be a monotone toric negative line bundle, with $c_1(E)=-k[\omega_B]$. We describe these in detail in Section \ref{Section The moment polytope of a toric negative line bundle}. Since toric manifolds are simply connected, monotonicity implies:
$$
c_1(TE)=\lambda_E[\omega_E] = \lambda_E\pi^*[\omega_B], \qquad c_1(TB)=\lambda_B[\omega_B], \qquad \lambda_E=\lambda_B -k>0.
$$
(We used that $c_1(TE)=\pi^*c_1(TB)+\pi^*c_1(E)$ by splitting $TE$). The toric divisors are:
$$
D_i=\pi^{-1}(D_i^B) \textrm{ for }i=1,\ldots,r,\quad \textrm{ and } D_{r+1}=[B],
$$
where $D_i^B$ are the toric divisors in $B$ and $[B]\subset E$ is the zero section.
The Hamiltonians $H_i$ which generate the standard rotations $g_i$ about $D_i$ (see \ref{Subsection Review of the McDuff-Tolman proof of the presentation of QH}) actually depend on the radial coordinate $R$, despite what \eqref{EqHamToric} might suggest. This is inevitable since the Hamiltonians $H_i$ satisfy relations $H_{i_1}+\cdots + H_{i_a}=c_{1} H_{j_1}+\cdots + c_{b}H_{j_b}+\textrm{constant}$, and $H_{r+1}$ obviously depends on $R$. In fact, if $H_i$ were $R$-independent, then both $g_i,g_i^{-1}$ are in $\pi_1\mathrm{Ham}_{\ell\geq 0}(M)$, so $r_{\widetilde{g}_i}(1)$ would be invertible with inverse $r_{\widetilde{g}_i^{-1}}(1)$. But this is not the case for example for $E=\mathcal{O}(-1)\to \P^1$ where $r_{\widetilde{g}_1}(1)=x=\pi^*[\omega_{\P^1}]$ satisfies $x^2+Tx=0$ in $QH^*(E)$.
%
%
\begin{theorem}\label{Theorem Hamiltonians for rotations about divisors}
The Hamiltonians $H_i$ for the rotation $g_i$ about $D_i=\pi^{-1}(D_i^B)$ are
 $$H_i(x) = (1+k\pi) R \cdot \pi^*H_i^B(x)$$
(where by convention $H_i=0$ on $D_i$, and $H_i^B=0$ on $D_i^B$),
%
%
and the rotation $g_{r+1}$ about the zero section $D_{r+1}=[B]$ has Hamiltonian
$$
H_{r+1}=\tfrac{1+k\pi}{k}R-\tfrac{1}{k}
$$
\end{theorem}
\begin{proof}
We first clarify what the coordinate $R$ is. By Section \ref{Subsection The moment map}, we know a formula for the moment map $\mu_E(x)$ when $x$ satisfies a certain quadratic equation $f_E(x)=0$, in particular we know the last entry: $\mu_E(x) = (\ldots,\tfrac{1}{2}|x_{r+1}|^2)$. By Lemma \ref{Lemma forms agree for line bundle}, the norm in the fibre is then $|w|=\tfrac{1}{\sqrt{2\pi}}|x_{r+1}|$, so by
Section \ref{Subsection Negative line bundles} the radial coordinate $R$ is 
$$
R=\tfrac{1+k \pi |w|^2}{1+k\pi} = \tfrac{1}{1+k\pi} (1 + k\tfrac{1}{2}|x_{r+1}|^2).
$$

Lifting $X_i=X_{H_i}$ from $E=(\C^{r+1}-Z_E)/G_E$ to $\C^{r+1}$ yields the vector field $X_i = \tfrac{\partial}{\partial \theta_i}$, where $x_i=|x_i|e^{2\pi \sqrt{-1} \theta_i}$ is the $i$-th factor of $\C^{r+1}$, indeed its flow is the rotation $g_i(t):x_i \mapsto e^{2\pi \sqrt{-1}t} x_i$ (fixing all other variables $x_j$). The projection $\pi:E\to B$ is the forgetful map $(x_1,\ldots,x_{r+1})\mapsto (x_1,\ldots,x_r)$ on homogeneous coordinates (this can be easily verified in a local trivialization). Therefore $d\pi\cdot \tfrac{\partial}{\partial \theta_i} = X_{H_i^B}$ in the zero section $B\subset E$.

By Section \ref{Subsection Negative line bundles}, $\omega=(1+k \pi r^2)\pi^*\omega_B + \pi d(r^2)\wedge\theta$ outside of the zero section, where $r: E\to \R$ is the Hermitian norm in the fibre. Therefore,
$$
\begin{array}{lll}
dH_i &=& \omega(\cdot, \tfrac{\partial}{\partial \theta_i}) \\
&=&
 (1+k\pi r^2) \pi^*\omega_B(\cdot, \tfrac{\partial}{\partial \theta_i}) + \pi (d(r^2)\wedge \theta)(\cdot, \tfrac{\partial}{\partial \theta_i})  \\
&=&
(1+k\pi r^2) d(\pi^*H_i^B)(\cdot) + \pi d(r^2)(\cdot) \theta(\tfrac{\partial}{\partial \theta_i}),
\end{array}
$$
where in the last line we use the fact that $r$ does not vary when we rotate $x_i$ (indeed, 
rotating $x_i$ does not change $|x_i|^2$, so it preserves the equation $f_E(x)=0$, so the last entry $\tfrac{1}{2}|x_{r+1}|^2$ of $\mu_E(x)$ will be preserved, and this determines $r$).

The last term above equals $\tfrac{1}{k}d(1+k\pi r^2)(\cdot) \theta(\tfrac{\partial}{\partial \theta_i})$, and recall $1+k\pi r^2 = (1+k\pi)R$, thus:
$$
d(H_i - (1+k\pi)R \pi^*H_i^B) = (1+k\pi) [\tfrac{1}{k}\theta(\tfrac{\partial}{\partial \theta_i})-\pi^*H_i^B] dR.
$$
Now observe that, in general, if $dH=GdR$ for functions $H,G,R$, such that $R$ is a local coordinate, then completing $R=R_1$ to a system of local coordinates $R_1,R_2,\ldots,R_d$, implies that $G=\tfrac{\partial H}{\partial R}$ and $\tfrac{\partial H}{\partial R_j}=0$ for $j\neq 1$. So $H=H(R)$ and $G=G(R)$ only depend on $R$.

In our situation above, this implies that $H_i - (1+k\pi)R \pi^*H_i^B=h_i(R)$ for some function $h_i$ and that $\tfrac{\partial H_i}{\partial R} = \tfrac{1+k\pi}{k}\theta(\tfrac{\partial}{\partial \theta_i})$.

Now evaluate $H_i = (1+k\pi)R \pi^*H_i^B +h_i(R)$ at $D_i = \pi^{-1}(D_i^B)$, using that $H_i^B=0$ on $D_i^B$ and that $H_i=0$ on $D_i$, to deduce that $h_i(R)\equiv 0$. The first claim follows. The second claim follows from $H_{r+1}=\tfrac{1}{2}|x_{r+1}|^2 = \tfrac{1+k\pi}{k}R-\tfrac{1}{k}$ (the first equality holds since $f_E(x)=0$).
%
\end{proof}
%
%
%
\begin{corollary}\label{Corollary Hams in toric case are OK}
The Hamiltonians $H_i$ which define the $S^1$-rotations $g_i$ about the toric divisors $D_i$ satisfy Theorem \ref{Theorem Representation for lager class of Hams} for $\mathrm{Ham}_{\ell\geq 0}(E)$. So, picking lifts $\widetilde{g}_i$ (Section \ref{Subsection Invertibles in the symplectic cohomology}), they give rise to
$$
r_{\widetilde{g}_i}(1)\in QH^*(E)
 \quad \textrm{ and } \quad 
\mathcal{R}_{\widetilde{g}_i}(1) \in SH^*(E)^{\times} 
  \qquad (i=1,\ldots,r+1).
$$
\end{corollary}
\begin{proof}
For $i=1,\ldots,r$:
 $H_i=f_i(y)R$ for $f_i(y) = (1+k\pi)\pi^*H_i^B(1, y)$. Notice $\pi^*H_i^B$ does not depend on $R$, it only depends on the point $(1,y)$ in the sphere bundle $\Sigma=SE = \{R=1\}$ (or rather, on the projection of $(1,y)$ via $\pi: SE \to B$). Notice the $f_i(y)$ are invariant under the Reeb flow (which is rotation in the fibre). Finally $f_i(y)\geq 0$ since $H_i^B\geq 0$, by \eqref{EqHamToric}.
\end{proof}

\begin{lemma}\label{Lemma R of gi standard lift}
Let $g_i^{\wedge}$ be the lift of $g_i$ (in the sense of Section \ref{Subsection Invertibles in the symplectic cohomology}) which maps the constant disc $(c_x,x)$ to itself, for $x\in D_i$.
Then
$$
r_{g^{\wedge}_i}(1)=\mathrm{PD}[D_i]\in QH^2(E)
 \quad \textrm{ and } \quad 
\mathcal{R}_{g^{\wedge}_i}(1) = c^*\mathrm{PD}[D_i]\in SH^2(E)^{\times}.
$$
\end{lemma}
\begin{proof}
This follows by Lemma \ref{Lemma rg1} since $\mathrm{Fix}(g_i)=D_i$, using the fact that $I(g_{i}^{\wedge})=1$ (one can explicitly compute $I(g_i^{\wedge})$ as in \cite[Section 7.8]{Ritter4}).
\end{proof}
\subsection{The problem with relating the lifted rotations}
\label{Subsection The problem with relating the lifted rotations}
Although these lifts $g_i^{\wedge}$ appear to be canonical, a relation $\prod g_i^{a_i}=\prod g_j^{b_j}$ only implies the lifted relation $\prod (g_{i}^{\wedge})^{a_i}=\prod (g_{j}^{\wedge})^{b_j}$ up to factor $t^d$ corresponding to a deck transformation in $\pi_2(M)/\pi_2(M)_0$ (see Section \ref{Subsection Invertibles in the symplectic cohomology}).
%

For closed symplectic manifolds $C$, this issue did not arise -- in fact, for closed $C$ one can define the Seidel representation directly on $\pi_1 \mathrm{Ham}(C)$ (rather than on an extension thereof) by a work-around which involves normalization arguments for the Hamiltonians, as explained for example in \cite[page 433]{McDuff-Salamon2}. For non-compact $M$, it is unclear to us whether a work-around exists (of course those normalization arguments involving integration of $\omega^{\mathrm{top}}$ over $M$ will fail).

In any case, for monotone $M$ this is never a problem since one can measure the discrepancy between $\prod \mathcal{R}_{g_i^{\wedge}}(1)^{a_i}$ and $\prod \mathcal{R}_{g_j^{\wedge}}(1)^{b_j}$ simply by comparing gradings: 

\begin{lemma}\label{Lemma relations in SH}
A relation $\prod g_i^{a_i}=\prod g_j^{b_j}$ corresponds to the following relation in $SH^*(M)$:
$$
\textstyle \prod x_i^{a_i} = T^{\sum a_i-\sum b_j} \prod x_j^{b_j} \qquad (\textrm{where }x_i=\mathrm{PD}[D_i]).
$$
\end{lemma}
\begin{proof}
 By Lemma \ref{Lemma R of gi standard lift}, $\mathcal{R}_{g_i}^{\wedge}(1)=x_i$ has grading $2=2\, \mathrm{codim}_{\C} D_i$. The claim follows since, for $M$ monotone, $SH^*(M)$ is $\Z$-graded over the (graded) Novikov ring (Section \ref{Subsection Novikov ring}).
%
\end{proof}
%
%
\subsection{The quantum SR-relations for monotone toric negative line bundles}
\label{Subsection The SR-relations for monotone toric negative line bundles}
By Lemma \ref{Lemma line bundle from fan of base}, $E$ arises as $\mathcal{O}(\sum n_i D_i^B)\stackrel{\pi}{\longrightarrow} B$, for $n_i\in \Z$, so $c_1(E)=\sum n_i \mathrm{PD}[D_i^B]$. The edges of the fan for $E$ are
$$e_1=(b_1,-n_1),\; \ldots,\; e_r =(b_r,-n_r),\; e_{f}=(0,\ldots,0,1) \in \Z^{n+1}.$$
where $b_j$ are the edges of the fan for $B$. We often use the index `$f$' instead of $r+1$, so $e_f=e_{r+1}$, to emphasize that this index corresponds to the fibre coordinate $x_f=x_{r+1}$.
The cones of the fan for $E$ are $\mathrm{span}_{\R_{\geq 0}} \{e_{j_1},\ldots,e_{j_k}\}$ and $\mathrm{span}_{\R_{\geq 0}} \{e_{j_1},\ldots,e_{j_k},e_{f}\}$ whenever $\mathrm{span}_{\R_{\geq 0}}\{b_{j_1},\ldots,b_{j_k}\}$ is a cone for $B$.

\begin{corollary}
 The primitive collections for $E$ are those of $B$: $I=I^B$.
\end{corollary}
\begin{proof}
 If $I^B \cup \{e_{f}\}$ was primitive, then $I^B$ would determine a cone in $E$ and hence in $B$, but then $I^B \cup \{e_{f}\}$ would be a cone for $E$ and so would not be primitive.
\end{proof}
%
%
\begin{corollary}\label{Corollary SR relns of E in terms of B}
Recall the relations for $B$ are:
\begin{enumerate}
 \item $\sum \langle \xi, b_i \rangle x_i = 0$ as $\xi$ ranges over the standard basis of $\R^r$; \hspace{6ex} \emph{(linear relations)}
  
 \item $\displaystyle \prod_{p\in I^B} x_{i_p} = T^{|I^B|-\sum c_q} \prod_q x_{j_q}^{c_q}$ for primitive collections $I^B$. \hspace{3ex} \emph{(quantum SR-relations)}\\
 \emph{(Corresponding to the relation $\sum b_{i_p} = \sum c_q b_{j_q}$ among edges)}
\end{enumerate}
Then the relations for $E$ are:
\begin{enumerate}
 \item the linear relations for $B$;
 \item the new linear relation $x_{f} = \sum n_i x_i$;
 \item $\displaystyle \prod_{p\in I^B} x_{i_p} = T^{|I^B|-\sum c_q - c_f} \cdot x_f^{c_f} \cdot \prod_{q} x_{j_q}^{c_q}$ where $c_f = -\sum_{p\in I^B} n_{i_p} + \sum_{q} c_q \cdot n_{j_q}$.\\
\emph{(Corresponding to the relation $\sum_{p\in I^B} e_{i_p} = \sum_q c_q e_{j_q} + c_f e_f$ among edges)}
\end{enumerate}
\end{corollary}

\begin{theorem}\label{Theorem relations in QH for Fano toric NLB}
The above linear relations and quantum SR-relations hold in $QH^*(E)$.
\end{theorem}
\begin{proof}
 The linear relations from $B$ clearly hold also in $H^*(E)$ since $H^*(E)\cong H^*(B)$. The new linear relation $x_f=\sum n_i x_i$ holds because the toric divisor $D_{r+1}=[B]$ as an lf-cycle is Poincar\'e dual to $\pi^*c_1(E)$ which in turn is Poincar\'e dual to $\pi^*(\sum n_i [D_i^B])$.

 The quantum SR-relations hold in $SH^*(E)$ by Lemma \ref{Lemma relations in SH}. The fact that they hold in $QH^*(E)$ follows from using the representation 
$
r: \widetilde{\pi_1}\mathrm{Ham}_{\ell \geq 0}(E,\omega) \to QH^*(E)
$
and Theorem \ref{Theorem Representation into SH for new Hamiltonians}, and using the following Lemma which ensures that only Hamiltonians of positive slope are involved in the quantum SR-relations.
\end{proof}

\begin{lemma}\label{Lemma Quantum SR-relations involve positive slope Hamiltonians}
 The quantum SR-relations for $E$ involve only positive slope Hamiltonians.
\end{lemma}
\begin{proof}
 By Corollary \ref{Corollary Hams in toric case are OK}, any monomial involving non-negative powers of the rotations $g_1^{\wedge}$, $\ldots$, $g_r^{\wedge}$, $g_f^{\wedge}$ will be generated by a Hamiltonian of positive slope. The quantum SR-relations in $QH^*(E)$ arise from comparing the values of the homomorphism $r: \widetilde{\pi_1}\mathrm{Ham}_{\ell \geq 0}(E,\omega) \to QH^*(E)$ (Theorem \ref{Theorem Representation into SH for new Hamiltonians}) on the rotations $\prod_{p\in I^B} g_{i_p}^{\wedge}$ and $(g_f^{\wedge})^{c_f} \cdot \prod_{q} (g_{j_q}^{\wedge})^{c_q}$. The values are well-defined because all the powers in those monomials are non-negative: $c_{q}$ are positive integers by definition, and we now prove $c_f\geq 0$.

Applying the formula for $c_1(TM)(\beta_I)$ explained at the end of Section \ref{Subsection Review of the Batyrev-Givental presentation of QH}:
$$
c_1(TE)(\beta_I) = |I^B|-\sum c_q -c_f = c_1(TB)(\beta_I) - c_f.
$$
But $c_1(TE) = \pi^*c_1(TB) + \pi^*c_1(E)$, by splitting $TE$ into horizontal and vertical spaces, so 
$$c_f = -c_1(E)(\beta_I) = k \omega_B(\beta_I)>0,$$ 
using Batyrev's result: $\omega_B(\beta_I)>0$. \emph{(Remark: we identify $\beta_I\in H_2(E)\cong H_2(B)$ calculated in $E$ and $B$ as it is prescribed by the same intersection products: $\beta_I \cdot D_{i_p}=1$, $\beta_I\cdot D_{j_q} = -c_{q}$.)} 
\end{proof}

\subsection{Presentation of $\mathbf{QH^*(E)}$ and $\mathbf{SH^*(E)}$}
\label{Subsection Presentation of QHE SHE for monotone toric nlb}
%
\begin{theorem}\label{Theorem presentation of QH and SH for Fano toric NLB}
Let $E\to B$ be a monotone toric negative line bundle. Then the presentation of the quantum cohomology of $E$ can be recovered from the presentation for $B$:
$$
\boxed{
\begin{array}{l}
QH^*(B) \cong \Lambda [x_1,\ldots,x_r] / (\textrm{linear rel'ns in }B, \textrm{quantum SR-rel'ns }\prod x_{i_p} = T^{|I^B|-\sum c_q} \cdot \prod x_{j_q}^{c_q})
\\
QH^*(E) \cong \Lambda [x_1,\ldots,x_r] / (\textrm{linear rel'ns in }B,\, \textrm{and}\,\prod x_{i_p} = T^{|I^B|-\sum c_q - c_f} \cdot (\sum n_i x_i)^{c_f} \cdot \prod x_{j_q}^{c_q})
\end{array}
}
$$
running over primitive relations $\sum b_{i_p} = \sum c_{q} b_{j_q}$ in $B$, $i_p\in I^B$, and $c_f = -\sum n_{i_p} + \sum c_q \cdot n_{j_q}.$

Moreover, the symplectic cohomology is the localization of the above ring at $x_f=\sum n_i x_i$,
$$
\boxed{
SH^*(E) \cong \Lambda [x_1,\ldots,x_r,z] / (\textrm{linear relations in }B, \textrm{quantum SR-relations}, z\cdot {\textstyle\sum} n_i x_i - 1)
}
$$
and the canonical map $c^*: QH^*(E) \to SH^*(E)$ is the canonical localization map sending $x_i \mapsto x_i$ (recall this induces the isomorphism in Theorem \ref{Theorem SH of negative line bundle}).
\end{theorem}
\begin{proof}
 The computation of $QH^*(E)$ follows from Theorem \ref{Theorem relations in QH for Fano toric NLB} and Lemma \ref{Lemma McDuffTolman Algebra trick}. The compuation of $SH^*(E)$ then follows, using Corollary \ref{Corollary SH is a localization of QH}.
\end{proof}
%
\subsection{Examples: computation of $\mathbf{QH^*}$ and $\mathbf{SH^*}$ for $\mathbf{\mathcal{O}_{\P^m}(-k)}$ and $\mathbf{\mathcal{O}_{\P^1\times \P^1}(-1,-1)}$}
\label{Subsection Compuation of QH and SH of O(-k)}
\label{Subsection Compuation of QH and SH of O(-1,-1)}

\begin{corollary}
Let $E=\mathcal{O}_{\P^m}(-k)$ be monotone (meaning: $1\leq k \leq m$). Then
$$
\begin{array}{rcl}
 QH^*(E) &\cong & \Lambda[x]/(x^{1+m}-T^{1+m-k}(-kx)^k)\\
 SH^*(E) &\cong & \Lambda[x]/(x^{1+m-k}-T^{1+m-k}(-k)^k).
\end{array}
$$
\end{corollary}
\begin{remark} This recovers, for $1\leq k\leq m/2$, the computation from \cite{Ritter4} (see Theorem \ref{Theorem SH of negative line bundle}) which was a rather difficult virtual localization computation -- that approach was computationally unwieldy for $m/2< k \leq m$. So for $m/2< k\leq m$ the above result is new.
\end{remark}
\begin{proof}
The standard presentation of $QH^*(\P^m)$ involves variables $x_1,\ldots,x_{m+1}$; the linear relations impose that all $x_j$ are equal, call this variable $x$ (which represents $\omega_{\P^m}$); and the quantum SR-relation is $x_1 \cdots x_{m+1}=T^{m+1}$. Thus $QH^*(\P^m)=\Lambda[x]/(x^{1+m}-T^{1+m})$.

Representing $\mathcal{O}(-k)$ as $E=\mathcal{O}(-kD_{m+1})$, apply Theorem \ref{Theorem presentation of QH and SH for Fano toric NLB}. The quantum SR-relation in $E$ becomes $x_1 \cdots x_{m+1}=T^{1+m-k}(-kx_{m+1})^k$ since $c_f=k$. The claim follows. 
\end{proof}

\begin{corollary}
Let $E=\mathrm{Tot}(\mathcal{O}(-1,1)\to \P^1\times \P^1)$, that is: $B=\P^1\times\P^1$ with $\omega_B=\omega_1+\omega_2$, and $c_1(E)=-(\omega_1 + \omega_2)$ (where $\omega_j=\pi_j^*\omega_{\P^1}$ via the two projections $\pi_j:B\to \P^1$). Then 
$$
\begin{array}{rcl}
 QH^*(B) & \cong & \Lambda[x_1,x_2]/(x_1^2-t, x_2^2-t)\\
 QH^*(E) &\cong & \Lambda[x_1,x_2]/(x_1^2+t(x_1+x_2), x_2^2+t(x_1+x_2))\\
 SH^*(E) &\cong & \Lambda[x_1,x_2]/(x_1^2-4t^2, x_2^2-4t^2,x_1+x_2+4t) \cong \Lambda.
\end{array}
$$
where the $x_j=\omega_j$, and the last isomorphism\footnote{This assumes $\mathrm{char}\,\mathbb{K} = 0$. It follows from $0=(x_1^2-4t^2)-(x_2^2-4t^2)=(x_1-x_2)(x_1+x_2)$ and using $x_1+x_2=-4t$.}
 sends $x_j \mapsto -2t$. In particular, using the convention that $T=t^{1/\lambda_M}$ lies in grading $2$ (Section \ref{Subsection Novikov ring}), $\omega_B=x_1+x_2$ has:
\begin{itemize}

\item minimal polynomial $X^3-4tX=X(X-2T)(X+2T)$ and characteristic polynomial $X^4-4tX^2=X^2(X-2T)(X+2T)$ in $QH^*(B)$,

\item minimal polynomial  $X^3+4tX^2=X^2(X+4T)$ and characteristic polynomial\\ $X^4+4t X^3=X^3(X+4T)$ in $QH^*(E)$,

\item minimal and characteristic polynomials both equal to $X+4T$ in $SH^*(E)$.
\end{itemize}
\end{corollary}
\begin{proof}
The computation of $QH^*(B)$ can be obtained from the moment polytope (a square with inward normals $e_1=(1,0)$, $-e_1$, $e_2=(0,1)$ and $-e_2$) or by explicitly computing the quantum product $\omega_1*\omega_2=\omega_1\wedge \omega_2$. The minimal polynomial of $\omega_B$ follows by computation. The matrix for multiplication by $\omega_B=x_1+x_2$ is $\left(\begin{smallmatrix} 0 & t & t & 0 \\ 1 & 0 & 0 & t\\ 1 & 0 & 0 & t\\ 0 & 1 & 1 & 0 \end{smallmatrix}\right)$ in the basis, $1,x_1,x_2,x_1x_2$, and since this has rank 2 the characteristic polynomial has an extra $X$ factor.

Observe that $k=1$, $\lambda_B=2$, $\lambda_E=1$. So $t_B=T_B^2$, $T_E=t_E$.
By Theorem \ref{Theorem Minimal poly and char poly from B to E} for $E=\mathcal{O}(-D_1-D_2)$, the presentation of $QH^*(E)$ follows by replacing $t_B$ by $t_E (-x_1-x_2)$. One obtains the characteristic polynomial of $\pi^*\omega_B$ by inspecting the new matrix $\left(\begin{smallmatrix} 0 & 0 & 0 & 0 \\ 1 & -t & -t & 2t^2 \\1 & -t & -t & 2t^2 \\ 0 & 1 & 1 & -2t \end{smallmatrix}\right)$.

To compute $SH^*(E)$ localize at $c_1(E)$: introduce $z$ with $1-z(-x_1-x_2)=0$. For  $X=x_1+x_2$, $X^3(X+4T)=0\in QH^*(E)$. Multiplying by $z^3$ we get $X=-4T$, $z=-(4T)^{-1}$.
\end{proof}

\noindent\emph{Remarks. With some effort, the calculation of $QH^*(\mathcal{O}_{\P^1\times \P^1}(-1,-1))$ can in fact be verified by hand (that is by explicitly computing $\omega_i*\omega_j$).
One can also check that $QH^*(\P^1\times \P^1)$ is semisimple (a direct sum of fields as an algebra) since 
%
%
the $0$-eigenspace splits as an algebra into two fields spanned by $2Ty \pm y^2$ where $y=x_1-x_2$.
Whereas $QH^*(\mathcal{O}_{\P^1\times \P^1}(-1,-1))$ is not semisimple, due to the generalized 0-eigenspace which creates zero divisors in the ring.
%
%
%
}

\subsection{The change in presentation from $\mathbf{QH^*(B)}$ to $\mathbf{QH^*(E)}$ is a change in the Novikov variable}
\label{Subsection The change in presentation is a change in the Novikov variable}
%
In this Section we will restrict the Novikov field of Section \ref{Subsection Novikov ring} to a smaller subring $R=\K[t,t^{-1}]$ or $\K[t]$, and we write $QH^*(M;R)$ when we work over $R$ (recall the quantum product only involves positive integer powers of $t$). We put subscripts $B,E$ on $t$ (so $t_B=T_B^{\lambda_B}$, $t_E=T_E^{\lambda_E} = T_E^{\lambda_B-k}$) when we want to distinguish the variable $t$ used for $B$ and for $E$.

Recall by Theorem \ref{Theorem presentation of QH and SH for Fano toric NLB} that
$$\sum n_i x_i = c_1(E)=-k[\omega_B] = -\tfrac{k}{\lambda_B}c_1(TB)=-\tfrac{k}{\lambda_B}\sum x_i.$$ 

\begin{theorem}\label{Theorem change of Novikov param}
The presentations of $QH^*(B)$, $QH^*(E)$ in Theorem \ref{Theorem presentation of QH and SH for Fano toric NLB} are identical after the change of Novikov parameter:
$$\boxed{T^{\lambda_B}\mapsto T^{\lambda_B-k}(\sum n_i x_i)^{k}}$$
\end{theorem}
\begin{proof}
By Equation \eqref{EqnC1andOmega} in Section \ref{Subsection Review of the Batyrev-Givental presentation of QH}, for primitive $I=I^B$,
$$\textstyle
|I^B|-\sum c_q=c_1(TB)(\beta_I).
$$
By the proof of Lemma \ref{Lemma Quantum SR-relations involve positive slope Hamiltonians},
$-c_f = c_1(TE)(\beta_I)-c_1(TB)(\beta_I) = c_1(E)(\beta_I),$
so the quantum SR-relations for $QH^*(E)$ in Theorem \ref{Theorem presentation of QH and SH for Fano toric NLB} become
$$
\textstyle
\prod x_{i_p} = T^{c_1(TE)(\beta_I)} \cdot (\sum n_i x_i)^{-c_1(E)(\beta_I)} \cdot \prod x_{j_q}^{c_q}.
$$
So the presentations of $QH^*(B)$, $QH^*(E)$ are identical after the change of Novikov parameter:
$$\textstyle
T^{c_1(TB)(\beta_I)} \mapsto T^{c_1(TE)(\beta_I)} \cdot (\sum n_i x_i)^{-c_1(E)(\beta_I)},
$$
and it remains to check that this is consistent as $\beta_I$ varies (varying $I=I^B$). 
But $c_1(TB)(\beta_I)=\lambda_B [\omega_B](\beta_I)$, $c_1(TE)(\beta_I)=(\lambda_B-k)[\omega_B](\beta_I)$, and $-c_1(E)(\beta_I)=k[\omega_B](\beta_I)$.
So the above change of Novikov parameters are all implied by $T^{\lambda_B}\mapsto T^{\lambda_B-k}(\sum n_i x_i)^{k}$, as claimed.
%
%
%
\end{proof}

\begin{theorem}\label{Theorem Minimal poly and char poly from B to E}
There is a ring homomorphism (which is not a $\Lambda$-module homomorphism)
\begin{equation}\label{Equation varphi QHB to SHE}
\begin{array}{l}
\varphi\co QH^*(B;\K[t,t^{-1}]) \to  SH^*(E;\K[t,t^{-1}])\\ \varphi(x_i)=x_i,\quad \varphi(t_B)=t_E c_1(E)^k = t_E (\textstyle\sum n_i x_i)^k
\end{array}
\end{equation}
equivalently $\varphi(T_B^{\lambda_B})=T_E^{\lambda_B-k} (\sum n_i x_i)^{k}$. Over $\K[t]$, $\varphi$ lifts to $QH^*(E;\K[t])$:
\begin{equation}\label{Equation varphi QHB to QHE}
\varphi\co QH^*(B;\K[t])\to QH^*(E;\K[t]), \quad \varphi(x_i)=x_i, \quad \varphi(t_B)=t_E (\textstyle\sum n_i x_i)^k,
\end{equation}
so we obtain a factorization: $\varphi\co QH^*(B;\K[t])\to QH^*(E;\K[t]) \to QH^*(E)\to SH^*(E)$. \emph{Remark. The restriction to $\K[t]$ is necessary, as $c_1(E)$ is never invertible in $QH^*(E)$ by \cite{Ritter4}.}

Any polynomial relation $P(x,t_B,t_B^{-1})=0$ in $QH^*(B)$ yields $P(x,\varphi(t_B),\varphi(t_B)^{-1})=0$ in $SH^*(E)$; and any polynomial relation $P(x,t_B)=0$ in $QH^*(B)$ yields $P(x,\varphi(t_B))=0$ in $QH^*(E)$ (this applies to: linear relations, quantum SR-relations, characteristic/minimal polynomials of $y\in H^{2d}(B)\cong H^{2d}(E)$ but we don't claim the image is characteristic/minimal).
\end{theorem}
\begin{proof}
The ring homomorphism $\varphi$ in \eqref{Equation varphi QHB to SHE} and \eqref{Equation varphi QHB to QHE} is well-defined since the linear/quantum relations for $B$ map to those for $E$. In particular in \eqref{Equation varphi QHB to SHE} we use the fact that $\pi^*c_1(E)$ is an invertible in $SH^*(E)$ so $\varphi(t_B^{-1})$ is well-defined. 

The final part of the claim follows immediately from the existence of $\varphi$ if the polynomials $P(x,t,t^{-1}), P(x,t)$ were known to vanish in the smaller rings $QH^*(B;\K[t,t^{-1}]),$ $QH^*(B;\K[t])$ rather than in $QH^*(B)$. It remains to justify why the vanishing in $QH^*(B)$ implies the vanishing in those smaller rings. For monotone toric $X$, by exactness of localization at $t$,
\begin{equation}\label{Equation Exactness of localization}
QH^*(X;\K[t,t^{-1}])= QH^*(X;\K[t])\otimes_{\K[t]} \K[t,t^{-1}].
\end{equation}
Since the toric divisors and any of their intersections yield a chain complex computing the cohomology, the quantum cohomologies $QH^*(X;\K[t])$, $QH^*(X;\K[t,t^{-1}])$ are free modules respectively over $\K[t]$, $\K[t,t^{-1}]$.
%
Similarly, by exactness of completion in $t$, one can replace $\K[t,t^{-1}]$ in \eqref{Equation Exactness of localization} by the ring $\K(\!(t)\!)$ of Laurent series. Since $QH^*$ is $\Z$-graded, \eqref{Equation Exactness of localization} also holds if we replace $\K[t,t^{-1}],\K[t]$ by $\Lambda,\K(\!(t)\!)$ respectively. It follows that 
%
%
$$
QH^*(X)=QH^*(X;\Lambda)= QH^*(X;\K[t,t^{-1}])\otimes_{\K[t,t^{-1}]} \Lambda,
$$
%
%
and, as before, the two $QH^*$ are free modules over the relevant ring. Thus, if $P(x,t_B,t_B^{-1})=0$ in $QH^*(X)$ then this also holds in $QH^*(X;\K[t,t^{-1}])$. So in particular if $P(x,t_B)=0$ in $QH^*(X)$, then this also holds in $QH^*(X;\K[t,t^{-1}])$ and hence in $QH^*(X;\K[t])$ by \eqref{Equation Exactness of localization}.
%
%

We now prove the claim about the characteristic/minimal polynomials of $y\in H^{2d}(B)$. Notice we are assuming that $y\in H^{2d}(B)$ does not involve $t$ and lies entirely in a fixed degree.
It follows that these polynomials are homogeneous with respect to the $\Z$-grading on $QH^*(B)=QH^*(B;\Lambda)$. Now consider the algebra $QH^*(B;\K)$, in other words setting $t=1$. This algebra is no longer $\Z$-graded like $QH^*(B)$, but it is still well-defined (in the monotone setting there are no compactness issues in defining the quantum cohomology). Observe that the minimial and characteristic polynomials for $y\in QH^*(B;\K)$ and $y \in QH^*(B)$ are related by homogenization (i.e. insert powers of $t$ to ensure the monic polynomials are homogeneous in the $\Z$-grading). It follows that these polynomials only involve positive powers of $t$, so we may apply the map $\varphi$, using that $\varphi(y)=y\in QH^{2d}(E)$ since it does not involve $t$.
\end{proof}

\subsection{$SH^*(E)$ is the Jacobian ring}
\label{Subsection The Jacobian ring for nlb}
%
\begin{theorem}\label{Theorem Jacobian ring is SH}
The Jacobian ring for $E$ (Definition \ref{Definition Jacobian ring}) is isomorphic to $SH^*(E)$ via
$$
\begin{array}{rcl}
SH^*(E)\cong QH^*(E)[c]/(c\cdot \pi^*c_1(E)-1) &\to & \mathrm{Jac}(W_E),\\
\mathrm{PD}[D_i] &\mapsto & t^{-\lambda_i^E}z^{e_i}\\ \pi^*c_1(E)=\mathrm{PD}[B] &\mapsto & z_{n+1}.
\end{array}
$$
which is the $i$-th summand in the definition of the superpotential $W$ (Definition \ref{Definition Superpotential}). 

In particular, $c_1(TE)$ maps to $W_E$. Moreover, $t^{-\lambda_i^E}z^{e_i}=(tz_{n+1}^k)^{-\lambda_i^B} z^{(b_i,0)}$ in terms of the data $b_i,\lambda_i^B$ which determines the moment polytope of $B$.
\end{theorem}
\begin{proof}
 This follows from Theorem \ref{Theorem presentation of QH and SH for Fano toric NLB}, Corollary \ref{Corollary QH of noncompact toric}, Lemma \ref{Lemma moment polytope for neg l bdle} and Example \ref{Example Superpotential of nlb}.
\end{proof}

In Example \ref{Example Superpotential of nlb} of Appendix A (Section \ref{Section The moment polytope of a toric negative line bundle}) we check that the superpotential of a negative line bundle $E$, with $c_1(E)=-k[\omega_B]$, is
$$
\begin{array}{rcl}
W_E(z_1,\ldots,z_n,z_{n+1}) &= & z_{n+1} +  \left.W_B(z_1,\ldots,z_n)\right|_{\left(t \textrm{ replaced by }tz_{n+1}^k\right)}\\
& = & z_{n+1} + \varphi(W_B)
\end{array}
$$
where $\varphi(z_i)=z_i$ and $\varphi(t_B)=t_E z_{n+1}^k$ (as is consistent with Theorem \ref{Theorem Minimal poly and char poly from B to E}).
%
%
%
Working over $\K[t]$ instead of $\Lambda$, this $\varphi$ defines a ring homomorphism 
$$\varphi: \mathrm{Jac}(W_B;\K[t])\to \mathrm{Jac}(W_E;\K[t])$$
identifiable precisely with the $\varphi: QH^*(B;\K[t])\to SH^*(E;\K[t])$ from Theorem \ref{Theorem Minimal poly and char poly from B to E}.

\begin{remark} At this stage one could, as was done in \cite{RitterSmith}, recover the eigenvalues of $c_1(TE)$ acting on $QH^*(E)$ in terms of those of $c_1(TB)$ acting on $QH^*(B)$ by working with $\mathrm{Jac}(W_E),\mathrm{Jac}(W_B)$ and comparing the critical values of $W_E,W_B$. But this will not recover the multiplicities of the eigenvalues and it will not recover Theorem \ref{Theorem Eigenvalues of c1 from B to E}.
\end{remark}
%
\subsection{The eigenvalues of $c_1(TE)$ and the superpotential $W_E$}
\label{Subsection The eigenvalues of c_1(TX)}
$$\textbf{We now assume }\K\textbf{ is algebraically closed}$$ in the definition of $\Lambda$ in Section \ref{Subsection Novikov ring} (but we make no assumption on the characteristic). This is necessary so that we can freely speak of the eigenvalues of $c_1(TE)$ (the roots of the characteristic polynomial of quantum multiplication by $c_1(TE)$ on $QH^*(E)$).
Indeed, for $X$ monotone, $QH^*(X)$ can be defined over $\K$ at the cost of losing the $\Z$-grading, and a splitting $\prod (x-\mu_i)$ of the characteristic polynomial of $c_1(TX)$, where $\mu_i\in \K$, immediately yields a splitting $\prod (x-\mu_i T)$ when working over $\Lambda$ instead of $\K$. This follows because $QH^*(X)$ is $\Z$-graded and $x,T$ are both in degree $2$. Observe that this factorization is not legitimate over $\K[t]$ in general,  since $T=t^{1/\lambda_X}$ is a fractional power unless $\lambda_X=1$. However, the following Lemma (and later results) will show that the factors $x-\mu_i T$ can be collected in $\lambda_X$-families, so the resulting factorization (in higher order factors) will be legitimate over $\K[t]$.

Since $c_1(TB)\in H^2(B,\Z)$ is integral and $c_1(TB)=\lambda_B [\omega_B]$, by rescaling $\omega_B$ we may assume $[\omega_B]$ is a primitive integral class and $\lambda_B\in \N$ (called the index of the Fano variety $B$).

\begin{lemma}\label{Lemma Critical points of W come in families}\strut
For monotone toric negative line bundles $E\to B$, with $c_1(E)=-k[\omega_B]$,
\begin{enumerate}
 \item\label{Item evalues for B} Non-zero eigenvalues $\lambda=\mu T_B$ of $c_1(TB)$ arise in $\lambda_B$-families $\lambda,\xi \lambda, \ldots, \xi^{\lambda_B-1} \lambda$ where $\xi\neq 1$ is a $\lambda_B$-th root of unity.

 \item\label{Item action of lambdaB - k roots} There is a free action of $(\lambda_E=\lambda_B -k)$-th roots of unity on the critical points of $W_E$.

 \item\label{Item nonzero evals of E} The non-zero eigenvalues of $c_1(TE)$ are precisely the critical values of $W_E$ and they arise in families of size $\lambda_E=\lambda_B-k$.

 \item \label{Item dim gen subespaces of B} The dimensions of the generalized sub-eigenspaces $\ker (c_1(TB)-\lambda\,\mathrm{Id})^d\subset QH^*(B)$ for $d\in \N$ are invariant under the action $\lambda\mapsto \xi\lambda$ by $\lambda_B$-th roots of unity.

 \item \label{Item dim gen subespaces of E} The dimensions $\dim \ker (c_1(TE)-\lambda\,\mathrm{Id})^d\subset QH^*(E)$ are invariant under the action $\lambda\mapsto \xi\lambda$ by $(\lambda_B-k)$-th roots of unity.
\end{enumerate}
\end{lemma}
\begin{proof}
Recall from Section \ref{Subsection Batyrev's argument: from the presentation of QH to JacW} that $QH^*(B)\cong \mathrm{Jac}(W_B)$, $c_1(TB)\mapsto W_B$ by Batyrev. For purely algebraic reasons (Ostrover-Tyomkin \cite[Corollary 2.3 and Section 4.1]{Ostrover-Tyomkin}) it follows that the eigenvalues of $c_1(TB)$ acting on $QH^*(B)$ are precisely the critical values of $W_B$ (this result was originally proved by considering special Lagrangians in $B$, and is due to Kontsevich, Seidel, and Auroux \cite[Section 6]{Auroux}). Then \eqref{Item evalues for B} follows by Corollary \ref{Corollary crit pts of W arise in families}.

Now run the same argument for $E$. Claim \eqref{Item action of lambdaB - k roots} follows by Corollary \ref{Corollary crit pts of W arise in families}. 
By Theorem \ref{Theorem Jacobian ring is SH}, 
\begin{equation}\label{Equation Jac is SH is QH quotiented}
\mathrm{Jac}(W_E)\cong SH^*(E)\cong QH^*(E)/(\textrm{generalized }0\textrm{-espace of }\pi^*c_1(E)),\;  W_E\mapsto c_1(TE).
\end{equation}
 Again, for algebraic reasons, the first part of \eqref{Item nonzero evals of E} follows, and the second part  follows by \eqref{Item action of lambdaB - k roots}. 

Now we prove \eqref{Item dim gen subespaces of B}. Firstly, $\ker(c_1(TB)-\lambda)^d\cong \ker(W_B-\lambda)^d$ (acting on $\mathrm{Jac}(W_B)$). But 
$$
\ker(W_B(z)-\xi\lambda)^d=\ker \xi^d(W_B(\xi^{-1}z)-\lambda)^d\cong \ker(W_B(z)-\lambda)^d,
$$
where in the first equality we use $W_B(\xi z)=\xi W_B(z)$ (Lemma \ref{Lemma W(xi z) is xi W(z)}), and second equality is the isomorphism $f(z) \mapsto f(\xi z)$. So \eqref{Item dim gen subespaces of B} follows.
Similarly \eqref{Item dim gen subespaces of E} follows from \eqref{Equation Jac is SH is QH quotiented} by making $c_1(TE)-\lambda$ act on $SH^*(E)$ for $\lambda\neq 0$ (for $\lambda=0$ there is nothing to prove).
\end{proof}

For generation results, it will be important to know the dimensions of the generalized eigenspaces of the quantum multiplication action of $c_1(TE)=(\lambda_B-k)[\omega_E]\in QH^*(E)$ in terms of that for $c_1(TB)=\lambda_B[\omega_B]\in QH^*(B)$. The next result aims to describe the Jordan normal form (JNF) for $[\omega_E]=\pi^*[\omega_B]\in QH^*(E)$ in terms of the JNF for $[\omega_B]\in QH^*(B)$. Recall the JNF of $[\omega_B]$ is determined by taking the primary decomposition of $QH^*(B)$ viewed as a finitely generated torsion module over the principal ideal domain (PID) $\Lambda[x]$ (using that $\Lambda$ is a field), where $x$ acts by multiplication by $[\omega_B]$. Namely, each summand $\Lambda[x]/(x-\mu T)^d$ corresponds to a $d\times d$ Jordan block for $\lambda = \mu T$. By Lemma \ref{Lemma Critical points of W come in families}\eqref{Item dim gen subespaces of B}, for non-zero $\mu$ the factors $\Lambda[x]/(x-\mu T)^d$
arise in $\lambda_B$-families. By the Chinese Remainder Theorem such a family yields a summand $\Lambda[x]/(x^{\lambda_B}-\mu^{\lambda_B} t)^d$, corresponding to a $\lambda_B$-family of $d\times d$ Jordan blocks for the eigenvalues listed in Lemma \ref{Lemma Critical points of W come in families}\eqref{Item evalues for B}. So the JNF yields the $\Lambda[x]$-module isomorphism \eqref{Equation primary decomposition of QH(B)}.

In the following result, we use the convention that for $f\in \K[t_B][x]$, $\varphi(f)$ means we replace $t_B$ by $t_E(-kx)^k$, as is consistent with the definition of $\varphi$ in Theorem \ref{Theorem Minimal poly and char poly from B to E} and the $x$-actions. We remark that when $\mathrm{char}(\K)$ divides $k$, the SR-relations in $E$ are the classical SR-relations, by Theorem \ref{Theorem presentation of QH and SH for Fano toric NLB}, so $QH^*(E;\Lambda)\cong H^*(E)\otimes \Lambda$ as a ring, which we are not interested in. 

\begin{theorem}[Assuming $\mathrm{char}(\K)$ does not divide $k$]\label{Theorem Eigenvalues of c1 from B to E}
The isomorphism in \eqref{Equation primary decomposition of QH(B)} determines the $\Lambda[x]$-module isomorphisms in \eqref{Equation primary decomposition of QH(E)}, with $x$ acting as $[\omega_E]=\pi^*[\omega_B]$.

The characteristic polynomial of $[\omega_E]$ is the image under $\varphi$ of the characteristic polynomial of $[\omega_B]$. The minimal polynomial of $[\omega_E]$ is, possibly after dropping some $x$-factors, the image of the minimal polynomial of $[\omega_B]$. The $(\lambda_B-k)$-family of non-zero eigenvalues $\mu_j^E T_E$ of $\pi^*[\omega_B]\in QH^*(E)$ of Lemma \ref{Lemma Critical points of W come in families}\eqref{Item evalues for B} arises from a $\lambda_B$-family $\mu_j T_B$ for $B$, via $$(\mu_j^E)^{\lambda_B-k} = (-k)^k \mu_j^{\lambda_B}.$$
\end{theorem}
\begin{proof}
\textbf{Step 1}: Given $f(x)\in \K[t_E][x]$ and $v\in QH^*(B;\K[t])$, suppose $f(x)\varphi(v)=0$ in $QH^*(E;\K[t])$, where the polynomial $f(x)$ acts on $QH^*(E;\K[t])$ by making $x$ act by multiplication by $[\omega_E]=\pi^*[\omega_B]$. Then $x^{\textrm{large}}f(x)v=0$ in $QH^*(B;\K[t])$ after the change of variables 
\begin{equation}\label{Equation change of variables tE tB}
t_E(-kx)^k=t_B.
\end{equation}
\textbf{Proof of Step 1}: Since $f(x)\cdot \varphi(v)=0$, $f(x)\varphi(v)$ lies in the ideal generated by the linear relations and SR-relations for $E$. By Theorems \ref{Theorem presentation of QH and SH for Fano toric NLB} and \ref{Theorem Minimal poly and char poly from B to E}, the linear relations in $B,E$ agree and the SR-relations agree up to \eqref{Equation change of variables tE tB}. Thus $(-kx)^{\textrm{large}} f(x)\varphi(v)$ lies in the ideal generated by the linear/SR-relations for $B$, where the factor $(-kx)^{\textrm{large}}$ ensures that all the occurrences of $t_E$ can be turned into $\textrm{constant}\cdot x^{\textrm{positive}}\cdot t_B$ via \eqref{Equation change of variables tE tB}. Thus $(-kx)^{\textrm{large}}f(x)v=0$ in $B$ (making the power of $(-kx)$ larger if necessary to ensure that in the expression of $f(x)$, all occurrences of $t_E$ are again replaced via \eqref{Equation change of variables tE tB}). Step 1 follows since $k$ is invertible in $\K$. $\checkmark$

\textbf{Step 2}: We claim that \eqref{Equation primary decomposition of QH(B)} holds for $QH^*(B;\K[t,t^{-1}])$ after replacing $\Lambda[x]$ by $\K[t,t^{-1}][x]$.

\textbf{Proof of Step 2}: This is not immediate since $\K[t,t^{-1}][x]$ is not a PID. First, \eqref{Equation primary decomposition of QH(B)} holds for $QH^*(B;\K)$, with $\Lambda[x]$ replaced by $\K[x]$ and $t$ replaced by $1$, since $\K$ is a field. This decomposition yields a decomposition of the unit $1=\sum g_j+ \sum h_{p+j} \in QH^*(B;\K)$ in terms of polynomials $g_j,h_{p+j}\in \K[x]$ (corresponding to the units in the various summands) which are annihilated by $(x^{\lambda_B}-\mu_j^{\lambda_B})^{d_j}$ and $x^{d_{p+j}}$ respectively (and are not annihilated by any non-zero polynomials of lower degree). Now reinsert positive powers of $t$ so as to make everything of homogeneous degree (recall $QH^*$ is $\Z$-graded when working with a graded Novikov variable $t$), to obtain  $t^{\textrm{positive}}=\sum g_j+ \sum h_{p+j} \in QH^*(B;\K[t])$ where $g_j,h_{p+j}\in \K[t][x]$ are annihilated by $(x^{\lambda_B}-\mu_j^{\lambda_B}t)^{d_j}$ and $x^{d_{p+j}}$ respectively (but not by lower degree polynomials). Over $\K[t,t^{-1}][x]$, we can rescale by $t^{-\textrm{positive}}$ to obtain a decomposition of $1$, and then Step 2 follows. $\checkmark$

\textbf{Step 3}:
Observe that for $\xi^{\lambda_B}=1$ and $\mu\in \K$, the image of a $\lambda_B$-family of factors
$$ 
(x-\mu T_B)(x-\xi \mu  T_B)\cdots (x-\xi^{\lambda_B-1} \mu T_B)
=
x^{\lambda_B}-\mu ^{\lambda_B}t_B
$$
via the map $\varphi$ of Theorem \ref{Theorem Minimal poly and char poly from B to E} is, using $\sum n_i x_i =c_1(E)= -k[\omega_B] = -kx$,
$$
x^{\lambda_B}-\mu^{\lambda_B}t_E (-kx)^k
=
x^k (x^{\lambda_B-k}-(-k)^k \mu^{\lambda_B}t_E).
$$
\indent
\textbf{Step 4}: We claim that there is an isomorphism of $\K[t,t^{-1}][x]$-modules,
\begin{equation}\label{Equation SH is free}
QH^*(E;\K[t,t^{-1}])\cong SH^*(E;\K[t,t^{-1}])\oplus \ker(x^{\mathrm{large}}),
\end{equation}
in particular $SH^*(E;\K[t,t^{-1}])$ is a free $\K[t,t^{-1}]$-module, since $QH^*(E;\K[t,t^{-1}])$ is free (see the proof of Theorem \ref{Theorem Minimal poly and char poly from B to E}).

Over the PID $\Lambda[x]$, the above would follow by Theorem \ref{Theorem SH of negative line bundle}. To obtain the splitting over $\K[t,t^{-1}][x]$, it suffices to prove that there are polynomials $a(x,t),b(x,t)\in \K[t][x]$ satisfying the equality $a(x,t)x^{\mathrm{large}} +b(x,t)\overline{f}=t^{\textrm{positive}}$ (since rescaling by $t^{-\textrm{positive}}$ then decomposes $1$), where $f=x^{\mathrm{large}}\overline{f}$ is the minimal polynomial of $[\omega_E]$, and where $\overline{f}$ means we remove all $x$-factors from $f$. First work over $\K[x]$, then B\'{e}zout's Lemma would yield that equality for some polynomials $a(x,1),b(x,1)\in \K[x]$ if we put $t=1$ in $\overline{f}$. Now reinsert positive powers of $t$ to make those polynomials of homogeneous degree, to obtain the required $a(x,t),b(x,t)$. Step 4 follows. $\checkmark$

\textbf{Step 5}: Applying $\varphi$ to the decomposition $1=\sum g_j + \sum h_{p+j}$ (over $\K[t,t^{-1}][x]$) we obtain a decomposition of $1=\varphi(1)=\sum \varphi(g_j)$ in $SH^*(E;\K[t,t^{-1}])$. Here, since $x$ acts invertibly on $SH^*$, $\varphi(h_{p+j})=0$ and, using Step 3, $(x^{\lambda_B-k}-(-k)^k \mu_j^{\lambda_B}t_E)^{d_j}$ annihilates $\varphi(g_j)$. Conversely, if a polynomial $f(x)\in \K[t,t^{-1}][x]$ annihilates $\varphi(g_j)$, then $x^{\mathrm{large}}f(x)\in \K[t_E][x]$ annihilates $\varphi(g_j)$. So by Step 1, $x^{\mathrm{larger}}f(x)\in \K[t_B][x]$ annihilates $g_j$ in $QH^*(B;\K[t_B])$, and hence it must be divisible by $(x^{\lambda_B}-\mu_j^{\lambda_B}t_B)^{d_j}$. Applying $\varphi$ shows that $f(x)$ must be divisible by $(x^{\lambda_B-k}-(-k)^k \mu_j^{\lambda_B}t_E)^{d_j}$, so the latter is the minimal polynomial annihilating $\varphi(g_j)$.

\textbf{Step 6}: Each $\varphi(g_j)$ generates a submodule $C_i$ (the span of $\varphi(g_j),x\varphi(g_j),x^2\varphi(g_j),\ldots$) in $SH^*(E;\K[t,t^{-1}])$. We claim $\sum C_i$ is direct (note that the Theorem then follows).

\textbf{Proof of Step 6.} 
Suppose $\sum p_i(x)=0$ in $SH^*(E;\K[t,t^{-1}])$, where $p_i(x)\in C_i$. Then $\sum x^{\mathrm{large}}p_i(x) =0$ in $QH^*(E;\K[t])$. By Step 1, $\sum x^{\mathrm{large}}p_i(x) =0$ in $QH^*(B;\K[t])$. But in $QH^*(B;\K[t,t^{-1}])$ the summands generated by $g_j,xg_j,x^2g_j,\ldots$, as $j$ varies, are direct by construction. So $x^{\mathrm{large}}p_i(x)=0$ in $QH^*(B;\K[t,t^{-1}])$ for each $i$. Thus $x^{\mathrm{large}}p_i(x)=0$ in $SH^*(E;\K[t,t^{-1}])$ and so $p_i(x)=0$ since $x$ acts invertibly in $SH^*$. Step 6 thus follows. $\checkmark$
\end{proof}
%
\subsection{The Calabi-Yau case and the NEF case}
\label{Subsection The Calabi-Yau case}
We call the condition $c_1(TM)(\pi_2(M))=0$ from Section \ref{Subsection Invertibles in the symplectic cohomology} the \emph{Calabi-Yau case}, which ensures that the representation in Section \ref{Subsection Invertibles in the symplectic cohomology} is defined. For toric $M$, $\pi_1(M)=1$, so this condition is equivalent to $c_1(TM)=0$. More generally, one can work with \emph{NEF} toric $M$, meaning $c_1(TM)(A)\geq 0$ for all nontrivial spheres $A\in \pi_2(M)$ which have a $J$-holomorphic representative (the Fano case corresponds to requiring a strict inequality: $c_1(TM)(A)>0$). For closed toric manifolds, the NEF case is studied in McDuff-Tolman \cite[Example 5.4]{McDuffTolman}. The key observation is:
\begin{lemma}\label{Lemma NEF case}
 If we replace the monotonicity assumption by NEF, then Lemma \ref{Lemma R of gi standard lift} becomes:
$$
\begin{array}{rcl}
r_{g_i^{\wedge}}(1) &=& \mathrm{PD}[D_i] + (\textrm{higher order }t)\in QH^2(E)
\\ 
\mathcal{R}_{g_i^{\wedge}}(1) &=& c^*(r_{g_i^{\wedge}}(1))= c^*\mathrm{PD}[D_i]+(\textrm{higher order }t)\in SH^2(E)^{\times}.
\end{array}
$$
\end{lemma}
\begin{proof}
 This follows from Lemma \ref{Lemma rg1}.
\end{proof}

By Lemmas \ref{Lemma NEF case}, \ref{Lemma Choices of lifts} and \ref{Lemma McDuffTolman Algebra trick}, it follows that in the Batyrev presentation for (possibly non-compact) NEF toric varieties $M$ we must replace the quantum SR-relations \eqref{EqnQSRrelation} with
$$
\mathcal{R}_{i_1}  \cdots  \mathcal{R}_{i_a}=s^{\omega(\beta_I)}T^{c_1(TM)(\beta_I)}\mathcal{R}_{j_1}^{c_{1}} \cdots \mathcal{R}_{j_b}^{c_{b}}\qquad \textrm{ where }\mathcal{R}_{j} = \mathcal{R}_{g_{j}^{\wedge}}(1),
$$
and we work over a modified Novikov ring $\mathfrak{R}$ (see Section \ref{Subsection Invertibles in the symplectic cohomology} and \ref{Subsection New Novikov ring R}). The lowest order $T$ terms on each side of the equation agree with those in \eqref{EqnQSRrelation}, but the higher order $T$ terms are hard to compute in practice.

For negative line bundles $E\to B$ satisfying weak+ monotonicity (Section \ref{Subsection Invertibles in the symplectic cohomology}), $r_{g_f^{\wedge}}(1)=\pi^*c_1(E)$ may not in fact hold, where $g_f$ is the natural rotation in the fibre. In \cite{Ritter4} we proved that $r_{g_f^{\wedge}}(1)=(1+\lambda_+)\pi^*c_1(E)$, where $\lambda_+\in \Lambda$ involves only strictly positive powers of $t$ (and $s^0$-terms). This rescaling does not affect $SH^*(E)\cong QH^*(E)/\ker \pi^*c_1(E)^{\textrm{large}}$ (Theorem \ref{Theorem SH of negative line bundle}). 

The Calabi-Yau case is $k=\lambda_B$. By \cite{Ritter4}, $SH^*(E)$ is $\Z$-graded, finite-dimensional, and has an automorphism $\mathcal{R}_{g_f^{\wedge}}$ of degree $2$. It follows that $SH^*(E)=0$ and $\pi^*c_1(E)\in QH^*(E)$ is nilpotent. So the wrapped Fukaya category $\mathcal{W}(E)$ (assuming it is defined) would not be interesting since it would be homologically trivial (being a module over $SH^*(E)$).
%
\section{Twisted theory: non-monotone toric symplectic forms}
\label{Section Twisted theory: non-monotone toric symplectic forms}
\subsection{Toric symplectic forms}
\label{Subsection Toric symplectic forms}

For this background section, we refer for details to Ostrover-Tyomkin \cite{Ostrover-Tyomkin}, Fulton \cite[Sec.3.4]{Fulton}, Batyrev \cite{Batyrev} or Cox-Katz \cite[Sec.3.3]{Cox-Katz}. Some of the terminology is also illustrated in Appendix A (Section \ref{Section The moment polytope of a toric negative line bundle}).

Recall that from a \emph{fan} $\Sigma\in \Z^n$ one can construct a toric variety $X=X_{\Sigma}$, in particular this determines a complex structure $J$ on $X_{\Sigma}$. We always assume that $X_{\Sigma}$ is smooth.

Recall that a \emph{piecewise linear function} $F$ on $\Sigma$ means a real-valued function which is linear on each cone $\sigma$ of $\Sigma$, thus $F(v)=\langle u_{\sigma}, v\rangle$ for some $u_{\sigma}\in \R^n$, whenever $v\in \sigma$. So $F$ on $\sigma$ is determined by linearity by the values $F(e_i)=\langle u_{\sigma}, e_i\rangle$ for those edges $e_i$ of $\Sigma$ which lie in $\sigma$. $F$ is \emph{strictly convex} if $\langle u_{\sigma}, v\rangle > F(v)$ for $v\notin \sigma$, equivalently $\langle u_{\sigma}, e_i\rangle >F(e_i)$ for $e_i\notin \sigma$.

Choosing a piecewise linear strictly convex function $F$ on $\Sigma$ is equivalent to choosing a K\"ahler form $\omega_F$, satisfying
$$
[\omega_F] = \sum -F(e_i) \, \mathrm{PD}[D_i],
$$
where $D_i\subset X_{\Sigma}$ are the toric divisors. In particular, $(X_{\Sigma},\omega_F)$ is then a toric manifold (i.e. the torus action is Hamiltonian) and the symplectic form $\omega_F$ on $X_{\Sigma}$ is $J$-compatible. We call these the \emph{toric symplectic forms}.
The moment polytope of $(X_{\Sigma},\omega_F)$ is
$$
\Delta_F = \{ y\in \R^n: \langle y,e_i \rangle \geq \lambda_i \}
$$
where $e_i$ are the edges of the fan (which are inward normals to the facets of $\Delta_F$) and where
$$\lambda_i=F(e_i).$$

Piecewise linear strictly convex functions $F$, for which all $F(e_i)$ are integers, correspond to ample divisors: the divisor is $D_F=\sum -F(e_i)D_i$. The canonical divisor is $K=-\sum D_i$.

We always assume that $X_{\Sigma}$ is Fano, that is the anti-canonical bundle $\Lambda_{\C}^{\mathrm{top}}TX_{\Sigma}$ is ample. Thus the ample anti-canonical divisor $-K=\sum D_i$, corresponding to $c_1(TX)= \sum \mathrm{PD}[D_i]$, corresponds to the piecewise linear function $F(e_i)=-1$ for all $e_i$. This corresponds to a symplectic form $\omega_{\Delta}$, satisfying $[\omega_{\Delta}]=c_1(TX)\in H^2(X)$. This is the unique symplectic form for which the corresponding moment polytope $\Delta$ is \emph{reflexive} (see Section \ref{Subsection The polytope of a Fano variety}).

We always denote by $\omega_X$ the monotone integral K\"ahler form obtained from rescaling $\omega_{\Delta}$, so that $\omega_X\in H^2(X,\Z)$ is a primitive class. Thus
$$
c_1(TX)=[\omega_{\Delta}]=\lambda_X [\omega_X],
$$
where the positive integer $\lambda_X$ is called the \emph{index} of the Fano variety, and it is also the monotonicity constant for the monotone symplectic manifold $(X_{\Sigma},\omega_X)$.

\subsection{The Novikov ring $\mathbf{\mathfrak{R}}$ over $\mathbf{\Lambda_s}$}
\label{Subsection New Novikov ring R}
In the non-monotone case, we need to change the Novikov ring (Section \ref{Subsection Novikov ring}). Since the values of $c_1(TX)$ and $\omega_F$ on spheres may no longer be proportional, we use two formal parameters $s,t$ instead of one. In the definition of the quantum product, defining $QH^*(X,\omega_F)$, we use weights $$s^{\omega_F(A)}T^{c_1(TX)(A)}=s^{\omega_F(A)}t^{\omega_X(A)}$$ when counting spheres $A$ in $X$ (recall $t=T^{\lambda_X}$). The Novikov ring is now defined as 
$$\boxed{\mathfrak{R}=\Lambda_s(\!(T)\!) = \Lambda_s[T^{-1},T]\!]}$$
that is Laurent series in the formal variable $T$ with coefficients in the Novikov field
$\Lambda_s$,
$$
\Lambda_s = \left\{ \sum_{i=0}^{\infty} a_i s^{n_i}: a_i\in \K, n_i\in \R, \lim n_i = \infty\right\},
$$
where $s$ is a new formal parameter lying in grading $|s|=0$, and $|T|=2$ (since $t=T^{\lambda_X}$ has $|t|=2\lambda_X$). When $F$ is integer-valued, one could further restrict $n_i$ to lie in $\Z$.

As usual, $t$ ensures that $QH^*$ is $\Z$-graded. It is not necessary to complete in $t$ (i.e. allowing series in positive powers of $t$). At the cost of losing the $\Z$-grading, one could omit $t$ altogether. We complete in $t$ because we want $\mathfrak{R}$ to be a field (this is important in the generation results of Section \ref{Subsection Generation for 1-dimensional eigensummands}, although Remark \ref{Remark Novikov ring without completing mathfrak R} shows a work-around). If one wanted $\mathfrak{R}$ to be algebraically closed, it suffices to allow arbitrary real powers of $t$ with the growth condition as in Section \ref{Subsection Novikov ring} (alternatively, one can take the field of Puiseaux series in $t$ over $\Lambda_s$). This will not be necessary for us: we only need to factorize the characteristic polynomial of $c_1(TX)$, and to factorize this into linear factors it is enough to 
have the root $T=t^{1/\lambda_X}$.

\emph{\textbf{Technical Remark.} As in Section \ref{Subsection The eigenvalues of c_1(TX)}, when discussing eigenvalues of $c_1(TX)$ we want $\Lambda_s$ to be algebraically closed. As in \ref{Subsection The eigenvalues of c_1(TX)}, the presence or absence of $T$ is not an issue: the $\Z$-grading imposes how to insert powers of $T$ into a factorization of the characteristic polynomial of $c_1(TX)$ over $\Lambda_s$. We recall that if $\K$ is an algebraically closed field of characteristic zero, then $\Lambda_s$ is algebraically closed \cite[Lemma A.1]{FOOOtoric}.
}

\subsection{Using a non-monotone toric form is the same as twisting}
\label{Subsection Using a non-monotone toric form is the same as twisting}
We emphasize that the moduli spaces of spheres that we count has not changed, since we are using the same complex structure $J$ (which is compatible with both K\"ahler forms $\omega_F$ and the monotone $\omega_X$). In particular, there are no compactness issues in the definition of the quantum or Floer cohomologies for $\omega_F$ because $J$ and small perturbations of $J$ are also tamed by the monotone form $\omega_X$ (and we can use $c_1(TX)(A)=\lambda_X\omega_X(A)$ to control indices of solutions).
%
%
%
%
%
Observe that we are only changing the $s$-weights with which we count the solutions: for $\omega_F$ we use weight $s^{\omega_F(A)}T^{c_1(TX)(A)}$ whereas for the monotone form $\omega_X$ we would just use $T^{c_1(TX)(A)}$. By convention we will omit the irrelevant factor $s^{\omega_X(A)}=s^{\lambda_X c_1(TX)(A)}$ in the monotone case, since this can be recovered by formally replacing $T$ with $s^{\lambda_X}T$. This convention ensures, as we will now explain, that the theory for $\omega_F$ is just the $\omega_F$-twisted theory for $\omega_X$.

We briefly recall from \cite{Ritter1, Ritter3} the definition of the \emph{twisted quantum cohomology} and the \emph{twisted Floer cohomology} for the form $\omega_X$ twisted by the $2$-form $\omega_F$. One introduces a system of local coefficients $\underline{\mathfrak{R}}_{\omega_F}$, meaning that one counts Morse/Floer/GW-solutions $u$ with an extra weight factor $s^{[\omega_F](u)}$. Here $[\omega_F](u)$ can be defined either by evaluating the cocycle $[\omega_F]$ on the chain $u$, or by integrating $u^*\omega_F$ over the domain of $u$. Since $\omega_F$ is a $2$-cocycle, this weight is trivial ($s^0$) when $u$ is a Morse trajectory, thus the twisting does not affect the vector space $QH^*(X_{\Sigma},\omega_X; \underline{\mathfrak{R}}_{\omega_F})=H^*(X_{\Sigma})\otimes \mathfrak{R}$ but it does affect the quantum product by inserting the weights $s^{\omega_F(A)}$ when $u$ is a sphere in the class $A\in H_2(X_{\Sigma})$.

\begin{corollary}\label{Corollary QH twisted is same as non-exact QH} The quantum cohomology for $\omega_F$ can be identified with the $\omega_F$-twisted quantum cohomology for $\omega_X$:
$QH^*(X_{\Sigma},\omega_F)\cong QH^*(X_{\Sigma},\omega_X; \underline{\mathfrak{R}}_{\omega_F}).$
\end{corollary}

The construction of the twisted Floer cohomology and twisted symplectic cohomology, and the proof that these have a product structure when twisting by closed two-forms $\eta$ is carried out in \cite{Ritter3}. The discussion in Section \ref{Subsection Invertibles in the symplectic cohomology} and \ref{Subsection Invertibles in the symplectic cohomology for a larger class of Hamiltonians} on the construction of invertibles can be carried out in the twisted case, simply by working over $\mathfrak{R}$ and inserting weights $s^{\eta[u]}$.

\emph{\textbf{Technical Remark:} In \cite{Ritter4}, for weak+ monotone manifolds $M$, we worked over a very large Novikov ring generated by $\pi_2(M)$ (modulo those classes on which $\omega_M$ and $c_1(TM)$ both vanished). One can just as well only keep track of the $\omega_M$ and $c_1(TM)$ values on spheres by using weights $s,T$ as described above (using only $T$ in the monotone case, by convention).}

In particular, in the GW-section counting in \cite{Ritter4} for the bundles $E_g \to S^2$ with fibre $M$, we use weights $s^{\eta(A)}$ where $A$ is the relevant $\pi_2(M)$ class associated with the section (explicitly, in the notation of the proof of Lemma \ref{Lemma rg1}, this would be the weight $s^{\eta(\gamma)}$). 

\begin{corollary}\label{Corollary Twisted symplectic cohomology for toric forms} The symplectic cohomology for $\omega_F$ can be identified with the $\omega_F$-twisted symplectic cohomology for $\omega_X$:
$SH^*(X_{\Sigma},\omega_F)\cong SH^*(X_{\Sigma},\omega_X; \underline{\mathfrak{R}}_{\omega_F}).$
\end{corollary}
%
\subsection{Rotations and the lifting problem}
\label{Subsection Rotations and the lifting problem}
It follows that the Seidel representation for closed Fano toric varieties, and the representation defined in Section \ref{Subsection Invertibles in the symplectic cohomology for a larger class of Hamiltonians} for non-compact Fano toric varieties, are defined for $\omega_F$ and coincide with the $\omega_F$-twisted representations obtained for $\omega_X$. In particular, Lemma \ref{Lemma R of gi standard lift} still holds:
$$
\mathcal{R}(g_i^{\wedge}) = x_i,
$$
since the constant sections lie in the class $\gamma=0$, so no $s$-weight appears.

The lifting problem that already occurred in the monotone case (Section \ref{Subsection The problem with relating the lifted rotations}) is tricky in the non-monotone case, as Lemma \ref{Lemma R of gi standard lift} is no longer available. The following resolves this issue.

\begin{lemma}\label{Lemma Choices of lifts}
Let $(X_{\Sigma},\omega_F)$ be a Fano toric manifold. Any relation $\prod g_i^{a_i} = \mathrm{id}$ for $a_i\in \Z$ (corresponding to a relation $\sum a_i e_i=0$ amongst the edges of $\Sigma$) yields the relation:
$$
\prod (g_i^{\wedge})^{a_i} = s^{-\sum F(e_i) a_i}\, T^{\sum a_i}\mathrm{id}.
$$
\end{lemma}
\begin{proof}
By construction, $\psi=\prod (g_i^{\wedge})^{a_i}$ is a lift of the identity map, and so differs from the identity by an element of the deck group $\Gamma=\pi_2(M)/\pi_2(M)_0$ (see Section \ref{Subsection Invertibles in the symplectic cohomology}). This deck group can be identified with the monomials $s^{n}T^m$. Just as in the monotone case, $m=c_1(TX)[\beta_a]=\sum a_i$, where $\beta_a\in H_2(X)$ is the homology class corresponding to the relation $\sum a_i e_i=0$ (see the proof of Theorem \ref{Theorem change of Novikov param}). Recall from Section \ref{Subsection Review of the Batyrev-Givental presentation of QH} that $\beta_a$ satisfies the intersection products $\beta_a\cdot D_i = a_i$. We claim that the power of $s^n$ above is
$$
n = \omega_F(\beta_a) = \sum -F(e_i)\mathrm{PD}(D_i)(\beta_a) = -\sum F(e_i)a_i.
$$
To prove this, it suffices to show that the image of the constant disc $(c_x,x)$ at a point $x\in X_{\Sigma}\setminus \cup D_i$ under the action of $\psi$ is a sphere representing the class $\beta_a$.
Since $\psi$ is a lift of the identity map $\prod g_i^{a_i}$, we know that $\psi$ maps the constant loop $x$ to itself, therefore it maps the constant disc $c_x$ to a new disc bounding the constant $x$, so this new disc is in fact a sphere. To determine that it represents $\beta_a$ it now remains to check that the sphere intersects $a_i$ times the divisor $D_i$.

We may assume $x$ has homogeneous coordinates $x_i=1$. The image under $g_1$ of the path $(r,1,\ldots,1)_{1\geq r \geq 0}$ from $x=(1,\ldots,1)$ to the point $p_1=(0,1,\ldots,1)\in D_1$ is $\gamma_r(t)=(re^{2\pi i t},1,\ldots,1)$. The loop $\gamma_r(t)$ can be filled by the obvious disc $\Gamma_r(s,t)=(\gamma_{s}(t))_{r\geq s \geq 0}$. As $r$ varies, this gives a family of discs, ending at the constant disc $(c_{p_1},p_1)$ at $p_1\in D_1$. Since $g_1^{\wedge}(c_{p_1})=c_{p_1}$ by definition of the lift $g_1^{\wedge}$, it follows that $g_1^{\wedge}(c_x)=\Gamma_1(s,t)$. Similarly, $(g_1^{\wedge})^{a_1}(c_x)$ is represented by the obvious disc $(se^{2\pi i a_1 t},1,\ldots,1)$ where $0\leq s,t\leq 1$.

Inductively, we claim that $(g_k^{\wedge})^{a_k}\cdots (g_1^{\wedge})^{a_1}\cdot (c_x,x)$ is represented by the disc 
\begin{equation}\label{Equation c(s,t) disc}
c(s,t)=(se^{2\pi i a_1 t},\ldots,se^{2\pi i a_k t},1,\ldots,1) \qquad (\textrm{where }0\leq s,t\leq 1).
\end{equation}
For the inductive step, we use the following general trick. Since $\pi_1(X)=1$, any loop $y=y(t)$ gives rise to a disc $\overline{y}=\overline{y}(s,t)$ with $\overline{y}(1,t)=y(t)$ and $\overline{y}(0,t)$ equal to a chosen basepoint. Given a filling disc $(c,y)$ for $y$, we obtain a sphere $c\# -\overline{y}$ representing some class $\lambda = s^n T^m\in \Gamma$. Thus $(c,y)=\lambda\cdot (\overline{y},y)$ so $g_i^{\wedge}(c,y)=\lambda g_i^{\wedge}(\overline{y},y)$. If the basepoint is chosen to lie in $D_i$, then $g_i^{\wedge}(\overline{y})$ is the disc $(g_i)_t \overline{y}(s,t)$ -- this follows by the same argument as in the previous paragraph, by considering the path $\overline{y}(r,t)$ of loops and the path of discs $(\overline{y}(s,t))_{r\geq s\geq 0}$.

Now, the inductive step. Apply the observation to $(c,y)$ for $c(s,t)$ as in \eqref{Equation c(s,t) disc}. The homotopy $\overline{y}$ from $y(t)=(e^{2\pi i a_1 t},\ldots,e^{2\pi i a_k t},1,1,\ldots,1)$ to $(1,\ldots,1,0,1,\ldots,1)\in D_{k+1}$ first homotopes the first coordinate to $1$, then homotopes the second coordinate to $1$, and so forth, and finally it follows the path from $x_{k+1}=1$ to $x_{k+1}=0$ in that coordinate. By construction, the class of the sphere $c\# -\overline{y}$ is trivial in $\Gamma$ since we produced a path of discs from $c$ to $\overline{y}$. Applying $(g_{k+1})^{a_{k+1}}$ to the path of loops $\overline{y}(s,\cdot)$ corresponds to acting by $e^{2\pi i a_{k+1} t}$ on the $x_{k+1}$ coordinate, and it now easily follows that the class of this image disc is the same as the class of 
$(se^{2\pi i a_1 t},\ldots,se^{2\pi i a_{k+1}  t},1,\ldots,1)$, as required for the inductive step.

By induction, it follows that $\prod (g_i^{\wedge})^{a_i}\cdot (c_x,x)$ is represented by the disc $(se^{2\pi i a_1 t},\ldots,se^{2\pi i a_{r} t})$ which intersects $D_i=(x_i=0)$ precisely $a_i$ times, as required.
\end{proof}
\subsection{Presentation of $\mathbf{QH^*(B,\omega_F^B),QH^*(E,\omega_F^E),SH^*(E,\omega_F^E)}$ in the Fano case.}
\label{Subsection Presentation of QHB, QHE, SHE}
By Lemma \ref{Lemma Choices of lifts}, the presentation of $QH^*(B)$ for closed monotone $X_{\Sigma}=B$ (Section \ref{Subsection Review of the McDuff-Tolman proof of the presentation of QH}) and that of $QH^*(E)$, $SH^*(E)$ for monotone toric negative line bundles $X_{\Sigma}=E$ (Theorem \ref{Theorem presentation of QH and SH for Fano toric NLB}) holds for the non-monotone form $\omega_F$ by twisting coefficients. In the closed case, this was proposed by Batyrev \cite[Sec.5]{Batyrev}, and proved by Givental \cite{Givental,Givental2}, Cieliebak-Salamon \cite{Cieliebak-Salamon} and McDuff-Tolman \cite[Sec.5]{McDuffTolman}.

\begin{corollary} \label{Corollary Presentation of QH and SH in twisted case}
In the notation of Theorem \ref{Theorem presentation of QH and SH for Fano toric NLB}, but now using a (non-monotone) toric form $\omega_F^B$ on $B$ and working over $\mathfrak{R}$, abbreviating $\lambda_i=F(e_i)$,
$$
\boxed{
\begin{array}{l}
QH^*(B,\omega_F^B) \cong \mathfrak{R} [x_1,\ldots,x_r] / 
\left(\begin{smallmatrix} \textrm{linear rel'ns in }B, \textrm{ twisted SR-rel'ns:}\\
\prod x_{i_p} = s^{-\sum \lambda_{i_p} + \sum c_q \lambda_{j_q}}T^{|I^B|-\sum c_q} \cdot \prod x_{j_q}^{c_q}\end{smallmatrix}\right)
\\
QH^*(E,\omega_F^E) \cong \mathfrak{R} [x_1,\ldots,x_r] / \left(\begin{smallmatrix} \textrm{linear rel'ns in }B,\, \textrm{ twisted SR-rel'ns after Novikov parameter change:}\\
\prod x_{i_p} = s^{-\sum \lambda_{i_p} + \sum c_q \lambda_{j_q}} T^{|I^B|-\sum c_q - c_f} \cdot (\sum n_i x_i)^{c_f} \cdot \prod x_{j_q}^{c_q}
\end{smallmatrix}\right)
\\
SH^*(E,\omega_F^E) \cong \mathfrak{R} [x_1,\ldots,x_r,z] / (z\cdot {\textstyle\sum} n_i x_i - 1,\textrm{and the same rel'ns as for }QH^*(E,\omega_F^E))
\end{array}
}
$$
where, by Lemma \ref{Lemma moment polytope for neg l bdle},  $F^E(e_i)=\lambda_i^E=\lambda_i$ and $F^E(e_f)=\lambda_{f}^E=0$ (so $s^{c_f \lambda_f}$ does not appear). In particular, the form
$\omega_F^E= \pi^*\omega_F^B + \pi\Omega$ on $E$ arises as $\omega$ in Section \ref{Subsection Negative line bundles} from $\omega_F^B\in H^2(B)$ in place of $\omega_B$, and corresponds to the piecewise linear function $F^E$ on the fan for $E$.

Analogues of Theorems \ref{Theorem Minimal poly and char poly from B to E}-\ref{Theorem Eigenvalues of c1 from B to E} and Lemma \ref{Lemma Critical points of W come in families} hold. So there is a ring homomorphism $$\varphi: QH^*(B,\omega_F^B;\Lambda_s[t,t^{-1}]) \to SH^*(E,\omega_F^E;\Lambda_s[t,t^{-1}])$$ with $\varphi(x_i)= x_i$, $\varphi(t_B)=t_E  (\textstyle\sum n_i x_i)^k$ and $\varphi(s_B)=s_E$. Over $\Lambda_s[t]$, this factorizes as $$\varphi: QH^*(B,\omega_F^B;\Lambda_s[t])\to QH^*(E,\omega_F^E;\Lambda_s[t])\to SH^*(E,\omega_F^E;\Lambda_s[t]).$$ These can be identified with the ring homomorphisms obtained for the monotone form $\omega_B$ but twisting coefficients using $\omega_F^B$, respectively $\omega_F^E$.
\end{corollary}
\subsection{The $F$-twisted superpotential $W_F$ and $\mathbf{\mathrm{Jac}(W_F)}$}
\label{Subsection F-twisted superpotential and Jac WF}
The \emph{$F$-twisted superpotential} $W_F$ of $X_{\Sigma}$ is the superpotential associated to $(X_{\Sigma},\omega_F)$, 
$$
W_F: (\mathfrak{R}\setminus \{0\})^n \to \mathfrak{R}, \quad W_F(z) = \sum s^{-F(e_i)} T z^{e_i}
$$
In the case $F(e_i)=-1$, that is using the monotone $\omega_{\Delta}$, this becomes $\sum s T z^{e_i}$, so we can ignore $s$ and we obtain the untwisted superpotential for $\omega_{\Delta}$.

For closed toric $(X_{\Sigma},\omega_F)$, it follows from the presentation of $QH^*$ and Section \ref{Subsection Batyrev's argument: from the presentation of QH to JacW} that
$$
QH^*(X_{\Sigma},\omega_F)\cong \mathrm{Jac}(W_F),  \quad \mathrm{PD}[D_i] \mapsto s^{-F(e_i)} T z^{e_i}, \quad c_1(TX)\mapsto W_F,
$$
where $\mathrm{Jac}(W_F) \equiv \mathfrak{R}[z_1^{\pm 1},\ldots,z_n^{\pm 1}]/(\partial_{z_1}W_F,\ldots,\partial_{z_n}W_F)$ (where $n=\dim_{\C} X$).
\begin{theorem}[Analogue of Theorem \ref{Theorem Jacobian ring is SH}]\label{Theorem twisted Jacobian ring is SH}
$$
\begin{array}{rcl}
SH^*(E,\omega_F^E)\cong QH^*(E,\omega_F^E)[c]/(c\cdot \pi^*c_1(E)-1) &\to & \mathrm{Jac}(W_F^E)\\
\mathrm{PD}[D_i] &\mapsto & s^{-\lambda_i^E} T z^{e_i}\\
\pi^*c_1(E)=\mathrm{PD}[B] &\mapsto & T z_{n+1}.
\end{array}
$$
In particular, $c_1(TE)$ maps to $W_F^E$. Moreover, $s^{-\lambda_i^E} T z^{e_i}=(sz_{n+1}^k)^{-\lambda_i^B} T z^{(b_i,0)}$.
\end{theorem}

As in Section \ref{Subsection The Jacobian ring for nlb}, $\varphi:QH^*(B,\omega_F^B;\Lambda_s[t,t^{-1}])\to SH^*(E,\omega_F^E;\Lambda_s[t,t^{-1}])$ corresponds to $$\varphi:\mathrm{Jac}(W_F^B;\Lambda_s[t,t^{-1}])\to \mathrm{Jac}(W_F^E);\Lambda_s[t,t^{-1}]), \; \varphi(z_i)=z_i,\, \varphi(t_B)=t_Ez_{n+1}^k,\, \varphi(s_B)=s_E.$$
%
\section{The Fukaya category and generation results}
\label{Section Fukaya category, and generation results}
\subsection{The Fukaya category for non-monotone toric forms versus the twisted Fukaya category}
\label{Subsection Fukaya category for non-monotone toric}

Recall by \cite{Ritter3} that one can twist the Lagrangian Floer complexes $CF^*(L,L')$ by a closed two-form $\eta\in H^2(M)$ which vanishes on $L,L'$. Namely, one introduces a system of local coefficients $\underline{\mathfrak{R}}_{\eta}$ and disc counts are weighted by $s^{\eta(u)}$. By Stokes' theorem, these weights are invariant under a homotopy of the disc $u$ as long as the boundary of $u$ moves within $L\cup L'$ (the integral of $\eta$ over the path within $L\cup L'$ will vanish since $\eta$ vanishes on $L\cup L'$). Similarly, one can twist $A_{\infty}$-operations using $\eta$, and thus define an $\eta$-twisted Fukaya category $\mathcal{F}(M;\underline{\mathfrak{R}}_{\eta})$, provided that one restricts to considering only Lagrangians $L$ on which $\eta$ vanishes. In order to drop that restriction on Lagrangians, more work is required: assuming $\K$ has characteristic zero,
%
%
 one can define the bulk deformation $\mathcal{F}(M;D)$ for a divisor $D\subset M$ by work of Fukaya-Oh-Ohta-Ono \cite{FOOO} (see \cite{FOOOtoric} in the toric setting), which corresponds to twisting by $\eta$ when $\eta$ is Poincar\'e dual to $D$. One can also allow $\R$-linear combinations of divisors $\sum -\lambda_i D_i$, in which case each positive intersection of a disc $u$ with $D_i$ gets counted with weight $s^{-\lambda_i}$.
These are the particularly simple bulk deformations in real codimension $2$ which do not involve counting new, higher-dimensional, moduli spaces -- they only involve introducing weights in the original counts.

We are concerned with twisting the monotone toric manifold $(X_{\Sigma},\omega_X)$ by $\omega_F$, and comparing the $\omega_F$-twisted Fukaya category $\mathcal{F}(X_{\Sigma},\omega_X; \underline{\mathfrak{R}}_{\omega_F})$ with the (untwisted) Fukaya category for $(X_{\Sigma},\omega_F)$. If we restrict the Fukaya category to $\mathcal{F}^{\mathrm{toric}}$, meaning we only allow toric Lagrangians $L$ (the $\textrm{codim}_{\C}=1$ complex torus orbits, together with holonomy data), then the objects $L$ are Lagrangian submanifolds both for $\omega_X$ and $\omega_F$, provided all $F(e_i)$ are close to $-1$. Indeed, following \cite[Sec.2.1-2.2]{Guillemin}, when $F(e_i)$ is close to $-1$, the moment polytopes $\Delta_F$ and $\Delta_X$ (the polytopes for $\omega_F,\omega_X$ respectively) will undergo a small variation that does not change the diffeomorphism type of the toric manifolds $X_{\Delta_F}$ and $X_{\Delta_X}$ built as symplectic reductions (via the Delzant construction). One can then show \cite[Thm 2.7]{Guillemin} that $X_{\Delta_F}$ is $T^n$-equivariantly symplectomorphic to $(X_{\Delta_X},\omega_F)$.
%
%
In particular, $\omega_F$ will then vanish on the toric Lagrangians $L$ in $(X_{\Sigma},\omega_X)$, so the twisted category $\mathcal{F}^{\mathrm{toric}}(X_{\Sigma},\omega_X; \underline{\mathfrak{R}}_{\omega_F})$ is well-defined.

\begin{corollary} \label{Corollary twisted toric category is non-monotone toric category}
When all $F(e_i)$ are close to $-1$, the twisted category
 $\mathcal{F}^{\mathrm{toric}}(X_{\Sigma},\omega_X; \underline{\mathfrak{R}}_{\omega_F})$ is naturally identifiable with $\mathcal{F}^{\mathrm{toric}}(X_{\Sigma},\omega_F)$. $\qedhere$
\end{corollary}

\begin{remark} Since $H_2(X)$ is generated by toric divisors, $\omega_F$ is Poincar\'e dual to an $\R$-linear combination $D_F$ of toric divisors, so the bulk-deformed category $\mathcal{F}(X_{\Sigma},\omega_X; D_F)$ is defined and naturally contains the full subcategory $\mathcal{F}^{\mathrm{toric}}(X_{\Sigma},\omega_X; \underline{\mathfrak{R}}_{\omega_F})$. However, it is less clear how to compare the bulk-deformed category with $\mathcal{F}(X_{\Sigma},\omega_F)$, since the (non-toric) Lagrangians have changed when passing from $\omega_X$ to $\omega_F$. However, for the purposes of proving that the generation criterion holds for toric Lagrangians, we anyway restrict to the subcategory of toric Lagrangians and check that the open-closed string map hits an invertible element. So for the purposes of generation, we only deal with the categories in Corollary \ref{Corollary twisted toric category is non-monotone toric category}. In particular, we by-pass the technical issue of defining the Fukaya category $\mathcal{F}(X_
{\Sigma},\omega_F)$ in the non-monotone setting and proving that the structural results of \cite{RitterSmith} (the $\mathcal{OC}$-map, the $QH^*$-module structures, the generation criterion) generalize from the case of monotone K\"ahler forms $\omega_X$ to the case of a (non-monotone) K\"ahler form $\omega_F$ close to $\omega_X$ (so both are compatible with the complex structure $J$). This is a technical issue in the sense that the Kuranishi machinery of \cite{FOOO} is very likely to succeed in this generalization, but we will not undertake this onerous task. We emphasize that, when restricting to toric Lagrangians, this is not an issue since the Lagrangians for $\omega_X,\omega_F$ are then in common and thus $\omega_X$ can be used to control $J$-holomorphic discs in the construction of $\mathcal{F}^{\mathrm{toric}}(X_{\Sigma},\omega_F)$.
\end{remark}

Finally, the constructions of Ritter-Smith \cite{RitterSmith} obtained for monotone symplectic manifolds apply immediately to the $\omega_F$-twisted setting and to the bulk-deformed setting when deforming by divisors. Again, this is because we never change the moduli spaces being counted, but only the weights in these counts. In particular, the open-closed string map into twisted $QH^*$ is defined, it is a module map over twisted $QH^*$, and the generation criterion holds.

\subsection{What it means for the toric generation criterion to hold}
\label{Subsection What it means for the toric generation criterion to hold}

Recall that the generation criterion of Abouzaid \cite{Abouzaid}, generalized to the monotone setting by Ritter-Smith \cite{RitterSmith}, states that a full subcategory $\mathcal{S}$ will split-generate the wrapped Fukaya category $\mathcal{W}(X)$ (where $X$ is a non-compact symplectic manifold conical at infinity, assuming exactness or monotonicity) if $1\in \mathcal{OC}(\mathrm{HH}_*(\mathcal{S}))$, where $\mathcal{OC}$ is the \emph{open-closed string map}
$$
\mathcal{OC}: \mathrm{HH}_*(\mathcal{W}(X))\to SH^*(X).
$$
By Ritter-Smith \cite{RitterSmith}, this is an $SH^*$-module map so it is in fact enough to hit an invertible element. The discussion, here and below, holds analogously for the compact Fukaya category $\mathcal{F}(X)$, and it also holds for closed monotone symplectic manifolds $X$, in which cases one considers the $QH^*$-module homomorphism
$$
\mathcal{OC}: \mathrm{HH}_*(\mathcal{F}(X))\to QH^*(X).
$$

A closer inspection of the argument, shows that in fact a weaker condition is required, namely that for any object $K$ of $\mathcal{W}(X)$ which one hopes to split-generate, one wants
$$
\xymatrix{ 
\mathrm{HH}_*(\mathcal{S})\subset \mathrm{HH}_*(\mathcal{W}(X)) \ar@{->}^-{\mathcal{OC}}[r] &
SH^*(X) \ar@{->}^-{\mathcal{CO}}[r] &
HW^*(K,K)
}
$$
to hit the unit in the wrapped Lagrangian Floer cohomology $HW^*(K,K)$. Since the \emph{closed-open string map} $\mathcal{CO}$ is a ring homomorphism, this is automatic if $\mathcal{OC}$ hits $1$ since $\mathcal{CO}(1)=1$.

In the monotone case we must split the category into summands $\mathcal{W}_{\lambda}(X)$ (respectively $\mathcal{F}_{\lambda}(X)$) indexed by the $m_0$-value of the objects, so $\mathfrak{m}_0(L) = m_0(L)\, [L]=\lambda\, [L]$. Once we restrict $\mathcal{OC}$ to this summand, it becomes very unlikely that it will hit the unit. We now discuss the implications and the remedy to this issue.

\begin{lemma}\label{Lemma CO map vanishes on things}
 If $K$ does not intersect a representative of $\mathrm{PD}(c_1(TX))$, then
 $\mathcal{CO}:SH^*(X)\to HW^*(K,K)$ $($and respectively $\mathcal{CO}:QH^*(X)\to HF^*(K,K))$ vanishes on the following:
\begin{enumerate}
 \item the generalized eigensummands $SH^*(X)_{\lambda}$ $($respectively $QH^*(X)_{\lambda})$ for $\lambda\neq m_0(K)$;
 
 \item the ideal $(c_1(TX)-m_0(K))*SH^*(X)$ $($respectively $(c_1(TX)-m_0(K))*QH^*(X))$. So $\mathcal{CO}$ vanishes on eigenvectors of $c_1(TX)$ arising in Jordan blocks of size at least $2$.
\end{enumerate}
\end{lemma}
\begin{proof}
 Multiplication by $c_1(TX)-m_0(K)$ is invertible on a generalized eigensummand $G$ for eigenvalues $\lambda\neq m_0(K)$. So, $(c_1(TX)-m_0(K))G = G$. But $\mathcal{CO}$ is a ring map, so on any multiple of $c_1(TX)-m_0$ it will vanish:
$$
\mathcal{CO}((c_1(TX)-m_0(K))y) = (\mathcal{CO}(c_1(TX))-m_0(K)) * \mathcal{CO}(y) = (m_0(K)-m_0(K))\mathcal{CO}(y)=0.
$$
Here we used the equation $\mathcal{CO}(c_1(TX))=m_0(K)\, [K]$, due to Kontsevich, Seidel and Auroux \cite{Auroux} (see also the explanation in \cite{RitterSmith}). We point out that this equation is the count of holomorphic discs bounding $L$, through a generic point of $L$, which hit $\mathrm{PD}(c_1(TX))$. For index reasons (using monotonicity and the fact that $K$ is orientable), only index $0$ and $2$ discs contribute. The index $0$ discs are constant by monotonicity, and they would contribute $[K\cap \mathrm{PD}(c_1(TX))]$, but the assumption that $K$ does not intersect a representative of $\mathrm{PD}(c_1(TX))$ ensures that this term does not arise. Finally $m_0(K)$ is by definition the count of the index $2$ discs.

The remark about Jordan blocks of size at least 2 also follows, since such a block arises as a summand $\mathfrak{R}[x]/(x-m_0(K))^d$ for $d\geq 2$ when we decompose $SH^*$ (resp. $QH^*$) as an $\mathfrak{R}[x]$-module, with $x$ acting by multiplication by $c_1(TX)$. The eigenvectors in this summand are multiples of $(x-m_0(K))^{d-1}$, so they vanish under $\mathcal{CO}$.
\end{proof}

That $K$ does not intersect a representative of the anticanonical class $\mathrm{PD}(c_1(TX))$ holds for toric Lagrangians since they do not intersect the toric divisors. However, it holds in fact in general, after homotopying the representative, as follows. We emphasize that we only work with orientable Lagrangian submanifolds \cite{RitterSmith} (so we may use Poincar\'{e} duality arguments). The following would fail in the non-orientable setting of $\R P^2\subset \C P^2$.

\begin{lemma}\label{Lemma L does not intersect c1}
 For any (orientable) Lagrangian submanifold $K\subset X$ in a monotone symplectic manifold $(X,\omega_X)$, $K$ does not intersect some representative of $\mathrm{PD}(c_1(TX))$.
\end{lemma}
\begin{proof}
$D=PD[c_1(TX)]$ can be represented by the vanishing of a generic smooth section $s: X\to \mathcal{E}$ of a complex line bundle $\mathcal{E}\to X$ with Chern class $c_1(TX)$. Now for any Lagrangian $j: K\hookrightarrow X$, the pull-back bundle $j^*\mathcal{E}$ has first Chern class $j^*c_1(TX)$. But $c_1(TX)=\lambda_X [\omega_X]$ by monotonicity, so $j^*c_1(TX)=\lambda_X [j^*\omega_X]=0\in H^2(X;\R)$ since $K$ is Lagrangian. So $j^*\mathcal{E}$ is a trivial smooth line bundle. Moreover, the intersection $K \cap D$ (for generic $s$, $D=s^{-1}(0)$ will be transverse to $K$) can be represented by the vanishing of the pulled-back section $j^*s$. Since $j^*\mathcal{E}$ is trivial, we can homotope the smooth section $j^*s$ so that it does not vanish. More precisely, in a tubular neighbourhood of $K$ (which has the same homotopy type of $K$, and so the restriction of $\mathcal{E}$ over that neighbourhood is still trivial), we can deform $s$ using a bump function supported near $K$ so that $s$ no longer vanishes near $K$. This is equivalent to a homotopy of $D$ which 
ensures that $K\cap D$ is empty, as required.
%
%
\end{proof}

\begin{theorem}[{Ritter-Smith \cite{RitterSmith}}]\label{Theorem OC respects eigensummands}
$\mathcal{OC}:\mathrm{HH}_*(\mathcal{W}_{\lambda}(X))\to SH^*(X)$ lands in the generalized $\lambda$-eigensummand $SH^*(X)_{\lambda}$ of $SH^*(X)$ (for the $QH^*(X)$-module action of $c_1(TX)$), and respectively $\mathcal{OC}:\mathrm{HH}_*(\mathcal{F}_{\lambda}(X))\to QH^*(X)$ lands in the generalized $\lambda$-eigensummand $QH^*(X)_{\lambda}$. Moreover, if $\mathcal{OC}$ hits an invertible element in that eigensummand, and $m_0(K)=\lambda$, then $\mathcal{CO}\circ \mathcal{OC}$ hits $1 \in HW^*(K,K)$ (respectively $1=[K]\in HF^*(K,K)$).
\end{theorem}

\begin{definition}[Generation Criterion]\label{Definition Generation Criterion holds} Let $B$ be a closed Fano toric manifold, with a choice of toric symplectic form $\omega_F$. For an eigenvalue $\lambda$ of the action of $c_1(TB)$ on $QH^*(B)$, we will say that \emph{``the toric generation criterion holds for $\lambda$''} if the composite %
$$\mathcal{CO}\circ \mathcal{OC}: \mathrm{HH}_*(\mathcal{F}_{\lambda}^{\mathrm{toric}}(B)) \to QH^*(B) \to HF^*(K,K)$$
hits the unit $[K]\in HF^*(K,K)$ for any Lagrangian $K\in \mathrm{Ob}(\mathcal{F}_{\lambda}(B))$. Note that we restricted $\mathcal{OC}$ to the subcategory $\mathcal{F}_{\lambda}^{\mathrm{toric}}(B)\subset \mathcal{F}_{\lambda}(B)$ generated by the toric Lagrangians (with holonomy data).
As remarked above, this condition holds if $\mathcal{OC}$ hits an invertible element in the generalized eigensummand $QH^*(B)_{\lambda}$.

We will say that \emph{``the toric generation criterion holds''} if it holds for all eigenvalues $\lambda$.

The same terminology applies for the compact category $\mathcal{F}(M)$ for admissible toric manifolds (Definition \ref{Definition admissible toric manifolds}). For the wrapped category $\mathcal{W}(M)$, we instead work with
$$
\mathcal{CO}\circ \mathcal{OC}: \mathrm{HH}_*(\mathcal{W}_{\lambda}^{\mathrm{toric}}(M)) \to SH^*(M) \to HW^*(K,K).
$$

By the acceleration diagram \eqref{Equation Intro Comm Diagram Ritter Smith}, the wrapped $\mathcal{OC}$ map can be identified with $$\mathcal{OC}: \mathrm{HH}_*(\mathcal{F}_{\lambda}^{\mathrm{toric}}(M)) \to QH^*(M) \stackrel{c^*}{\to} SH^*(M)$$ because the toric Lagrangians are compact objects, and so the subcategories generated by them in $\mathcal{F}(M)_{\lambda}$ and $\mathcal{W}(M)_{\lambda}$ are quasi-isomorphic via the acceleration functor. Thus, for the purposes of toric generation for the wrapped category, we reduce to working with $\mathcal{CO}\circ \mathcal{OC}: \mathrm{HH}_*(\mathcal{F}_{\lambda}^{\mathrm{toric}}(M)) \to QH^*(M) \to HF^*(K,K)$, since the acceleration map $HF^*(K,K)\to HW^*(K,K)$ sends unit to unit.
\end{definition}

By the above discussion of the generation criterion, and the comments about the twisted generalization in Section \ref{Subsection Fukaya category for non-monotone toric}, observe that when \emph{``the toric generation criterion holds for $\lambda$''} it implies that the relevant category $\mathcal{F}_{\lambda}(B),\mathcal{F}_{\lambda}(M), \mathcal{W}_{\lambda}(M)$ is split-generated by the toric Lagrangians. By Cho-Oh \cite{Cho-Oh} (see also \cite[Sec.6]{Auroux}), the toric Lagrangians $L$ (with holonomy) with $HF^*(L,L)\neq 0$ are known: they are in bijection with the critical points of the superpotential, as explained in the Appendix Section \ref{Subsection Superpotential}.
%
%
Note that for monotone toric negative line bundles $E$, the toric generation for $\lambda=0$ for $\mathcal{F}(E)$ can never hold since $\mathcal{F}(E)_0$ does not contain toric Lagrangians with $HF^*(L,L)\neq 0$, since the critical values of $W_E$ are all non-zero by Theorem \ref{Theorem Jacobian ring is SH}.

\begin{remark} \label{Remark Novikov ring without completing mathfrak R}
We emphasize that our Novikov ring $\mathfrak{R}=\Lambda_s(\!(T)\!)$ defined in Section \ref{Subsection Using a non-monotone toric form is the same as twisting} is a field. This is practical when considering the generation criterion, since a non-zero element in a field is always invertible. We mention however that one could avoid completing in $T$: one can work with $\Lambda_s[t,t^{-1}]$. Indeed, one can work with the field $\Lambda_s$, temporarily losing the $\Z$-grading, prove invertibility, and then reinsert powers of $t$ a posteriori to ensure the $\Z$-grading is respected (this may yield an inverse up to $t^{\textrm{positive}}$, but that we can invert). 
So invertibility of, say, $\mathcal{OC}([\mathrm{pt}])$, for $[\mathrm{pt}]\in C_*(L)=CF^*(L,L)$, will hold over $\Lambda_s$ if and only if it holds over $\Lambda_s[t,t^{-1}]$. Over $\Lambda_s[t,t^{-1}]$ the later statements about the existence of a ``field summand'' $($a copy of $\mathfrak{R}=\Lambda_s(\!(T)\!)
)$ should then be rephrased as ``a free summand over $\Lambda_s[t,t^{-1}]$''.
\end{remark}

\subsection{Calculation of $\mathcal{OC}=\mathcal{OC}^{0|0}: HF(L,L) \to QH^*(X)$ on the point class}
\label{Subsection Calculation of OC on point class}
\begin{lemma}\label{Lemma calculation of OC of point}
Let $X$ be one of the following:
\begin{enumerate}
\item a closed Fano toric manifold $(B,\omega_F)$;
\item a Fano negative line bundle $E\to B$;
\item a (non-compact) admissible toric manifold $M$ (Definition \ref{Definition admissible toric manifolds}).
\end{enumerate}
Let $L$ be the Lagrangian with holonomy corresponding to a critical point $x$ of the (possibly $F$-twisted) superpotential $W$. Consider the map $\mathcal{OC}: HF^*(L,L) \to QH^*(X)$ on the point class $[\mathrm{pt}]\in HF^*(L,L)$ (which is a cycle since $x$ is critical \cite{Cho-Oh,Auroux}). Then respectively:
\begin{enumerate}
\item $\mathcal{OC}([\mathrm{pt}]) = \mathrm{PD}(\mathrm{pt}) + (\textrm{higher order }t) \neq 0 \in QH^*(B)$, where the leading term arises as the constant disc at the input point;
\item $\mathcal{OC}([\mathrm{pt}]) = (\textrm{fibre holonomy})\cdot T \cdot \mathrm{PD}([B]) + (\textrm{higher order }t) \neq 0 \in QH^*(E)$, where the leading term arises as the standard fibre disc through the input point;
\item $\mathcal{OC}([\mathrm{pt}]) = m_0(L)\, \mathrm{PD}(C^{\vee}) + (\textrm{linearly independent terms}) \in QH^*(M),$ which is non-zero if $\lambda=m_0(L)\neq 0$, where $C=\mathrm{PD}(c_1(TM))$ is a compact cycle representative and the lf-cycle $C^{\vee}$ is its intersection dual (with respect to a choice of basis which affects the other terms).
\end{enumerate}
\end{lemma}
\begin{proof}
Case (1) and (2) follow directly upon inspection (in case (2) the constant disc does not contribute because $\mathrm{PD}[\mathrm{pt}]=0\in H^*(E)$ since $H^{\dim_{\R}E}(E)=0$, and the next order term is determined by Maslov $2$ discs, but for dimension reasons $[B]$ is the only test-cycle we can use up to rescaling, and only the fibre Maslov 2 disc hits $B$). One can also prove (2) via (3): the base $[B]$ as an lf-cycle represents $\pi^*c_1(E)=-k[\pi^*\omega_B]$, and $c_1(TE)=\pi^*c_1(TB)+\pi^*c_1(E) = (\lambda_B-k)\pi^*\omega_B$. So we can take $C=\frac{\lambda_B-k}{-k}[B]$. Then $C^{\vee}=-(k/\lambda_E)\textrm{fibre}\in H_*^{lf}(E)$ (so $\mathrm{PD}(C^{\vee})=-(k/\lambda_E)\pi^*\mathrm{vol}_B$, where recall $\lambda_E=\lambda_B-k$). Finally $\mathrm{PD}(C^{\vee}) = \frac{-k}{\lambda_B-k}\pi^*\omega_B^{\mathrm{top}}$ (pull-back is Poincar\'e dual to taking preimages, and $[\mathrm{pt}]$ is PD to $\omega_B^{\mathrm{top}}$ in $H^*(B)$).

For case (3), recall (e.g. see \cite{RitterSmith}) that $\lambda=m_0(L)$ is the count of Maslov 2 discs which bound $L$ and whose boundary passes through the input point in $L$ and that the $c_1(TM)$ action on $HF^*(L,L)$ gives $c_1(TM)*[L]=m_0(L)\, [L]=\lambda\, [L]$ (because those Maslov 2 discs automatically hit the anti-canonical divisor $C=\mathrm{PD}(c_1(TM))$ once in the interior of the disc).
Now $c_1(TM)$ is a multiple of $[\omega_M]$, and $\omega_M$ is exact at infinity ($M$ is conical at infinity). Therefore $c_1(TM)$ can be represented by a closed de Rham form which is compactly supported, which in turn is Poincar\'e dual to a cycle $C\in H_2(M)$. It follows there exists a \emph{compact} lf-cycle $C$ representing $c_1(TM)$ (and since on cohomology the disc count above does not depend on the choice of representative, we may use $C$). Extend $C$ to a basis $C,C_2,C_3,\ldots$ of lf-cycles for $H_*^{lf}(M)$, and take the dual basis $C^{\vee},C_2^{\vee},C_3^{\vee},\ldots$ in $H_*(M)$ with respect to the intersection product (in particular, in this sense the cycle $C^{\vee}$ is then dual to the lf-cycle $C$). By definition, the disc count which yields $c_1(TM)*[L]=m_0(L)[L]$ is the same as the disc count which determines the coefficient of $\mathrm{PD}(C^{\vee})$ in $\mathcal{OC}([\mathrm{pt}])$.
The other terms in the expansion of $\mathcal{OC}([\mathrm{pt}])$ involve $\mathrm{PD}(C_j^{\vee})$ for $j\geq 2$.
These other terms cannot cancel the $\mathrm{PD}(C^{\vee})$ term in the vector space $H^*(M;\Lambda)$  since $C^{\vee},C_2^{\vee},C_3^{\vee},\ldots$ are linearly independent in $H_*(M)$. Thus computation (3) follows (where $m_0(L)$ already contains the relevant power of $t$, namely $t^{1/\lambda_M}$).
\end{proof}
\subsection{Generation for $1$-dimensional eigensummands}
\label{Subsection Generation for 1-dimensional eigensummands}
 
\begin{theorem}[{Ostrover-Tyomkin \cite[Lemma 3.5]{Ostrover-Tyomkin}}]\label{Theorem Ostrover-Tyomkin field summand}
 The Jacobian ring of a closed Fano toric variety has a field summand for each non-degenerate critical point of the superpotential.
\end{theorem}

\begin{theorem}\label{Theorem toric generation for isolated critical value}
 Let $X$ be a closed Fano toric manifold $(B,\omega_F)$, or a Fano negative line bundle $E\to B$, or an admissible toric manifold $M$  (Definition \ref{Definition admissible toric manifolds}). Let $W$ be the associated (possibly $F$-twisted) superpotential. If $p$ is a non-degenerate critical point of $W$, and it is the only critical point with critical value $\lambda =W(p)$, then the toric generation criterion holds for $\lambda\neq 0$ (for $B$ also $\lambda=0$ works).
\end{theorem}
\begin{proof}
Essentially this now follows by Lemma \ref{Lemma calculation of OC of point}.
Since $p$ is non-degenerate, by Theorem \ref{Theorem Ostrover-Tyomkin field summand} the Jacobian ring has a field summand (their argument also holds in the non-compact toric setting). Since it is the only critical point with that eigenvalue, for $B$ we deduce that $QH^*(X)_{\lambda}\cong \mathrm{Jac}(W_X)_{\lambda}$ is a field summand, and for $X=E$ and $X=M$ we deduce $SH^*(X)_{\lambda}\cong \mathrm{Jac}(W_X)_{\lambda}$ is a field summand. 
By Theorem \ref{Theorem OC respects eigensummands}, $\mathcal{OC}(\mathrm{pt})\in QH^*(X)_{\lambda}$, so  $\mathcal{OC}(\mathrm{pt})$ is a non-zero element in a field, and therefore it is invertible. The claim follows.
\end{proof}

Even when the superpotential $W$ is a Morse function (critical points are non-degenerate) the above argument may not apply. Indeed if there are several critical points with the same critical value $\lambda=W(x)$, then $\mathcal{OC}(\mathrm{pt})$ is non-zero in $QH^*(X)_{\lambda}$, which is a sum of fields, but it could be non-invertible. We will study this problem by first considering the twisted theory, where critical values generically separate, and then reconsider the untwisted case in the limit.

\subsection{Generation for generic toric forms}
\label{Subsection Generation for generic toric forms}
The following is due to Ostrover-Tyomkin \cite[Theorem 4.1]{Ostrover-Tyomkin}, based on arguments in Iritani {\cite[Cor.5.12]{Iritani}} and {Fukaya-Oh-Ohta-Ono \cite[Prop.8.8]{FOOOtoric}. Recall that an algebra is \emph{semisimple} if it is isomorphic to a direct sum of fields.
\begin{theorem}[\cite{Iritani,FOOOtoric,Ostrover-Tyomkin}]
\label{Theorem Ostrover-Tyomkin generic semisimplicity}
For a closed Fano toric variety $B$, the superpotential $W_F$ is Morse for a generic choice of toric symplectic form $\omega_F$ $($meaning, after a generic small perturbation of the values $F(e_i)\in \R)$. In particular, $QH^*(B)$ is then semisimple.
\end{theorem}
\begin{remark}\label{Remark Eigensummands for nondegenerate critical points}
The Jordan normal form for the endomorphism of $\mathrm{Jac}(W_F)$ given by multiplication by $W_F$ arises from the primary decomposition of the $\mathfrak{R}[x]$-module $\mathrm{Jac}(W_F)$ (where $x$ acts by $W_F$). The summands have the form 
$$\mathfrak{R}[x]/(x-\lambda)^d$$
where $\lambda$ is an eigenvalue of multiplication by $W_F$, and there can be several summands with the same $\lambda$. Theorem \ref{Theorem Ostrover-Tyomkin field summand} \cite[Theorem 4.1 and Corollary 2.3]{Ostrover-Tyomkin} consists in showing that each critical point $p$ of $W_F$ with critical value $W_F(p)=\lambda$ gives rise to such a summand, and that if $p$ is non-degenerate then $d=1$ (so the summand is a field: a copy of $\mathfrak{R}$). One also needs to show that this decomposition respects the algebra structure, not just the module structure.

Each critical point $p$ of $W_F$ determines an idempotent $1_p\in \mathrm{Jac}(W_F)$ (representing $1\in \mathfrak{R}[x]/(x-\lambda)^d$ in the relevant summand) such that multiplication by $1_p$ corresponds to projection to the summand. These elements $1_p$, as $p\in \mathrm{Crit}(W_F)$ vary, determine a decomposition of $1=\sum 1_p \in \mathrm{Jac}(W_F)$ satisfying the ``orthonormality relations'' $1_p\cdot 1_q = \delta_{p,q}\,1_p$, and the primary decomposition of $\mathrm{Jac}(W_F)$ corresponds to mapping $f\in \
\mathrm{Jac}(W_F)$ to $\sum 1_p\cdot f$.
\end{remark}

This theorem is still not sufficient to deduce that toric generation holds generically, due to the issue of repeated critical values mentioned above. So we need to prove Lemma \ref{Lemma separate critical values}.

\begin{proof}[{\bf Proof of {Lemma \ref{Lemma separate critical values}}}]
 Consider the superpotential $W(a) = \sum s^{a_i} T z^{e_i}$ where $F(e_i)=-a_i$.
By Theorem \ref{Theorem Ostrover-Tyomkin generic semisimplicity} for generic $a$, $W(a)$ is Morse. Now let the values $a_i$ vary locally near a point $A$ for which $W(A)$ is Morse, so $W(a)$ is also Morse. Consider the map
$$
F: \{ (a,z)\in \R^r\times \mathfrak{R}^n: dW(a)|_z = 0 \} \to \mathfrak{R}, \quad F(a,z)=W(a)(z)
$$
which assigns the critical value to the critical points. Near a given critical point $(A,Z)$ of $W(A)$, the critical points of $W(a)$ are smoothly parametrized by $a$: $z=z(a)$ (this follows by an implicit function theorem argument, using that the superpotential is an analytic function and that the critical points are all simple since $W(A)$ is Morse). So the domain of $F$ is parametrized as $(a,z(a))$ near $(A,Z)$ for some smooth function $a\mapsto z(a)$. To see how $F$ varies locally as we vary $a$, consider the partial derivatives
$$
\partial_{a_i} (F(a,z(a))) = (\partial_{a_i} F) + \sum_j (\partial_{z_j} F) \cdot \partial_{a_i} (z_j(a)).
$$
But $\partial_{z_j}F = \partial_{z_j}W(a)=0$ at the critical point $z(a)$, therefore 
$\partial_{a_i} (F(a,z(a))) = s^{a_i} T z^{e_i}$ recovers the $i$-th term in $W(a)$ (here we stipulated that $\partial_{a_i}s^{a_i}=s^{a_i}$ -- one could run the argument first by fixing a real value $s=\exp(1)$, since if critical values are distinct when putting $s=\exp(1)$ then they certainly are also for a formal variable $s$).

As explained in the proof of Lemma \ref{Lemma W(xi z) is xi W(z)}, by making an $SL(n,\Z)$ transformation of the fan one can assume that a subset of the edges $e_1,\ldots,e_n$ become the standard basis of $\R^n$, so $W(a)$ would have the form 
$s^{a_1} T z_1 + \cdots + s^{a_n} T z_n + s^{a_{n+1}} T d_{n+1}(z) + \cdots + s^{a_{r}} T d_r(z)$ where $d_j(z)$ are monomials in $z_i^{\pm 1}$. In particular, $\partial_{a_i}(F(a,z(a)))=s^{a_i} T z_i$ for $i=1,\ldots,n$.

Hence if all partial derivatives $\partial_{a_i}F$ at two different critical points points $(a,z)$ and $(a,\widetilde{z})$ are equal, then all $z_i$-coordinates must be equal, and so $z=\widetilde{z}$. Thus, if any of the critical values were repeated, then a generic small variation of $a$ will separate those critical values, because at different critical points $z,\widetilde{z}$ at least one of the partial derivatives $\partial_{a_i}F$ is different (it is generic because fixing one of the $a_i$ coordinates is a codimension $1$ condition).
%
%
\end{proof}

\begin{remark}\label{Remark rank SH can jump}
By Lemma \ref{Lemma separate critical values}, Theorem \ref{Theorem Eigenvalues of c1 from B to E}, Corollary \ref{Corollary Presentation of QH and SH in twisted case}, if $c_1(TB)\in QH^*(B,\omega_B)$ has a zero eigenvalue, then $\mathrm{rank}\,SH^*(E,\omega_E)$ will jump when we generically deform $\omega_E$ to $\omega_F^E$, since $QH^*(B,\omega_F^B)$ will have no $0$-eigenspace.
This applies to $\mathcal{O}_{\P^1\times \P^1}(-1,-1)$ (Section \ref{Subsection Compuation of QH and SH of O(-1,-1)}).
\end{remark}

\begin{theorem}\label{Theorem toric generation for generic forms}\label{Theorem Toric generation for generic twist}
For a closed Fano toric manifold $B$ with a generic choice of toric symplectic form $\omega_F^B$ close to $\omega_B$, the toric generation criterion holds for $\mathcal{F}(B,\omega_F^B)$ and for the twisted category $\mathcal{F}(B,\omega_B; \underline{\mathfrak{R}}_{\omega_F^B})$. For a Fano negative line bundle $E$, constructing $\omega_F^E$ from a generic $\omega_F^B$ as in Corollary \ref{Corollary Presentation of QH and SH in twisted case}, the toric generation criterion holds for $\mathcal{W}(E,\omega_F^E)$ and for the twisted category $\mathcal{W}(E,\omega_E; \underline{\mathfrak{R}}_{\pi^*\omega_F^B})$ (and it holds for  $\mathcal{F}(E,\omega_F^E)$ and for $\mathcal{F}(E,\omega_E; \underline{\mathfrak{R}}_{\pi^*\omega_F^B})$ for all $\lambda\neq 0$).
For an admissible toric manifold $M$ (Definition \ref{Definition admissible toric manifolds}), with a generic choice of toric symplectic form $\omega_F$ close to the monotone $\omega_M$, for all $\lambda\neq 0$ the toric generation criterion holds for $\mathcal{W}(M,\omega_F)$, $\mathcal{F}(M,\omega_F)$, $\mathcal{W}(M,\omega_M; \underline{\mathfrak{R}}_{\omega_F})$ and $\mathcal{F}(M,\omega_M; \underline{\mathfrak{R}}_{\omega_F})$.
In these generic settings, $QH^*(B,\omega_F^B)\cong QH^*(B,\omega_B;\underline{\mathfrak{R}}_{\omega_F^B})$, $SH^*(E,\omega_F^E)\cong SH^*(E,\omega_E;\underline{\mathfrak{R}}_{\omega_F^E}),$ $SH^*(M,\omega_F)\cong SH^*(M,\omega_M;\underline{\mathfrak{R}}_{\omega_F})$ are semisimple, with each field summand corresponding to a different eigensummand.
\end{theorem}
\begin{proof}
 This follows by Lemma \ref{Lemma separate critical values} and Theorem \ref{Theorem toric generation for isolated critical value}. For $E$, we use Theorem \ref{Theorem Eigenvalues of c1 from B to E} to obtain the result for the toric form $\omega_F^E$ induced on $E$ by the generic $\omega_F^B$.
The twisted analogues are equivalent rephrasings since the twisted monotone setup and the (untwisted) non-monotone setup are identifiable (Corollary \ref{Corollary QH twisted is same as non-exact QH}, Corollary \ref{Corollary Twisted symplectic cohomology for toric forms}, and Corollary \ref{Corollary twisted toric category is non-monotone toric category}).
\end{proof}
%
\subsection{Generation for semisimple critical values}
\label{Subsection Generation for semisimple critical values}
\begin{definition}[Semisimple critical value] A critical value $\lambda$ of the superpotential $W$ of a Fano toric manifold is called \emph{semisimple} if all $p\in \mathrm{Crit}(W)$ with $W(p)=\lambda$ are non-degenerate critical points. We sometimes also call $\lambda$ a \emph{semisimple eigenvalue}.
\end{definition}

\begin{lemma}
 For a semisimple critical value $\lambda$, the generalized eigenspace $QH^*(X)_{\lambda}$ is just the eigenspace of $c_1(TX)$ for $\lambda$.
\end{lemma}
\begin{proof}
 This is a rephrasing of Remark \ref{Remark Eigensummands for nondegenerate critical points}. A generalized eigensummand $\mathfrak{R}[x]/(x-\lambda)^d$ with $d\geq 2$ would give rise to a non-zero nilpotent element $x-\lambda$, which is not allowed in a semisimple summand of $QH^*(X)$. So generalized eigenvectors for $\lambda$ are eigenvectors.
\end{proof}

\begin{theorem}\label{Theorem toric generation for semisimple critical values} (Working with $\mathfrak{R}$ defined over $\K=\C$)\\
For a closed monotone toric manifold $(B,\omega_B)$, the toric generation criterion holds for any semisimple eigenvalue $\lambda$.
For an admissible toric manifold $M$ (Definition \ref{Definition admissible toric manifolds}, for example a monotone toric negative line bundle) the toric generation criterion holds for $\mathcal{W}(M)$ and $\mathcal{F}(M)$ for any semisimple eigenvalue $\lambda \neq 0$.
\end{theorem}
\begin{proof}
By Theorem \ref{Theorem Toric generation for generic twist}, for the $\omega_F$-twisted category the toric generation criterion holds when $F(e_i)$ are generic and close to $-1$. Recall that when $F(e_i)=-1$ then we can identify the twisted with the untwisted theory.

Abbreviate by $QH_{F}^*(B)=QH^*(B,\omega_B;\underline{\mathfrak{R}}_{\omega_F})$ the twisted theory, and similarly $SH_F^*(M)$, $QH_F^*(M)$. Recall by Remark \ref{Remark Eigensummands for nondegenerate critical points} that there is a decomposition $1=\sum 1_p$ in $QH^*(B)$ (respectively $SH^*(M)$, since $\mathrm{Jac}(W_M)\cong SH^*(M)$) where $p$ runs over the critical points of the relevant superpotential $W$, and similarly in the twisted case $1=\sum 1_{p,F}$ in $QH_{F}^*(B)$ (resp. $SH_F^*(M)$) where $p$ runs over the critical points of the twisted superpotential.

The summands $\mathfrak{R}[x]/(x-\lambda)^d$ in the primary decomposition mentioned in Remark \ref{Remark Eigensummands for nondegenerate critical points} arise as the local rings $\mathcal{O}_{Z_W,p}$ of $\mathcal{O}(Z_W)= \mathrm{Jac}(W)$ for a certain scheme $Z_W$ associated to the superpotential $W$ \cite[Corollary 2.3 and Lemma 3.4]{Ostrover-Tyomkin}, and $1_p$ is the relevant unit in that local ring.
As in the proof of Lemma \ref{Lemma separate critical values}, the non-degenerate critical points $p$ and their critical values $\lambda_p$ vary smoothly as we vary the values $F(e_i)$ since the superpotential is an analytic function. A priori, it is not clear whether (or in what sense) the decomposition of $\mathcal{O}(Z_W)$ into summands $\mathcal{O}_{Z_W,p}$ varies continuously, smoothly or analytically with the parameters $F(e_i)$ near parameter values where different eigenvalues $\lambda_p$ collide. For this problem, we developed the necessary matrix perturbation theory in Appendix B (Section \ref{Section Appendix B: Matrix perturbation theory}).

In the terminology of Appendix B, multiplication by $1_p$ is the eigenprojection $P_{\lambda_p}$ onto the generalized eigenspace for the eigenvalue $\lambda_p$ for the family of endomorphisms of the vector space $H^*(B;\mathfrak{R})\cong QH_F^*(B) \cong \mathrm{Jac}(W_F^B)$ given by quantum multiplication by the superpotential $W_F$ (respectively for $M$: working with $H^*(M;\mathfrak{R})\cong QH_F^*(M)$ and quantum multiplication by $c_1(TM)$, or working with a localization thereof: $SH^*_F(M)\cong \mathrm{Jac}(W_F^M)$ using multiplication by $W_F^M$). We will phrase the rest of the argument for $B$, as it is analogous for $M$.

Before applying Appendix B, we need to rephrase the problem in terms of a family $A(x)\in \mathrm{End}(\C^n)$ with $A(x)$ depending holomorphically on a complex parameter $x$ near $0\in \C$. So instead of working over $\mathfrak{R}$, we therefore briefly work over $\C$ by replacing $T=1$, $s=e^x$ (this is legitimate since $B,M$ are Fano: the quantum relations and quantum product only involve finitely many powers of $s,T$). So $W = \sum s^{-F(e_i)} T z^{e_i}$ becomes 
$$A(x) = W(a,x)= \sum \exp(a_i x) z^{e_i}$$
acting by quantum multiplication as explained above, where $a_i=-F(e_i)$, and working over $\C$ instead of $\mathfrak{R}$. In other words, we have \emph{specialized} the quantum cohomology by fixing a value $s=e^x$ for $x\in \C$ close to $0$. Observe that taking $x=0$, so $s=1$, recovers by definition the quantum cohomology for the monotone symplectic form (so $F(e_i)=-1$) except for having dropped $T$ due to the substitution $T=1$ -- but powers of $T$ can be reinserted a posteriori as dictated by the $\Z$-grading.

By assumption, $\lambda$ is a semisimple eigenvalue, so Corollary \ref{Corollary Cgce of espaces for semisimple evalue} applies to the family of matrices $A(x)$. In particular, there is a continuous family of eigenvalues $\lambda_j(x)=\lambda_{p_j}$ converging to $\lambda$, which is analytic on a punctured disc around $x=0$ and which is real-differentiable at $x=0$. By the proof of Lemma \ref{Lemma separate critical values} in Section \ref{Subsection Generation for generic toric forms}, for a generic choice of $a_i=-F(e_i)$ the $\lambda_j(x)$ are pairwise distinct for all small $x\neq 0$, and the $\lambda_j(x)$ are in fact holomorphic even at $x=0$ since $\lambda_j(x) = W(a,x)(z)|_{z=p_j(a,x)}$ where $p_j(a,x)$ are non-degenerate critical points of $W(a,x)$ for small $x$ (since they are non-degenerate at $x=0$). Keeping track also of $a$, the derivatives of the $\lambda_j(a,x)$ at $x=0$ are:
$$
\partial_x|_{x=0} \lambda_j(a,x) = \sum \left.\frac{\partial}{\partial z_i}\right|_{z=p_j(a,x),x=0} \hspace{-9ex}W(a,x)(z)\cdot \left.\frac{\partial}{\partial x}\right|_{x=0} \hspace{-3ex} [p_j(a,x)]_i + \left.\frac{\partial}{\partial x}\right|_{z=p_j(a,x),x=0}\hspace{-9ex} W(a,x)(z) = \sum a_i [p_j(a,0)]^{e_i}
$$
since the first derivative vanishes, as $p_j$ is critical. For $a$ close to $\underline{1}=(1,\ldots,1)$, the leading term is $\sum a_i [p_j(\underline{1},0)]^{e_i}$. We want to show that for generic $a$, these leading terms are different for different $j$, since then the $\partial_x|_{x=0} \lambda_j(a,x)$ are different, and so the second half of Corollary \ref{Corollary Cgce of espaces for semisimple evalue} applies. Consider the auxiliary function $G(z)=\sum a_i z^{e_i}$. Then 
$$
\partial_{a_i}G|_{p_j(\underline{1},0)} = p_j(\underline{1},0)^{e_i}.
$$
Now as in the proof of Lemma \ref{Lemma separate critical values} in Section \ref{Subsection Generation for generic toric forms}, for two different choices of $j$ the $p_j(\underline{1},0)^{e_i}$ must differ for some $i$ otherwise two of the critical points $p_j(\underline{1},0)$ would coincide, contradicting non-degeneracy. Therefore, varying $a_i$ generically ensures that the values $G(p_j(\underline{1},0))=\sum a_i [p_j(\underline{1},0)]^{e_i}$ differ for different $j$, as required.

By Corollary \ref{Corollary Cgce of espaces for semisimple evalue}, it follows that (for generic $a$ close to $\underline{1}$) the eigenspaces $$E_j(x)=QH^*_F(B;\C)_{\lambda_j(x)}$$ for $\lambda_j(x)$ vary holomorphically in $x$ even at $x=0$ (continuously would have sufficed). The final step of the argument is to consider the linear map: 
$$
\mathcal{OC}_j: HF^*(L_{p_j(x)},L_{p_j(x)}) \to H^*(B;\C) \equiv QH_F^*(B;\C),
$$
where $L_{p_j(x)}$ is the toric Lagrangian (together with holonomy data) corresponding to the critical point $p_j(x)$ of $W(a,x)$ (we recall this correspondence is the Appendix Section \ref{Subsection Superpotential}).

By Theorem \ref{Theorem OC respects eigensummands}, $\mathcal{OC}_j$ lands in $E_j(x)$. We care about the image of $[\mathrm{pt}]$:
$$\mathcal{OC}_j(x) = \mathcal{OC}_j[\mathrm{pt}]\in E_j(x).$$
%
%

We will show below that $\mathcal{OC}_j(x)$ depends continuously on $x$. Therefore, at $x=0$, $\mathcal{OC}_j(0)$ lands in $E_j(0)$ which is one of the field summands in the $\lambda$-eigenspace $E_{\lambda}(A(0)) = \oplus E_i(0)$ of $A(0)=W(a,0)$ (the superpotential in the monotone case, working over $\C$). This continues to hold even if we insert the appropriate powers of $T$ to achieve $\Z$-grading, so it holds for the $\mathcal{OC}_j$ map constructed in the monotone case over the usual Novikov ring of Section \ref{Subsection Novikov ring}.

Therefore, in the monotone case, $\mathcal{OC}_j[\mathrm{pt}]$, for $[\mathrm{pt}]\in HF^*(L_{p_j(0)},L_{p_j(0)})$, can only be non-zero in a specific field summand of the $\lambda$-eigenspace of $QH^*(B;\omega_B)$ determined by the convergence of the eigenspaces of $A(x)$ described in Corollary \ref{Corollary Cgce of espaces for semisimple evalue}.

The final ingredient is that, as in the proof of Lemma \ref{Lemma calculation of OC of point}, $\mathcal{OC}_j(0)$ is non-zero and so is invertible in that field summand $E_j(0)$. Therefore $\oplus \mathcal{OC}_j(0)$, summing over all $j$ for which $\lambda_j(0)=\lambda$, is invertible in $\oplus E_j(0)=E_{\lambda}(A(0))$, so the generation criterion holds and the Theorem follows.

Finally, we explain why $\mathcal{OC}(x)$ depends continuously on $x$. The difficulty is that not only the holonomy data but also the toric Lagrangian submanifolds $L_j=L_{p_j(a,x)}$ may vary with $x$. For the monotone form $\omega_B$ (so all $F(e_i)=-1$) the point $y_j\in \Delta$ in the moment polytope corresponding to $L_j$ is always the barycentre (Section \ref{Subsection The barycentre of the moment polytope for monotone toric manifolds}). But for some small deformations of $F$, it may happen that $y_j$ moves away from the barycentre (e.g. this is shown for the one point blow-up of $\CP^2$ in \cite{FOOOtoric}).

Nevertheless, we claim that the moduli spaces of rigid discs bounding $L_p$ counted by $\mathcal{OC}([\mathrm{pt}] )$ for $[\mathrm{pt}]\in C_*(L_p)=CF^*(L_p,L_p)$ (so, with suitable intersection conditions at marked points) vary in a smooth $1$-family with $F$, when all $F(e_i)$ are close to $-1$. 

For low Maslov index discs, that claim can be checked explicitly. Indeed such discs intersect a low number of toric divisors (half of the Maslov index \cite{Cho-Oh,Auroux}), and so the complement of the toric divisors that it does not intersect can be parametrized by a holomorphic chart $\C^n$ in which discs bounding a torus $S^1(r_1)\times \cdots \times S^1(r_n)\subset \C^n$ are known explicitly and they vary smoothly when varying the radii $r_j$. For high Maslov indices, there may not be such a chart, so instead we use the global holomorphic action of the complex torus $T$ on the toric variety as follows (recall that a toric variety is defined in terms of a dense complex torus $T$, see Section \ref{Subsection The fan of a line bundle over a toric variety}).

Toric Lagrangians arise as orbits of the real torus $T_{\R}\subset T$, and we can use elements of $T$ to map one toric Lagrangian $L$ to another $L'$. Since the action of $T$ is holomorphic, this mapping will yield a natural bijection between moduli spaces of holomorphic discs bounding $L$ with those bounding $L'$. If we impose generic intersection conditions at marked points for the discs bounding $L$ then, provided $L'$ is sufficiently close to $L$, the corresponding discs will be close, so the corresponding discs bounding $L'$ will satisfy the same intersection conditions (at some other nearby marked points). 
For sake of clarity, we emphasize that, of course, the weights with which these moduli spaces are counted at the chain level are the culprit of dramatic changes in the homology, rather than an essential change in the moduli spaces (if we generically vary $y$, keeping the holonomy fixed, then $p$ will no longer be a critical point, so $HF^*(L_p,L_p)$ will vanish, but this happens because with the incorrect weights the cancellations of disc counts in the boundary operator will fail).

The upshot, is that $\mathcal{OC}_{F,\lambda_p} : HF_F^*(L_p,L_p) \to QH_F^*(X)_{\lambda_p}$ on the point class in $C_*(L_p)=CF^*(L_p,L_p)$, at the chain level, is the same Laurent polynomial in the generators $x_i=\mathrm{PD}[D_i]\in C^*(X;\mathfrak{R})$ except for the coefficients in $\Lambda_s$ (the power of $t$ in the summands is determined by the grading). These coefficient functions vary smoothly as we vary the values $F(e_i)$ since they come from $s^{\omega_F(u)}$, integrating $\omega_F=\sum - F(e_i) \mathrm{PD}[D_i]$ over the discs $u$.%
%
%
\end{proof}

\begin{remark}
 Theorem \ref{Theorem toric generation for semisimple critical values} also holds for the (non-monotone) Fano toric manifold $(B,\omega_F)$ and for Fano negative line bundles over it. In this case, a perturbation of $\omega_F$ to $\omega_{F'}$ corresponds to keeping $\omega_F$ but twisting by $\omega_{F'}-\omega_F$ $($such as $QH^*(B,\omega_{F'}) \cong QH^*(B,\omega_F; \underline{\mathfrak{R}}_{\omega_{F'}-\omega_{F}}))$.
\end{remark}

\subsection{$\mathbf{\mathcal{OC}^{0|0}:HF^*(L_p,L_p)\to QH^*(B)_{\lambda_p}}$ lands in the $\lambda_p$-eigenspace}
\label{Subsection OC00 lands in espace}

\begin{theorem}\label{Theorem OC on HF(L,L) lands in eigenspace} (Working with $\mathfrak{R}$ defined over $\K=\C$)\\
Let $X$ be a closed monotone toric manifold $B$ or a (non-compact) admissible toric manifold $M$ (Definition \ref{Definition admissible toric manifolds}). 
For $p\in \mathrm{Crit}(W)$, the image of $\mathcal{OC}^{0|0}: HF^*(L_p,L_p)\to QH^*(X)_{\lambda_p}\subset QH^*(X;\omega_X)$ lands in the eigenspace (not just the generalized eigenspace).
\end{theorem}
\begin{proof}
By Theorem \ref{Theorem OC respects eigensummands}, working over $\C$, the map
$$
\mathcal{OC}(x)=\mathcal{OC}^{0|0}: HF^*(L_{p_j(a,x)},L_{p_j(a,x)}) \to QH_F^*(B;\C)
$$
must land in the eigensummand $QH_F^*(B;\C)_{\lambda_j(x)}$ corresponding to the eigenvalue $\lambda_j(x)$ associated with $L_j(x)=L_{p_j(a,x)}$, where $p_j(a,x)\in \mathrm{Crit}(W_F)$, using the notation from the proof of Theorem \ref{Theorem toric generation for semisimple critical values}. By Theorem \ref{Theorem toric generation for semisimple critical values}, for generic $a_i=-F(e_i)$ near $F(e_i)=-1$, the eigenvalues are distinct so this eigensummand is $1$-dimensional. So the image of $\mathcal{OC}(x)$ consists of eigenvectors (or zero).

The theorem follows if we show that in the limit $x=0$, also the image of $\mathcal{OC}(0)$ consists of eigenvectors (observe that the critical points $p_j(a,x)$ of an analytic function will vary continuously, so they are continuous also at $x=0$). By Lemma \ref{Lemma Evecs converge}, the $1$-dimensional eigenspaces $E_j(x)$ of the $\lambda_j(x)$ vary continuously in $\P(\C^n)$ even at $x=0$ where they converge to a $1$-dimensional subspace $E_j(0)$ of the $\lambda$-eigenspace of $W$ (where $\lambda=\lambda_j(0)$).
We claim that $\mathcal{OC}(x)$ is continuous in $x$, therefore $\mathcal{OC}(0)$ is trapped inside the limit $E_j(0)$ as required.

That $\mathcal{OC}(x)$ is continuous in $x$, is proved just like in the proof of Theorem \ref{Theorem toric generation for semisimple critical values}, replacing the point cycle $[\mathrm{pt}]$ by any cycle in $C_*(L)$, and using the vector space isomorphism from \eqref{EqnChoOh}: $H^*(L_p;\Lambda)\cong HF^*(L_p,L_p)$. More precisely, the latter isomorphism is used in the following sense: a Floer cycle corresponds to an ordinary cycle together with (finitely many) higher order $t$ correction terms which are constructed inductively to kill off the Floer coboundary \cite{Cho-Products}. The coefficients in these correction terms depend smoothly on $x$ since the area/holonomy-weights are smooth in $x$ (the relevant finite moduli spaces of discs vary smoothly for small $x$ as in the proof of Theorem \ref{Theorem toric generation for semisimple critical values}, so the areas/holonomies of those discs vary smoothly in $x$). Therefore composing with the isomorphism $H^*(L_p;\Lambda)\cong HF^*(L_p,L_p)$ we can pretend that $\mathcal{OC}^{0|0}$ is defined as a map of vector spaces $H^*(L;\Lambda) \to H^*(X;\Lambda)$ so the machinery of Appendix B applies, just like it did for the point class in Theorem \ref{Theorem toric generation for semisimple critical values}.
\end{proof}
%

\begin{remark}
For the full $\mathcal{OC}$ map defined on $\mathrm{HH}_*(\mathrm{A}_{\infty}\textrm{-algebra of }L_p)$, the above argument does not apply because there is no analogue of \eqref{EqnChoOh}: in fact the rank of this map is expected to jump (see the discussion in Section \ref{Subsection The motivation for requiring the superpotential to be Morse} and Remark \ref{Remark Intro}).
\end{remark}

\subsection{Failure of $\mathcal{OC}^{0|0}$ to detect generation for some degenerate critical values}
\label{Subsection Failure of the generation criterion for some degenerate critical values}

By Ostrover-Tyomkin \cite{Ostrover-Tyomkin} there are closed smooth toric Fano varieties $B$ for which $QH^*(B,\omega_B)$ is not semisimple, due to the presence of a generalized eigenspace (which is not an eigenspace) for multiplication by $W$ on $\mathrm{Jac}(W)\cong QH^*(B)$. The example in \cite[Section 5]{Ostrover-Tyomkin} is the smooth Fano $4$-fold called $U_8$, number 116, in Batyrev's classification \cite{BatyrevClassification}. It has a superpotential $W$ with a critical point $p$ such that $W(p)=-6$ and $\mathrm{Hess}_pW$ is degenerate. A further simple calculation shows that this is the only critical point with value $W(p)=-6$, therefore the eigensummand of $QH^*(U_8)$ for the eigenvalue $-6$ is a generalized eigenspace isomorphic to $\Lambda[x]/(x+6)^d$ for some $d\geq 2$.

\begin{corollary} (Working with $\mathfrak{R}$ defined over $\K=\C$)\\
The map $\mathcal{OC}^{0|0}: HF^*(L_p,L_p)\to QH^*(B)_{\lambda_p}$ does not hit an invertible element for the closed smooth toric $4$-fold $B=U_8$ taken with the monotone toric symplectic form.
\end{corollary}
\begin{proof}
This follows from the above observations and Theorem \ref{Theorem OC on HF(L,L) lands in eigenspace}, since the eigenspace of $x$ in $\Lambda[x]/(x+6)^d$ is spanned by $(x+6)^{d-1}$ and $\mathcal{CO}:QH^*(B)\to HF^*(K,K)$ vanishes on multiples of $c_1(TB)+6$ whenever $m_0(K)=-6$ by Lemma \ref{Lemma CO map vanishes on things}. 
\end{proof}
%
%
\subsection{A remark about generation, in view of Galkin's result}
\label{Subsection Galkin's result}

\begin{theorem}[{Galkin \cite{Galkin}}]\label{Theorem Galkin}
 For closed monotone toric manifolds $(B,\omega_B)$, the complex-valued superpotential $W: (\C^*)^n\to \C$ always has a non-degenerate critical point $p\in (\C^*)^n$ with strictly positive real coordinates.
\end{theorem}

The idea of the proof is to show that a function of type $\sum a_i z^{e_i}$, with positive real $a_i>0$, restricted to $z=\exp(u)\in (\R^*)^n$, has positive-definite Hessian in $u_j$, and that $W\to \infty$ in any direction $u\in \R^n$ going to infinity. The first property is a computation; the second property follows because the cones of the fan of a closed toric manifold cover $\R^n$ so at least one of the terms of $W$ will grow to infinity (the other terms are positive). Thus there is a global minimum $p$. One then checks it satisfies Theorem \ref{Theorem Galkin}.

In the non-compact setting, for admissible toric manifolds $M$ (Definition \ref{Definition admissible toric manifolds}), the first property still holds, but the second property can fail. For example, for $\mathcal{O}_{\P^1}(-1)$, $W=z_1+t z_1^{-1}z_2 + z_2$, putting $t=1$, will not grow to infinity in the direction $u=(-1,-1)$ and the only critical point of $W$ is $z=t\cdot (-1,1)$ which does not have the positive coordinates predicted by Theorem \ref{Theorem Galkin}.

The condition that $W$ grows to infinity at infinity is equivalent to requiring that the fan of $M$ does not lie entirely in a half-plane (this fails for negative line bundles), since then some inner product $\langle u,e_i \rangle >0$ and thus the term $z^{e_i}=\exp \langle u,e_i \rangle$ will grow to infinity as we positively rescale $u$. Subject to checking that $W$ has this growth-property for a particular $M$ (implying Theorem \ref{Theorem Galkin} for $M$) the arguments below would generalize to $M$.

The passage between the complex-valued and the Novikov-valued superpotentials is done as follows. The monotone toric form $\omega_{\Delta}=\sum \mathrm{PD}[D_i]=c_1(TB)$ yields $\lambda_i=-1$ so 
$W=\sum t^{-\lambda_i}z^{e_i}=t\sum z^{e_i}$. Thus it suffices to study the critical points of the complex-valued Laurent polynomial $W=\sum z^{e_i}$ and then reinsert powers of $T$ as dictated by the $\Z$-grading of $QH^*(B)$ (compare with the sanity check in the proof of Theorem \ref{Theorem barycentre of E is critical}).

\begin{corollary}
 For any closed monotone toric manifold $B$, $c_1(TB)\in QH^*(B)$ is not nilpotent.
\end{corollary}
\begin{proof}
Observe that $W=\sum z^{e_i}$ will take a strictly positive real value on the $p$ in Theorem \ref{Theorem Galkin}, so $c_1(TB)\in QH^*(B;\C)$ has a strictly positive real eigenvalue $\lambda_p$. The same is true for the superpotential defined over the Novikov ring $\Lambda$ by reinserting powers of $T$ as dictated by the $\Z$-grading of $QH^*(B)$  (the eigenvalue will be $\lambda_p T$, with $T$ in grading 2).
\end{proof}

\begin{lemma}[{\cite{RitterSmith}}] \label{Lemma c1TB non-nilpotent implies c1TE non-nilpotent}
 For any monotone negative line bundle $E\to B$, if $c_1(TB)\in QH^*(B)$ is non-nilpotent then $c_1(TE)\in QH^*(E)$ is non-nilpotent.
\end{lemma}
\begin{proof}
 This follows by Theorem \ref{Theorem Eigenvalues of c1 from B to E}: a non-zero eigenvalue of $c_1(TB)$ gives rise to a non-zero eigenvalue of $c_1(TE)$. Since only the latter claim is needed, rather than the full Theorem \ref{Theorem Eigenvalues of c1 from B to E}, one can also prove the claim by carefully comparing the critical points of the superpotential $W_E$ in terms of those of $W_B$ and applying Theorem \ref{Theorem crit values of W are evalues}, as was done in \cite{RitterSmith}.
\end{proof}

\begin{corollary}[{\cite{RitterSmith}}]
For any monotone negative line bundle $E\to B$, the symplectic cohomology $SH^*(E)\neq 0$ is non-zero and there is a non-displaceable monotone Lagrangian torus $L\subset E$ with $HF^*(L,L)\cong HW^*(L,L)\neq 0$ using suitable holonomy data.
By Corollary \ref{Corollary Lagrangian tori in E in SE}, $L$ is the unique monotone Lagrangian torus orbit in $E$. It lies in the sphere bundle $SE\subset E$ of radius $1/\sqrt{\pi \lambda_E}$, and it projects to the unique monotone Lagrangian torus orbit in $B$.
\end{corollary}
\begin{proof}
Using Lemma \ref{Lemma c1TB non-nilpotent implies c1TE non-nilpotent}, this follows from $SH^*(E)\cong \mathrm{Jac}(W_E)$ (Theorem \ref{Theorem Jacobian ring is SH}): each non-zero eigenvalue of $c_1(TE)$ must arise as a critical value of $W_E$, and the corresponding critical point gives rise to such an $L$. In \cite{RitterSmith}, this was proved by computing the critical point of $W_E$ in terms of that for $W_B$ and then applying Theorem \ref{Theorem crit values of W are evalues}.
\end{proof}

Galkin's result implies that there is a monotone toric Lagrangian $L_p\subset B$ (with holonomy data) with $HF^*(L_p,L_p)\neq 0$, corresponding to the non-degenerate critical point $p$ of $W_B$ of Theorem \ref{Theorem Galkin} (in fact, $L_p\subset B$ is the toric fibre over the barycentre of the moment polytope, see Section \ref{Subsection The barycentre of the moment polytope for monotone toric manifolds}). If the real positive number $W_B(p)$ is a semisimple critical value, then by Theorem \ref{Theorem toric generation for semisimple critical values} the toric generation criterion holds for $\lambda=W_B(p)$ (and if $p$ is the only critical point with critical value $\lambda=W_B(p)$, then $L_p$ split-generates $\mathcal{F}(B)_{\lambda}$ by Theorem \ref{Theorem toric generation for isolated critical value}). However, $L_p$ is unlikely to split-generate all of $\mathcal{F}(B)_{\lambda}$ if $\dim QH^*(B)_{\lambda} \geq 2$: one can only expect $L_p$ to split-generate the subcategory of those $K\in \mathrm{Ob}(\mathcal{F}(B)_{\lambda})$ for which
$$
1_p * [K] = [K],
$$
where $1_p$ is the unit in the field summand of $QH^*(B)_{\lambda}$ corresponding to $p$. This is true for the following reason: let $1_q$ denote the units in the summands of the primary decomposition of $QH^*(B)_{\lambda}$ as a $\Lambda[x]$-module, where $x$ acts by $c_1(TB)*$ (that is, $q\in \mathrm{Crit}(W_B)$ with $W_B(q)=\lambda$). Then $1=\sum 1_q$ and the ``orthonormality relations'' $1_q*1_p = \delta_{q,p}\, 1_p$ hold by Remark \ref{Remark Eigensummands for nondegenerate critical points}. Since $\mathcal{CO}:QH^*(B)\to HF^*(K,K)$ is a ring map, $\mathcal{CO}(\sum 1_q)=[K]$. Therefore the condition $1_p * [K] = [K]$ implies that $\mathcal{CO}(1_p)=[K]$. Finally Lemma \ref{Lemma calculation of OC of point} implies that $\mathcal{OC}: HF^*(L_p,L_p) \to QH^*(B)$ hits $1_p$ so the generation criterion applies to Lagrangians $K\in \mathrm{Ob}(\mathcal{F}(B)_{\lambda})$ with $1_p*[K]=[K]$.
However, in general $\mathcal{CO}(1_p)=1_p*[K]$ need not equal $[K]$, so more Lagrangians than just $L_p$ are needed to split-generate $\mathcal{F}(B)_{\lambda}$.

\section{Appendix A: The moment polytope of a toric negative line bundle}
\label{Section The moment polytope of a toric negative line bundle}

\subsection{The fan of a line bundle over a toric variety}
\label{Subsection The fan of a line bundle over a toric variety}

For basics on toric geometry, we refer the reader to Fulton \cite{Fulton}, Guillemin \cite{Guillemin}, Cox-Katz \cite[Chp.3]{Cox-Katz}. There are also useful summaries contained in Batyrev \cite{Batyrev} and Ostrover-Tyomkin \cite{Ostrover-Tyomkin}.

A \emph{toric variety} $X$ of complex dimension $n$ is a normal variety which contains a complex torus $T=(\C^*)^{n}$ as a dense open subset, together with an action of $T$ on $X$ extending the natural action of $T$ on itself by multiplication. We always assume that $X$ is non-singular.

A toric variety is described by a \emph{fan}, which is a certain collection of cones inside $\Z^n$. The $n$-dimensional cones correspond to affine open sets in $X$ and their arrangement prescribes the way in which these glue together to yield $X$. More globally, from the fan one can explicitly write down a \emph{homogeneous coordinate ring} for $X$, with coordinates $x_1,\ldots,x_r$ indexed by the edges $e_1,\ldots,e_r\in \Z^n$ of the fan. Namely,
$$
X = (\C^n \setminus Z)/G \; \textrm{ with torus } \; T = (\C^*)^r/G.
$$
Here $Z$ is the union of the vanishing sets of those $x_i$'s  corresponding to subsets of edges which do not span a cone (so $Z$ is the union of the vanishing sets $\cap_{i\in I}(x_i=0)$ for primitive subsets of indices $I=\{ i_1,\ldots,i_a\}$, see Definition
\ref{Definition Primitive}); and $G$ is the kernel of the homomorphism 
$$\mathrm{Exp}(\beta): (\C^*)^r \to (\C^*)^n, \;  (t_1,\ldots,t_r) \mapsto (t_1^{e_{1,1}}t_2^{e_{2,1}}\cdots t_r^{e_{r,1}},\; t_1^{e_{1,2}}t_2^{e_{2,2}}\cdots t_r^{e_{r,2}},\; \ldots),$$
determined by the coordinates of the edges $e_i=(e_{i,1},\ldots,e_{i,n})$. More directly, $G$ will contain $(t^{a_1},\ldots,t^{a_r})$ precisely if $\sum a_i e_i=0$ is a $\Z$-linear dependence relation amongst the edges.

\begin{example}\label{Example O-k over Pm}
 For $X=\mathrm{Tot}(\mathcal{O}(-k)\to \P^m)$, the edges in $\Z^{m+1}$ are $e_1=(1,0,\ldots,0)$, $\ldots$, $e_{m}=(0,\ldots,1,0)$, $e_{m+1}=(-1,\ldots,-1,k)$, $e_{m+2}=(0,\ldots,0,1)$. The cones are the $\R_{\geq 0}$-span of any subset of the edges, except for the two subsets $\{e_1,\ldots,e_{m+1}\}$ and $\{e_1,\ldots,e_{m+2}\}$. Those exceptional subsets determine $Z=(x_1=\cdots=x_{m+1}=0)$. Since $\mathrm{Exp}(\beta): (\C^*)^{m+2} \to (\C^*)^{m+1}$, $t\mapsto (t_1t_{m+1}^{-1},t_2t_{m+1}^{-1},\ldots,t_{m+1}^kt_{m+2})$ it follows that $G=\{(t,\ldots,t,t^{-k}): t\in \C^*\}$ (which corresponds to the linear relation $e_1+\cdots+e_{m+1}-ke_{m+2}=0$). Thus $X=(\C^{m+2}\setminus Z)/G$, with torus $T=(\C^*)^{m+2}/G\cong (\C^*)^{m+1}$, has homogenous coordinates $x_1,\ldots,x_{m+2}$. Then $[x_1:\cdots:x_{m+1}]$ defines the projection $X \to \P^m$ to the base, and the zero section is $(x_{m+2}=0)$. Over the coordinate patches $U_1=(x_1\neq 0)$, $U_2=(x_2\neq 0)$, we can assume $x_1=1,x_2=1$ respectively, and on the overlap $U_1\cap U_2$ we can identify $(1,x_2,\ldots,x_{m+2})$ with $(x_2^{-1},1,x_2^{-1}x_3,\ldots,x_2^{-1}x_{m+1},x_2^kx_{m+2})$ using the $G$-action. So the transition $U_1 \to U_2$ multiplies by $\varphi=1/x_2$ in the base coordinates and by $\varphi^{-k}$ in the fiber coordinate. In general, $g_{ij}=(x_i/x_j)^k$ is the transition in the fibre going from $U_j$ to $U_i$, as expected for $\mathcal{O}_{\P^m}(-k)$.
 \end{example}
 
 We now describe the general construction of $\mathrm{Tot}(\pi: E\to B)$ for a line bundle over a toric variety $B$. We will always assume $B$ is a smooth closed manifold, which imposes constraints on what fans can arise:
\begin{enumerate}
 \item 
 Compactness is equivalent to the condition that the cones of the fan of $B$ cover $\R^n$ (in particular, the $\R_{\geq 0}$-span of the edges $b_1,\ldots,b_r\in \Z^n$ is $\R^n$). 

\item Smoothness is equivalent to the condition that each subset of edges which generates a cone extends to a $\Z$-basis of $\Z^n$ (in particular $b_i$ is \emph{primitive}: it is the first $\Z^n$-point on the ray $\R_{\geq 0} b_i\subset \R^n$).

\end{enumerate}

The edges of $B$ correspond to homogeneous coordinates $x_i$ for $B$. They give rise to divisors 
$$D_i=(x_i=0) \subset B,$$
called \emph{toric divisors}. These are precisely the irreducible $T$-invariant effective divisors of $B$.

\begin{lemma}[see {\cite[Sec.3.4]{Fulton}}]
Any holomorphic line bundle $E \to B$ is isomorphic to $\mathcal{O}(\sum n_i D_i)$ for some $n_i\in \Z$.
\end{lemma}

\noindent{\bf Example.} \emph{The description of $\mathcal{O}(-k) \to \P^m$ in Example \ref{Example O-k over Pm} is $\mathcal{O}(-kD_{m+1})$.}

\begin{lemma}\label{Lemma line bundle from fan of base}
 Let $E=\mathcal{O}(\sum n_i D_i)\stackrel{\pi}{\longrightarrow} B$, for $n_i\in \Z$. A fan for $E$ is given by the edges 
$$e_1=(b_1,-n_1),\; \ldots,\; e_r =(b_r,-n_r),\; e_{r+1}=(0,\ldots,0,1) \in \Z^{n+1}$$
with the following cones. Whenever the $\R_{\geq 0}$-span of a subset $b_{j_1},\ldots,b_{j_k}$ of the $b_i$'s is a cone for $B$, then the $\R_{\geq 0}$-span of $e_{j_1},\ldots,e_{j_k}$ is a cone for $E$ and the $\R_{\geq 0}$-span of $e_{j_1},\ldots,e_{j_k},e_{r+1}$ is a cone for $E$. That subset corresponds to the subvariety $V_J=(x_{j_1}=\cdots=x_{j_k}=0)\subset B$, and it gives rise respectively to two subvarieties in $E$: $\pi^{-1}(V_J)\subset E$ and $V_J\subset B\hookrightarrow E$. 
\end{lemma}
\begin{proof}
In general, an effective divisor $D$ can be described by $f_{\alpha}=0$ on $U_{\alpha}$, where: $U_{\alpha}$ are open affines covering $B$; $f_{\alpha}$ are non-zero rational functions; $f_{\alpha}/f_{\beta}$ are nowhere zero regular functions on $U_{\alpha}\cap U_{\beta}$. Then $\mathcal{O}(D)$ is the $\mathcal{O}_B$-subsheaf of the sheaf of rational functions $\mathcal{M}_B$ on $B$ generated by $1/f_{\alpha}$ on $U_{\alpha}$. Viewed as a line bundle, this has transition function $g_{\beta\alpha}=(f_{\beta}\circ \varphi_{\beta\alpha})/f_{\alpha}:U_{\alpha}\times \C \to U_{\beta}\times \C$, where $x'=\varphi_{\beta\alpha}(x)$ is the change of coordinates $U_{\alpha}\to U_{\beta}$. Indeed comparing inside $\mathcal{M}_B$ we have: $z\cdot \tfrac{1}{f_{\alpha}(x)} = g_{\beta\alpha} z \cdot \tfrac{1}{f_{\beta}(\varphi_{\beta\alpha}(x))}$.

First, consider the simple case $\mathcal{O}(D_1)$. Take $f_1=1$ and the other $f_i=x_1$, where $U_i=(x_i\neq 0)\subset B$. So $\mathcal{O}(D_1)=\mathcal{O}_B\subset \mathcal{M}_B$ on $U_1$, and $\mathcal{O}(D_1)=\tfrac{1}{x_1}\mathcal{O}_B\subset \mathcal{M}_B$ on the other $U_i$. Let us check $g_{21}$, the other cases are similar. Denote by $G$ the group determined by the fan for $E$ described in the claim, and let $G_B$ denote the analogous group for the fan of $B$. There is an element $(t_1,\ldots,t_r)\in G_B$ which via multiplication identifies 
$$\varphi_{21}:U_1\to U_2,\; (1,x_2,\ldots,x_r) \mapsto (x_1',1,x_2',\ldots,x_r')
=(t_1\cdot 1,t_2\cdot x_2,\ldots,t_r\cdot x_r)
$$ 
In particular $t_1=x_1'$. Now $G$ is the kernel of a homomorphism $(\C^*)^{r+1} \to (\C^*)^{n+1}$ with last entry $t_1^{-1}t_{r+1}$ (the powers are the last entries of $e_1,e_{r+1}$). So $t_{r+1}=t_1=x_1'$. So for the fibre coordinate $x_{r+1}\mapsto t_{r+1}x_{r+1}=x_1' x_{r+1}$. So $g_{21}=x_1'=\tfrac{f_2(x')}{f_1(x)}$.

Now consider the general case $\mathcal{O}(\sum n_i D_i)$. Take $f_i(x)=\prod_{j\neq i} x_j^{n_j}$. Now we need 
$$g_{21}=\tfrac{f_2(x')}{f_1(x)}=\tfrac{\prod_{j\neq 2} (x_j')^{n_j}}{\prod_{j\neq 1} x_j^{n_j}} = \tfrac{(x_1')^{n_1}}{x_2^{n_2}}\prod_{j\geq 3} \left(\tfrac{x_j'}{x_j}\right)^{n_j}.
$$
 This time $t_1^{-n_1}\cdots t_r^{-n_r}t_{r+1}=1$, where: $t_1=x_1'$, $t_2=\tfrac{1}{x_2}$, $t_i=\tfrac{x_j'}{x_j}$. Thus $x_{r+1}\mapsto t_{r+1}x_{r+1} = g_{21}x_{r+1}$ as required.
\end{proof}

\subsection{Toric symplectic manifolds and moment polytopes}
\label{Subsection Toric symplectic manifolds and moment polytopes}
Let $(X,\omega_X)$ be a closed real $2n$-dimensional symplectic manifold together with an effective Hamiltonian action of the $n$-torus $U(1)^n$. This action determines a \emph{moment map} $\mu_X: X \to \R^n$, which is determined up to an additive constant (we tacitly identify $\R^n$ with the dual of the Lie algebra of $U(1)^n$). The image $\Delta=\mu_X(X)\subset \R^n$ is a convex polytope, called the \emph{moment polytope}.
By Delzant's theorem $\Delta$ determines, up to isomorphism, $(X,\omega_X)$ together with the action. More precisely, two toric manifolds with moment polytopes $\Delta_1,\Delta_2$ are equivariantly symplectomorphic if and only if $\Delta_2 = A \Delta_1 + \textrm{constant}$ where $A\in SL(n,\Z)$.

The moment polytope has the form
\begin{equation}\label{Equation Moment polytope yei geq lambdai}
\Delta = \{ y\in \R^n: \langle y,e_i \rangle \geq \lambda_i \textrm{ for } i=1,\ldots,r \},
\end{equation}
where $\lambda_i\in \R$ are parameters and $e_i\in \Z^n$ are the primitive inward-pointing normal vectors to the facets of $\Delta$ (the codimension 1 faces). For Delzant's theorem to hold, we always assume that at each vertex $p$ of $\Delta$ there are exactly $n$ edges of $\Delta$ meeting at $p$; that the edges are rational, that is they are of the form $p+\R_{\geq 0}v_i$ for $v_i \in \Z^n$; and that these $v_1,\ldots,v_n$ are a $\Z$-basis for $\Z^n$.

From the polytope, one can construct a fan as follows: a face $F$ is determined by a subset $I_F$ of indices $i$ for which the inequality $\langle y,e_i \rangle \geq \lambda_i$ is an equality. The data $(e_1,\ldots,e_r)$ and $(I_F: F\textrm{ is a face of }\Delta)$ defines the fan, taking cones $\sigma_{I_F}$ to be the $\R_{\geq 0}$-span of the $(e_i: i\in I_F)$.

By construction, the polytope $\Delta$ is combinatorially dual to the fan, in particular the facets are orthogonal to the edges $e_1,\ldots,e_r$ of the fan for $X$. The fan determines the complex structure on $X$, but of course does not encode the $\lambda_i$. In particular, the location of $\Delta$ in $\R^n$ depends on the choice of additive constant in $\mu_X$, which in turn depends on the choice of symplectic form on $X$, and this is not encoded in the fan.

The $\lambda_i$ are related to the symplectic form $\omega_X$ on $X$, by the cohomological condition
$$
[\omega_X] = -\sum \lambda_i\, \mathrm{PD}[D_i] \in H^2(X;\Z)
$$
where $\mathrm{PD}[D_i]$ are the Poincar\'e duals of the divisors $D_i=(x_i=0)\subset X$ corresponding to the vanishing of one of the homogeneous coordinates $x_i$ (which correspond to the edges $e_i$ of the fan). This does not usually determine the $\lambda_i$, since the $[D_i]$ can be linearly dependent, but a refinement \cite[Appendix A.2.1]{Guillemin} of the above formula determines the K\"ahler form $\omega_X$ in terms of the $\lambda_i$ as follows:
$$
\omega_X|_{\mu_X^{-1}(\mathrm{int}(\Delta))} =\tfrac{\sqrt{-1}}{2\pi}\partial\overline{\partial}\,\left({\textstyle\sum} \lambda_i \log[\langle \mu_X(\cdot),e_i \rangle-\lambda_i] + \langle \mu_X(\cdot),{\textstyle\sum}e_i \rangle\right).
$$
\begin{remark}\label{Remark about missing 2pi}  The $2\pi$ ensures that for $\P^m$ one obtains the normalized Fubini-Study form: $\int_{[\P^1]}\omega_{FS}=1$. The $2\pi$ is missing in \cite[Appendix 2.1(1.3)]{Guillemin} and \cite[Appendix 2.3(4.5)]{Guillemin}, but should be there for \cite[Appendix 2.1(1.6)]{Guillemin} to hold, as can be checked for  $\P^1$. Therefore we differ from Cho-Oh \cite[Theorem 3.2]{Cho-Oh} by the rescaling $\omega_P=2\pi \omega_X$.
\end{remark}
\subsection{The moment map}
\label{Subsection The moment map}
The moment map $\mu_X$ is determined by the following diagram
$$
\xymatrix@C=40pt@R=14pt{ 
\C^r - Z \ar@{->}_-{\mathrm{quotient}}[d] \ar@{->}^-{\mathrm{inclusion}}[r] &  \C^r
 \ar@{->}^{\mu_{\C^r}}[r] & \R^r\\
 X = (\C^r - Z)/G \ar@{->}^-{\cong}_-{f_G}[r] & f^{-1}(0)/G_{\R}
\ar@{->}[r]_-{\mu_X} & \R^n
\ar@{->}[u]_-{\beta^t} 
}
$$
where, summarizing \cite[Appendix 1]{Guillemin},
\begin{enumerate} 
 
 \item $r=$ number of edges $e_i$ in the fan for $X$, and $n=\mathrm{dim}_{\C}X$.

 \item $\mu_{\C^r}(x)=\tfrac{1}{2}(|x_1|^2,\ldots,|x_r|^2) + (\lambda_1,\ldots,\lambda_r)$, the moment map for $U(1)^r$ acting by multiplication on $\C^r$ (convention: $z\mapsto e^{i\theta}z$ for $\theta\in \R/2\pi\Z$, $z\in \C$).

 \item $Z\subset \C^r$ is the union of the vanishing sets $(x_i=0: i\in I)$ for those multi-indices $I$ for which $(e_i: i\in I)$ do \emph{not} span a cone of the fan.

 \item $\beta: \R^r\to \R^n$ is the matrix whose columns are the edges $e_i \in \R^n$ of the fan. The matrix $\beta$ has full rank, so the transpose matrix $\beta^t$ is injective.

 \item Let $\mathfrak{g}_{\C} = \ker (\beta: \C^r \to \C^n)=\mathrm{Lie}\,G$ and $\mathfrak{g}_{\R} = \ker (\beta: \R^r \to \R^n)=\mathrm{Lie}\,G_{\R}$.

 \item $G\subset (\C^*)^r$, $G_{\R}\subset U(1)^r$ are the images of $\mathfrak{g}_{\C}$, $\mathfrak{g}_{\R}$ via $\alpha \mapsto (e^{i\alpha_1},\ldots,e^{i\alpha_r})$.

 \item $G,G_{\R}$ act on $\C^r$ by multiplication; and $G=\ker (\mathrm{Exp}(\beta):(\C^*)^r \to (\C^*)^n)$ where 
$\mathrm{Exp}(\beta)$ is $\beta$ conjugated by the maps $\C\to \C/2\pi \Z \to \C^*$, $w \mapsto e^{iw}$.

 \item $f:\C^r \to \R^{r-n}$ is the moment map for the $G_{\R}$-action on $\C^r$. Pick an identification $\ker \beta \cong \R^{r-n}$; $f(x)=\kappa^t\mu_{\C^r}(x)$ where $\kappa:\R^{r-n}\to \R^r$ is the inclusion of $\ker \beta$.

 \item $f_G:G\cdot x \mapsto ((G\cdot x)\cap f^{-1}(0))/G_{\R}$ and $f_G^{-1}: G_{\R}\cdot x \mapsto G\cdot x$.

 \item $\mu_X(x) = (\beta^t)^{-1}_{\mathrm{left}}\cdot \mu_{\C^r}(x)$ on $f^{-1}(0)/G_{\R}$ using the left-inverse, $(\beta^t)^{-1}_{\mathrm{left}}\beta^t=\mathrm{id}$.
\end{enumerate}

The moment polytope $\Delta$ can be recovered from $\mu_X$ by $\mu_X(X)=\Delta$; the $T$-fixed points of $X$ biject with the vertices of $\Delta$ via $\mu_X$; the $T$-orbits biject with the faces of $\Delta$ via $\mu_X$.

Since $(\beta)_{\mathrm{left}}^{-1} e_i$ is the $i$-th standard vector, the formula for $\omega_X$ simplifies when $x\in f^{-1}(0)$:
$$
\begin{array}{rcl}
\langle \mu_X(x),e_i \rangle-\lambda_i &=& 
\langle (\beta^t)^{-1}_{\mathrm{left}}\mu_{\C^r}(x),e_i\rangle - \lambda_i =
\langle \mu_{\C^r}(x),(\beta)_{\mathrm{left}}^{-1} e_i\rangle - \lambda_i = \tfrac{1}{2}|x_i|^2\\
\langle \mu_X(x),\sum e_i\rangle & = & \langle \mu_{\C^r}(x),\sum (\beta)^{-1}_{\mathrm{left}}e_i\rangle = \sum \tfrac{1}{2}|x_i|^2 + \sum \lambda_i\\
\omega_X|_{f^{-1}(0)\cap (x_i\neq 0)} &=&\tfrac{\sqrt{-1}}{2\pi}\partial\overline{\partial}\,\left({\textstyle\sum} \lambda_i \log \tfrac{1}{2}|x_i|^2 + {\textstyle\sum} \tfrac{1}{2}|x_i|^2\right).
\end{array}
$$
%

\subsection{The polytope of a Fano variety}
\label{Subsection The polytope of a Fano variety}
%
A polytope $\Delta \subset \R^n$ is called \emph{reflexive} if: its vertices lie in $\Z^n$, the only $\Z^n$-point lying in the interior of $\Delta$ is $0$, and the $\lambda_i=-1$. These were studied by Batyrev (see a discussion in \cite[Sec.3.5]{Cox-Katz}), in particular: a closed toric variety $X$ is Fano if and only if $X$ admits a polytope $\Delta$ which is reflexive. Recall \emph{Fano} means the anticanonical bundle $\Lambda^{\mathrm{top}}_{\C}TB$ is ample.
For a reflexive $\Delta$, the associated K\"ahler form $\omega_{\Delta}$ lies in the class $[\omega_{\Delta}]=\sum -\lambda_i \mathrm{PD}[D_i] = c_1(TX)$ since $\lambda_i=-1$ (as $c_1(TX)=\sum \mathrm{PD}[D_i]$ holds in general).
%
%
%
So monotone negative line bundles $E \to B$ over toric $B$ always arise as follows: 

\begin{enumerate}
 \item $B$ is a Fano variety with an integral K\"ahler form $\omega_{\Delta}$ coming from a reflexive $\Delta$;
 \item by rescaling $\omega_{\Delta}$ we obtain a primitive integral K\"ahler form $\omega_B$;
 \item $c_1(TB)=[\omega_{\Delta}]=\lambda_B [\omega_B]$ where $\lambda_B\in \Z$ is called the \emph{index} of the Fano variety;
 \item up to isomorphism there is only one n.l.b. $E=E_k$
with $c_1(TE)=-k[\omega_B]$, $k\in \Z_{>0}$;
 \item $E_k$ is monotone iff $1\leq k \leq \lambda_B -1$ (so we need $\lambda_B=$ index of Fano $\geq 2$).  
\end{enumerate}
%
%
\textbf{Example.} A smooth complete intersection $B\subset \P^{n+s}$ defined by equations of degrees $d_1\geq \cdots \geq d_s>1$ is Fano iff $1+n+s - (d_1+\cdots + d_s)\geq 1$, and this difference is the index of the Fano. One can replace $\P^{n+s}$ by weighted projective space $\P(a_0,\ldots,a_m)$, then for $\mathrm{dim}_{\C}\,B\geq 3$ the index is $ \sum a_j - \sum d_i \geq 1$, see Kollar \cite[p.245]{Kollar}. There are only finitely many Fano toric varieties of dimension $n$ up to isomorphism since there are only finitely many reflexive polytopes up to unimodular transformation. The example shows there are many of index $>1$.
%
\subsection{Using non-reflexive polytopes}
\label{Subsection non-reflexive polytopes}
For $\P^2$ the reflexive polytope has vertices $(-1,-1)$, $(2,-1)$, $(-1,2)$, barycentre $(0,0)$, symplectic form $\omega=(1+2)\omega_{\P^2}$. But often one prefers to use the non-reflexive polytope $\frac{1}{1+2} (\Delta -(-1,-1))$: vertices $(0,0),(1,0),(0,1)$, barycentre $(\tfrac{1}{1+2},\tfrac{1}{1+2})$, symplectic form $\omega_{\P^2}$. Here $1+2$ plays the role of $\lambda_{\P^m}=1+m$.

The next Lemma explains how, from a reflexive polytope $\Delta$ for $B$ one obtains a polytope $\Delta_B$ inducing $[\omega_B]=(1/\lambda_B)[\omega_{\Delta}]$, having $0$ as vertex. Using $\Delta_B$ we can construct the polytope $\Delta_E$ of negative line bundles in Section \ref{Subsection The moment polytope of toric negative line bundles} (whereas for $\Delta$, $[\omega_{\Delta}]=c_1(TB)$, so it would be unclear how to get $E=\mathcal{O}(\sum n_i D_i)$ with $c_1(E)=-(k/\lambda_B)[\omega_{\Delta}]$, since $k/\lambda_B$ is fractional). 

\begin{lemma}\label{Lemma description of translated polytope}
 Let $v$ be a vertex of $\Delta$. Then $\Delta_B = \tfrac{1}{\lambda_B} (\Delta -v)$ is a polytope with: vertices in $\Z^n$, $\lambda_i^B \leq 0$ in $\Z$, associated symplectic form cohomologous to $\omega_B$, so $[\omega_B]=\sum -\lambda_i^B \mathrm{PD}[D_i]$. Moreover, the barycentre $y_{\mathrm{bar}}$ of $\Delta_B$ satisfies $\langle y_{\mathrm{bar}},b_i \rangle - \lambda_i^B = \tfrac{1}{\lambda_B}$.
\end{lemma}
\begin{proof}
Let $b_i$ be the edges of the fan for $B$ (so the inward primitive normals of $\Delta$). Adding $-v$ to $\Delta$ changes $\lambda_i=-1$ to $\lambda_i + \langle -v,b_i \rangle \in \Z$, so changes $[\omega]$ by $ \sum \langle v,b_i \rangle \mathrm{PD}[D_i]$. But in general $\sum b_i [D_i]=0$ are $n$ relations satisfied by the divisor classes $[D_i]$ (combining the Lemma p.61 and the Corollary p.64 in Fulton \cite{Fulton}). So $[\omega]$ does not change.

The translated polytope $\Delta-v$ still has normals $b_i$ and has $0=v-v\in \partial \Delta$, so from the equations $\langle 0,b_i \rangle \geq \lambda_i^B$ we obtain $\lambda_i^B\leq 0$.

Applying $A \in GL(n,\Z)$ to a polytope changes $b_i$ to $A'b_i$ where $A'=(A^{-1})^T\in GL(n,\Z)$. So the $\lambda_i$ don't change. The fan changes by applying $A'$, but that keeps the toric variety unchanged. The moment map $\mu_B$ becomes $A\mu_B$, so the symplectic form associated to the polytope does not change.

Consider an edge from $v$ to $v'$ in $\Delta$. For some $A$ as above, $A(v'-v)=(a,0,\ldots,0)$, some $a\in \Z$. By Guillemin \cite[Thm.2.10]{Guillemin}, the Euclidean volume of a polytope is the symplectic volume of the toric manifold. Thus the Euclidean length of an edge in $A(\Delta-v)$ is the symplectic area of the $2$-sphere in $B$ corresponding to that edge. Let $u$ denote the sphere for the edge joining $0$ to $A(v'-v)$. Then $a=\int u^*[\omega] = \lambda_B \int u^*[\omega_B]$. But $\int u^*[\omega_B]\in \Z$ since $\omega_B$ is an integral form. So $a$ is divisible by $\lambda_B$. So also $v'-v\in \Z^n$ has entries divisible by $\lambda_B$. Applying this argument to any $v$ shows that differences of vertices lying on edges of $\Delta$ are divisible by $\lambda_B$. Thus the vertices of $\Delta-v$ are divisible by $\lambda_B$ so $\tfrac{1}{\lambda_B}(\Delta-v)$ is an integral polytope. Since each face of this polytope lies on a hyperplane given by equations $\langle y,b_i\rangle = \lambda_i^B$ for a 
certain subset of the indices $i$, we also conclude that $\lambda_i^B=\tfrac{1}{\lambda_B}(\lambda_i + \langle -v,b_i \rangle)$ are integers. In particular, since also $\mu_B$ gets rescaled by $\tfrac{1}{\lambda_B}$, the associated symplectic form gets rescaled by $\tfrac{1}{\lambda_B}$ so it now lies in the cohomology class $[\omega_B]$.

The barycentre: by construction $\langle y,b_i \rangle - \lambda_i$ is invariant under translating $\Delta$ (and $y$), and for $\Delta$ the barycentre $0$ gives value $1$. Rescaling the polytope by $\tfrac{1}{\lambda_B}$ rescales this $1$ by $\tfrac{1}{\lambda_B}$.
\end{proof}

\subsection{The moment polytope of toric negative line bundles}
\label{Subsection The moment polytope of toric negative line bundles}
%
%
\begin{lemma}\label{Lemma moment polytope for neg l bdle}
The moment polytope $\Delta_E$ for $E=\mathcal{O}(\sum n_i D_i) \to B$ inducing the symplectic form $[\omega_E]=[\pi^*\omega_B]$ is
$$
\begin{array}{lll}
\Delta_E  &=& \{ y\in \R^{n+1}: \langle y,e_i \rangle \geq \lambda_i^E\}
\\
&=& \{ (Y,y_{n+1})\in \R^n\times \R: y_{n+1}\geq 0 \textrm{ and } \langle Y, b_i \rangle \geq \lambda_i^B + y_{n+1}n_i  \textrm{ for } i=1,\ldots,r\}.
\end{array}
$$
Namely: $\lambda_{r+1}^E=0$ and $\lambda_i^E = \lambda_i^B$, $e_{r+1}=(0,\ldots,0,1)$ and $e_i=(b_i,-n_i)$ (see Lemma \ref{Lemma line bundle from fan of base}). In particular, $\omega_E|_B=\omega_B$.
Geometrically, $\Delta_E \subset \R^{n+1}$ lies in the upper half-space $y_{n+1}\geq 0$; its facet along $y_{n+1}=0$ is $\Delta_B\hookrightarrow \R^{n+1}$; it has the same vertices as $\Delta_B$; and the other facets lie in the hyperplanes in $\R^{n+1}\cap (y_{n+1}\geq 0)$ normal to $e_i=(b_i,-n_i)$ passing through the facets of $\Delta_B$ $($except $e_{r+1}$ which is  normal to the facet $\Delta_B)$.
\end{lemma}
\begin{proof}
 The facets are normal to the edges $e_i$ of the fan for $E$ because moment polytopes are combinatorially dual to fans.
The $T$-invariant divisors in $E$ are $\pi^{-1}D_i=(x_i=0)$ for $i\leq r$, and $B=(x_{r+1}=0)$. As locally finite cycles, these are Poincar\'e dual to $\pi^*\mathrm{PD}_B[D_i]$ and $\pi^*c_1(E)$ respectively. 
Recall $[\omega_B]=-\sum \lambda^B_i \mathrm{PD}_B[D_i]$, therefore
$$
[\omega_E] = -\sum (\lambda^E_i \mathrm{PD}_E[\pi^{-1}D_i]+\lambda^E_{r+1}[B]) = - \sum \lambda_i^B \pi^*\mathrm{PD}_B [D_i] = \pi^*[\omega_B].
$$
That $\omega_E|_B=\omega_B$ can be seen from the explicit formula for $\omega$ at the end of Section \ref{Subsection The moment map}, using $\lambda_{r+1}^E=0$ (which ensures the $\log \tfrac{1}{2}|x_{r+1}|^2$ term does not appear), restricting to the subspace $(dx_{r+1}=0)\subset TE$ and letting $x_{r+1}\to 0$.
%
\end{proof}
%
%
\begin{lemma}\label{Lemma forms agree for line bundle}
 For any negative line bundle $E$ over toric $B$, $\omega_E$ agrees with the form $\omega$ of \ref{Subsection Negative line bundles}, with Hermitian norm $\tfrac{1}{\sqrt{2\pi}}|w|$ for the fibre at a point $x$ satisfying $\mu_E(x) = (\ldots,\tfrac{1}{2}|w|^2)$.
\end{lemma}
\begin{proof}[Sketch Proof]
 The proof requires investigating the construction in \cite[Appendix A.2.1]{Guillemin} of the K\"ahler form $\omega_X$ on a toric $X$. The $\omega_X$ is determined from its restriction to the fixed point set $X_{r}$ of the involution on $X$ induced by complex conjugation on $\C^r$. On $X_r$, the K\"ahler metric becomes $\tfrac{1}{2\pi}\sum (dx_i)^{\otimes 2}$ (see \cite[Appendix A.2.2(2.5)]{Guillemin} using Remark \ref{Remark about missing 2pi}). So for $X=E$ the fibre component is $\tfrac{1}{2\pi}(dw)^{\otimes 2}$. Since this recovers the standard metric $(d\rho)^{\otimes 2}$ on the real ray $\R_{>0}\subset \C$ in the fibre, we deduce that the Hermitian norm for the fibre is $\rho=\tfrac{1}{\sqrt{2\pi}}|w|$.
\end{proof}
\subsection{The Landau-Ginzburg superpotential}
\label{Subsection Superpotential} 
%
\begin{definition*}[Preliminary Version]
 The superpotential $W:(\C^*)^n \to \C$ for a toric variety $X$, with $\dim_{\C} X=n$, is the Laurent polynomial %
$
W(z_1,\ldots,z_{n}) = \sum_{i=1}^r e^{\lambda_i} z^{e_i}
$
defined on the domain $(\C^*)^n \cap \{|e^{\lambda_i} z^{e_i}|< 1 \, \forall i\} = \Log^{-1}(\mathrm{int}(\Delta))$ described by
$$
\xymatrix@C=30pt@R=22pt{ 
  X\setminus \cup D_i = (\C^*)^r/G \ar@{->}^-{\mathrm{Exp}(\mu)}[r]\ar@{->}_-{\mathrm{inclusion}}[d]  & (\C^*)^n \supset \Log^{-1}(\mathrm{int}(\Delta))  \ar@{->}^{\Log}[drr] \ar@{->}^-{W}[rr] & &  \C
\\
 X= (\C^r-Z)/G 
\ar@{->}[rrr]_-{\mu_X} 
& & & \Delta \subset \R^n
}
$$
where
$
\Log(z) = (-\log |z_1|,\ldots,-\log |z_n|),
$
 $\mathrm{Exp}(\mu)(x) = (e^{-\mu_{X,1}(x)},\ldots,e^{-\mu_{X,n}(x)})$ involving the components $\mu_{X,j}(x)\in \R$ of $\mu_X(x)\in \R^n$.
\end{definition*}
\begin{explanation*}[Auroux {\cite[Prop.4.2]{Auroux}})]
The domain of $W$ is actually the moduli space $M$ of gauge equivalence classes of special Lagrangian submanifolds $L$ inside the torus $T=X\setminus \cup D_i\cong (\C^*)^n$ equipped with a flat unitary connection on the trivial complex line bundle over $L$. Then $W$ is a weighted count of $\mu=2$ holomorphic discs bounding $L$ with a boundary marked point constraint through a generic point of $L$, and the weight in the count is $e^{-\omega[u]}\cdot \mathrm{holonomy}(u|_{\partial \D})$.

A Lagrangian $L\subset X\setminus \cup D_i$ is called \emph{special} if some imaginary part $\mathrm{Im}(e^{-i\,\textrm{constant}}\Omega|_L)=0$, where $\Omega=d\log x_1 \wedge \cdots \wedge d\log x_n$ is a non-vanishing holomorphic $n$-form on $X\setminus \cup D_i$ (indeed it is a section of the canonical bundle $K_X$ with poles along $D_i$).
Such $L$ have the form $S^1(r_1)\times \cdots \times S^1(r_n) \subset (\C^*)^n$ where $r_i$ denote the radii. 

The biholomorphism $M \cong \Log^{-1}(\mathrm{int}(\Delta))$ is: $z_j=e^{-\mu_{X,j}(L)}\cdot \mathrm{holonomy}([S^1(r_j)])$, where $\mu_{X,j}$ is constant on $L$ since $L$ is a $T$-orbit, and where $[S^1(r_j)]\in H_1(L)$ determines the holonomy for the connection. Therefore the equations $|e^{\lambda_i} z^{e_i}|\leq 1$ correspond via $y=\mu_X(x)=\Log(z)$ to the equations $\lambda_i-\sum_{j=1}^n y_je_{i,j}\leq 0$ defining $\Delta$.

After this biholomorphic identification, by Cho-Oh \cite{Cho-Oh} and Auroux \cite[Prop.4.3]{Auroux}, 
$
W = \sum e^{\lambda_i} z^{e_i},
$
where $e_i$ is the primitive inward-pointing normal to the facet of $\Delta$ defined by the equation $\{y\in \R^n: \langle y,e_i\rangle = \lambda_i\}\cap \Delta$.

%
%
\emph{Remark: there are no $2\pi$'s arising in our $z_j$, $\Log$, $\mathrm{Exp}$, $W$ due to Remark \ref{Remark about missing 2pi}.}
%
%
%
\end{explanation*}
\begin{example*}
 For $\P^m$, $W=z_1+ \ldots + z_{m} + e^{-1}z_1^{-1}\cdots z_m^{-1}$.
\end{example*}

\begin{example}\label{Example W for monotone tori}
 If the special Lagrangian $L$ with connection $\nabla$ in the Explanation is monotone, and $\lambda_L,\lambda_X$ are the monotonicity constants for $L,X$ $($recall $2\lambda_L\lambda_X=1)$,
$$W(L,\nabla)=\#(\mathrm{discs})\, e^{-2\lambda_L} = \#(\mathrm{discs})\, e^{-1/\lambda_X}$$
 where $\#(\mathrm{discs})$ is the weighted count of Maslov $2$ discs bounding $L$ as in the Explanation, the weights being the holonomies around the boundary of the discs.
\end{example}


One now actually wants to deform the Floer theory for the special Lagrangians. This is explained in Auroux \cite[Sec.4.1]{Auroux2}: it can be done either by 
allowing non-unitary connections on $L$, or by deforming $L$ by a non-Hamiltonian Lagrangian isotopy (and using a unitary connection), or by formally deforming the Fukaya category by a cocycle $b_L\in CF^1(L,L)$ by the machinery of Fukaya-Oh-Ohta-Ono  \cite{FOOO} (this corresponds to the connection $\nabla=d+b_L$).
We will allow non-unitary connections, which is an idea that goes back to Cho \cite{Cho}, and is explained also in Fukaya-Oh-Ohta-Ono \cite[Sec.4, Sec.12]{FOOOtoric} and in Auroux 
\cite[Rmk 3.5]{Auroux}.
So we work over the Novikov ring $\Lambda$ with $\K=\C$; the $e^{-1}$ above is replaced by $t$; and we work with $\mathcal{W}(E)$: the wrapped category with local systems (see Ritter-Smith \cite{RitterSmith}).


%

\begin{definition}[Corrected Version]\label{Definition Superpotential}
 The superpotential $W:(\Lambda\setminus \{0\})^n \to \Lambda$ $($with $\K=\C)$ for a toric variety $X$, with $\dim_{\C}X=n$, is
$$
W(z_1,\ldots,z_{n}) = \sum_{i=1}^r t^{-\lambda_i} z^{e_i}.
$$
\end{definition}
\begin{example*}
 For $\P^m$, $W=z_1+ \ldots + z_{m} + tz_1^{-1}\cdots z_m^{-1}$.
\end{example*}

The condition that a point $z=(z_1,\ldots,z_r)$ lands inside the polytope is now the condition:
$\mathrm{val}_t(t^{-\lambda_i}z^{e_i}) > 0$ for $i=1,\ldots,r$,
where $\mathrm{val}_t$ is the valuation for the $t$-filtration, whose value on a Laurent series is the lowest exponent of $t$ arising in the series. Indeed if $z_i \in t^{y_i}\C^* + \textrm{(higher order)}$, then that condition becomes $\langle y,e_i \rangle > \lambda_i$, the equations defining $\mathrm{interior}(\Delta)$. We recover the point of the polytope over which the toric fibre $L$ lies and the holonomy around each generating circle of $\pi_1(L)$ by
$$
\Lambda^n \ni z \mapsto (\mathrm{val}_t(z),t^{-\mathrm{val}_t(z)}z) \in \mathrm{interior}(\Delta) \times (\Lambda_0^{\times})^n \subset \R^n \times H^1(L,\Lambda_0^{\times}) 
$$
where $\Lambda_0^{\times}$ is the multiplicative group of units in the subring $\Lambda_0 \subset \Lambda$ of series with $\mathrm{val}_t\geq 0$.
We say $z$ \emph{lands at} $y\in \R^n$ if
$
y = \mathrm{val}_t(z) = (\mathrm{val}_t(z_1),\ldots,\mathrm{val}_t(z_n)).
$
\begin{example}\label{Example Superpotential of nlb}
For $E=\mathcal{O}(\sum n_i D_i) \to B$, 
$
W(z_1,\ldots,z_{n+1}) = \sum_{i=1}^r t^{-\lambda_i^B} z^{(b_i,-n_i)} + z_{n+1},
$
by Lemma \ref{Lemma moment polytope for neg l bdle}.
Since $[\omega_B]=\sum -\lambda_i^B \mathrm{PD}[D_i]$, 
$E=\mathcal{O}(k\sum \lambda_i^B D_i)$ has $c_1(E)=-k[\omega_B]$ and
$$
W_E(z_1,\ldots,z_n,z_{n+1}) = \sum_{i=1}^r z^{b_i} (tz_{n+1}^k)^{-\lambda_i^B}
+ z_{n+1} = z_{n+1}+  \left.W_B(z)\right|_{\left(t \textrm{ replaced by }tz_{n+1}^k\right)}.
$$
\end{example}

\begin{theorem}\label{Theorem crit values of W are evalues}
 For monotone negative line bundles $E \to B$ over toric $B$, the critical values $W(p)$ of the superpotential are a subset of the eigenvalues of $c_1(TE):QH^*(E) \to QH^{*+2}(E)$ acting by quantum cup-product. The critical points $p$ of $W$ correspond to Lagrangian toric fibres $L_p$ of the moment map, together with a choice of holonomy data, such that 
 $$HF^*(L_p,L_p) \cong HW^*(L_p,L_p)\neq 0.$$
 Thus $L_p$ are non-displaceable Lagrangians, and the existence of an $L_p$ forces $SH^*(E)\neq 0$.
\end{theorem}
\begin{proof}
 This is proved by mimicking Auroux \cite[Sec.6]{Auroux}, except that in the non-compact setup we use locally finite cycles to represent cohomology classes. The monotonicity condition on $E$ ensures that $B$ and $E$ are Fano toric varieties. The fact that independently of whether $L$ is monotone or not,  Floer cohomology can be defined is by Cho-Oh \cite[Sec.7]{Cho-Oh}: the $m_1$-obstruction class vanishes when $(L,\nabla)$ arises as a critical point of $W$ and the toric Fano assumption then ensures that Lagrangian Floer cohomology exists and is non-trivial (see also Auroux \cite[Prop.6.9, Lemma 6.10]{Auroux}). That $HF^*(L_p,L_p)\cong HW^*(L_p,L_p)$ follows because the torus $L_p$ is compact. The final claim follows because if $SH^*(E)=0$, then $HW^*(L,L)=0$ since it is a module over $SH^*(E)$ (see \cite{Ritter3}).
\end{proof}

\begin{example}
For $E=\mathcal{O}_{\P^m}(-k)$, $W = z_1+\cdots+z_m+t^kz_1^{-1}\cdots z_m^{-1}z_{m+1}^{k} + z_{m+1}$. 
The critical points of $W$ are $z=(w,\ldots,w,-kw)$ with critical value $W(z)= (1+m-k)w$, for any
solution $w$ of $w^{1+m-k}=(-k)^k t^{k}$. The $w,W(z)$ are eigenvalues respectively of $\pi^*[\omega_{\P^m}], c_1(TE)=(1+m-k)\pi^*[\omega_{\P^m}]$. Here $1\leq k\leq m$ is required for $E$ to be monotone.
\end{example}

\begin{corollary}\label{Corollary crit values of W are homogeneous}
 The critical values of $W$ are homogeneous in $t$ of order $t^{1/\lambda_E}=T$.
\end{corollary}
\begin{proof}
 Since $E$ is monotone, $QH^*(E)$ is $\Z$-graded using grading $|t|=2\lambda_E$ (see \ref{Subsection Novikov ring}). This grading of $t$ ensures that quantum cup product is grading preserving on $QH^*(E)$. Since $c_1(TE)$ lies in grading $2$, the eigenvalues must lie in grading $2$. Finally $|t^{1/\lambda_E}|= 2$.
\end{proof}

\begin{remark}\label{Remark crit W only detects evals on SH}
 $\mathrm{Crit}(W)$ only detects the eigenvalues of the action of $c_1(TE)$ on $SH^*(E)$ by Theorem \ref{Theorem Jacobian ring is SH}, so it forgets the zero eigenvalues of the action on $QH^*(E)$.
\end{remark}

\subsection{Critical points of $W_X$ arise in $\lambda_X$-families}
\label{Subsection Critical values arise in lambdaX-families}

\begin{lemma}\label{Lemma W(xi z) is xi W(z)}
 For $(X,\omega_X)$ a monotone toric manifold with integral $\omega_X\in H^2(X,\Z)$, $$W_X(\xi z) = \xi W_X(z) \textrm{ whenever }\xi^{\lambda_X}=1.$$
 It also holds for Fano toric manifolds $(X,\omega_F)$ for
 the $F$-twisted superpotential (Section \ref{Subsection F-twisted superpotential and Jac WF}).
\end{lemma}
\begin{proof}
After relabelling the indices, we can assume $W_X = \sum t^{-\lambda_i} z^{e_i}$ has the form
$$
W_X = z_1 + \cdots + z_n + d_{n+1}(z) + \cdots + d_r(z)
$$
where $d_j(z)$ are monomials in the free variables $z_i^{\pm 1}$. This is proved as follows.
Consider a top-dimensional cone for $X$, say $\mathrm{span}_{\R_{\geq 0}}\{ e_1,\ldots,e_n \}$ (relabel indices if necessary). This corresponds to a chart for $X$. Making an $SL(n,\Z)$ transformation so that $e_1,\ldots,e_n$ become the standard basis of $\R^n$, and translating the moment polytope so that $\lambda_1,\ldots,\lambda_n$ become zero, will not affect $X$ (up to an equivariant symplectomorphism). It follows that $W_X$ has the above form, and the functions $d_j$ express the linear dependence relations amongst edges. Explicitly: $d_j(z)=t^{-\lambda_j} z_1^{a_1}\ldots z_n^{a_n}$ means 
$
e_j = a_1 e_1 + \cdots + a_n e_n.
$

\emph{Example: For $\P^2$, $W_X(z)=z_1+z_2+tz_1^{-1}z_2^{-1}$ so $d_3(z) = tz_1^{-1}z_2^{-1}$, which expresses the fact that $e_3=(-1,-1)=-(1,0)-(0,1)=-e_1-e_2$.}

Now observe how $W_X$ changes under the action:
$$
\begin{array}{lll}
W_X(\xi z) &=& \xi z_1 + \cdots +  \xi z_n + d_{n+1}( \xi z) + \cdots + d_r( \xi z) \\
 &= & \xi z_1 + \cdots +  \xi z_n + \xi^{\langle e_{n+1},(1,\ldots,1) \rangle} d_{n+1}( z) + \cdots + \xi^{\langle e_{r},(1,\ldots,1) \rangle} d_r( z).
\end{array}
$$
since, in the above notation, $d_j(\xi z)=\xi^{a_1+\cdots +a_n} z^{e_j}$. Thus $W_X(\xi z) = \xi W_X(z)$ follows if we can show that $a_1+\cdots +a_n$ is congruent to $1$ modulo $\lambda_X$.
%

The $\Z$-linear relation amongst edges $e_j - a_1 e_1 - \cdots - a_n e_n=0$ corresponds to a class $\gamma\in H_2(X)$ determined by the intersection products $\gamma\cdot D_i = -a_i$ for $i\leq n$ and $\gamma\cdot e_j=1$ (see Section \ref{Subsection Review of the Batyrev-Givental presentation of QH}). Since $c_1(TX)=\sum \mathrm{PD}[D_i]$, $c_1(TX)(\gamma)=1-a_1-\cdots-a_n$. By monotonicity, $c_1(TX)=\lambda_X\omega_X$, so $1-a_1-\cdots-a_n$ is an integer multiple of $\lambda_X$, since $\omega_X$ is integral.

Since $z^{e_i}$ gets rescaled by $\xi$ via $z\mapsto \xi z$, the proof also holds for the $F$-twisted superpotential (note we use $c_1(TX)=\lambda_X\omega_X$ at the end of the proof, not the non-monotone $\omega_F$). 
\end{proof}

\begin{corollary} \label{Corollary crit pts of W arise in families}
For $(X,\omega_X)$ as above, the critical points of $W_X$ arise in $\lambda_X$-families: if $p\in \mathrm{Crit}(W_X)$ then $\xi p\in \mathrm{Crit}(W_X)$ with $W_X(\xi p)=\xi W_X(p)$, whenever $\xi^{\lambda_X}=1.$
\end{corollary}
\begin{proof}
 This follows, since $d(W_X(\xi z))=d(\xi W_X(z)) = \xi d(W_X(z))$.
\end{proof}
%
%
\subsection{The barycentre of the moment polytope for monotone toric manifolds}
\label{Subsection The barycentre of the moment polytope for monotone toric manifolds}

Fukaya-Oh-Ohta-Ono showed in \cite[Theorem 7.11]{FOOOtoric} that for closed monotone toric manifolds $B$, the $p\in \mathrm{Crit}(W)$ always land at the barycentre of $\Delta$ and the corresponding $L_p$ is the unique monotone Lagrangian torus fibre of the moment map. This can be proved as follows.

\begin{lemma}\label{Lemma barycentre trick}
When $B$ is a closed monotone toric manifold, $p\in \mathrm{Crit}(W)$ always lands at the barycentre of $\Delta$ and $L_p$ is the unique monotone Lagrangian torus fibre of $\mu_B$.
\end{lemma}
\begin{proof}
Take the reflexive polytope $\Delta$ for $B$ (Section \ref{Subsection The polytope of a Fano variety}), so $W=t(\sum z^{e_i})$. Now simply calculate the critical points $z\in (\C^*)^n$ of $\sum z^{e_i}: (\C^*)^{n} \to \C$. So $y=\mathrm{val}_t(z)=0$, which is the barycentre. By Delzant's theorem, any other choice of polytope for $B$ is $A \Delta + \textrm{constant}$ for $A\in SL(n,\Z)$, but such transformations preserve the barycentre.

To prove monotonicity of $L$, we need $\omega[u]= (1/2\lambda_B)\mu[u]$ for discs $u: (\D,\partial \D) \to (B,L)$ (recall $2\lambda_L\lambda_B=1$). By the long exact sequence $\pi_2(L)=0\to \pi_2(B) \to \pi_2(B,L) \to \pi_1(L)\to 0=\pi_1(B)$ (using that toric $B$ are simply connected, and that a torus $L$ has $\pi_2(L)=0$), a basis for such discs are the standard $\mu=2$ discs $u_i$ in the homogeneous coordinate $x_i$ (keeping the other $x_j$ fixed). A calculation \cite[Theorem 8.1]{Cho-Oh} shows that $\omega[u_i]=\langle y,e_i \rangle -\lambda_i$ (without $2\pi$ due to Remark \ref{Remark about missing 2pi}). For the barycentre $y$ the latter equals $1/\lambda_B$, so $\omega[u_i]=(1/2\lambda_B)\mu[u_i]$ as required.
\end{proof}

\begin{theorem}\label{Theorem barycentre of E is critical}
For any admissible toric manifold $M$ (Definition \ref{Definition admissible toric manifolds}, e.g. a monotone toric negative line bundle $E\to B$), using the toric monotone form $\omega_{\Delta}=\sum \mathrm{PD}[D_i] = c_1(TM)$ so that $\lambda_i=-1$ define the polytope $\Delta$ for $M$, the critical points of the superpotential all lie over $y=0\in \Delta$.
More generally, for any polytope $\Delta_M$ for $M$ inducing a monotone toric symplectic form $\omega_M$ (so $c_1(TM)=\lambda_M \, [\omega_M]$), the critical points of the superpotential lie over the ``barycentre'' $y$ of $M$ defined as the unique point $y\in \Delta_M$ satisfying
\begin{equation}\label{EqnBarycentreEquation}
\qquad\qquad\qquad
\langle y,e_i \rangle - \lambda_i = \tfrac{1}{\lambda_M}.
\end{equation}
Moreover, the corresponding Lagrangians $L_p$ are monotone.
\end{theorem}
\begin{proof}
 Follows by the same proof as in Lemma \ref{Lemma barycentre trick}, using that $\Delta_M = \tfrac{1}{\lambda_M}(A \Delta + \textrm{constant})$.\\
\emph{{\bf Sanity check.} Let's run the calculation of the critical points of $W$, as described above Example \ref{Example Superpotential of nlb}. By Corollary \ref{Corollary crit values of W are homogeneous} we expect critical points $z=t^y\cdot c$ with $c\in (\C^*)^n$ and critical value $W\in T\cdot \C^*$ (recall $T=t^{1/\lambda_M}$). Indeed evaluating: 
\begin{equation}\label{EqnWmonotone}
W=\sum t^{-\lambda_i}z^{e_i} = \sum t^{-\lambda_i + \langle y,e_i \rangle} c^{e_i} = t^{1/\lambda_M} \sum c^{e_i} \in T\cdot \C^*,
\end{equation}
where in the last equality we used the barycentre equation \eqref{EqnBarycentreEquation}.}
\end{proof}

\begin{corollary}\label{Corollary Lagrangian tori in E in SE}
 The critical points of $W_E$ all give rise (with various holonomy data) to the unique monotone Lagrangian torus $L$ in the sphere bundle $SE\subset E$ of radius $1/\sqrt{\pi \lambda_E}$, which projects to the monotone Lagrangian torus in $B$ lying over the barycentre of $\Delta_B$.
\end{corollary}
\begin{proof}
 By Theorem \ref{Theorem barycentre of E is critical}, the last entry of $\mu_E(L)$ is $\mathrm{val}_t(z_{r+1})=\langle y, e_{r+1}\rangle = \tfrac{1}{\lambda_E}$ (using that $\lambda_{r+1}^E=0$ and $e_{r+1}=(0,\ldots,0,1)$). By Lemma \ref{Lemma forms agree for line bundle}, the Hermitian norm of points in $L$ is $1/\sqrt{\pi \lambda_E}$. We now verify that the projection to $B$ is as claimed. %
Since for critical $p$ of $W_E$, the Lagrangian $L_p$ lies over the barycentre of $\Delta_E$, we have $\omega_E[u_i^E]=1/\lambda_E$ for the standard discs $u_i^E$ bounding $L_p$ (see the proof of Lemma \ref{Lemma barycentre trick}). The projection $\pi: E \to B$ forgets the last homogeneous coordinate $x_{r+1}$, so $\pi(u_i^E)=u_i^B$ are the standard discs in $B$ bounding $\pi(L_p)$ for $i=1,\ldots,r$. Since $x_{r+1}$ is constant on $u_i^E$ for those $i$, $\omega_E|_{u_i^E}=(1+k\pi r^2)\pi^*\omega_B$ where $r$ is the radius of the sphere bundle $SE$ where $L_p$ lies. Thus
$$
(1+k\pi r^2)  \omega_B[u_i^B] = \omega_E[u_i^E] = \tfrac{1}{\lambda_E}.
$$
By the proof of Lemma \ref{Lemma barycentre trick}, $\pi(L_p)$ lies over the barycentre of $\Delta_B$ precisely if $\omega_B[u_i^B]=1/\lambda_B$ for $i=1,\ldots,r$. This holds precisely if $\lambda_B = (1+k\pi r^2)\lambda_E$ which, using $\lambda_E=\lambda_B-k$, is equivalent to $r=1/\sqrt{\pi \lambda_E}$ as claimed.
\end{proof}

\section{Appendix B: Matrix perturbation theory and Grassmannians}
\label{Section Appendix B: Matrix perturbation theory}
\subsection{Set up of the eigenvalue problem}
\label{Subsection Set up of the eigenvalue problem}
(Needed in Section \ref{Subsection Generation for semisimple critical values}) A family of matrices $$A(x)\in \mathrm{End}(\C^n)$$ is given, depending holomorphically on a complex parameter $x$ near $x=0$, and the generalized eigenspace decomposition of $\C^n$ for $A(x)$ needs to be compared with that of $A(0)$.

When $A(x)$ are normal matrices, this problem is rather well-behaved \cite[Ch.2 Thm 1.10]{Kato}: the eigenvalues and the projections onto the eigenspaces are all holomorphic in $x$. However, in our situation the $A(x)$ are not normal ($A(0)$ will typically not be diagonalizable), and the outcome is significantly more complicated. In this case, the eigenvalues $\lambda_j(x)$ and the \emph{eigenprojections} $P_j(x)$ (the projection onto the generalized eigenspace $\ker (A(x)-\lambda_j(x))^n$) are only holomorphic when $x$ is constrained to lie in a simply connected domain $D$ which does not contain \emph{exceptional points}, that is points $x=x_0$ where the characteristic polynomial $\chi_{A(x_0)}$ has repeated roots. The examples in Kato \cite[Ch.2 Example 1.1 and 1.12]{Kato} show that near $x_0$, the $\lambda_j(x),P_j(x)$ can have branch points, and the $P_j(x)$ typically have poles at $x_0$ (the $\lambda_j(x)$ are always continuous at $x_0$). Even when $x_0$ is not a branch point of $\lambda_j(x)$, then although $\lambda_j(x)$ is holomorphic on a whole disc around $x_0$ and $P_j(x)$ is single-valued on a punctured disc around $x_0$, nevertheless $P_j(x)$ can have a pole at $x_0$! For example, 
\begin{example} \label{Example bad matrix}
$\left(\begin{smallmatrix} x & 1 \\ 0 & 0\end{smallmatrix}\right)$ has eigenvalues $\lambda_1(x)=x$, $\lambda_2(x)=0$, eigenvectors $v_1(x)=(1,0)$, $v_2(x)=(1,-x)$, but eigenprojections $P_1(x)=\left(\begin{smallmatrix} 1 & x^{-1} \\ 0 & 0\end{smallmatrix}\right)$ and $P_2(x)=\left(\begin{smallmatrix} 0 & -x^{-1}\\ 0 & 1\end{smallmatrix}\right)$ which blow up at the exceptional point $x=0$.
\end{example} 
Athough one can \cite[Ch.2 Sec.4.2]{Kato} construct a basis of holomorphic generalized eigenvectors $v_j(x)$ for $P_j(x)\cdot \C^n$ in a domain $D$, as above, they cannot in general be constructed near exceptional $x_0$. We believe that the correct setting, to fix these convergence issues, is to instead consider how the eigenspaces and generalized eigenspaces vary in the relevant Grassmannian -- but it seems this is missing in the literature, so we develop it here.

In our notation, $x=0$ will be a possibly exceptional point of $A(x)$ and we write 
$$\D=\{x\in \C: |x|\leq \varepsilon\}\qquad \D^{\times}=\{x\in \C: 0<|x|\leq \varepsilon\}$$
to mean a small disc around $x=0$ (resp. punctured disc) containing no other exceptional points. The purpose of this appendix is two-fold:
\begin{enumerate}
 \item Suppose $\lambda$ is a \emph{semisimple eigenvalue} of $A(0)$, meaning the generalized eigenspace for $\lambda$ coincides with the eigenspace for $\lambda$ (for example, this holds for diagonalizable $A(0)$). Then the $\lambda_j(x)$ with $\lambda_j(0)=\lambda$ are all holomorphic on $\D$. If moreover their derivatives $\lambda_j'(0)$ are all distinct, then the eigenprojections $P_j(x)$ are also holomorphic on $\D$ and there is a linearly independent collection of holomorphic eigenvectors $v_j(x)$ on $\D$ converging to a basis of the $\lambda$-eigenspace of $A(0)$. In particular, $E_j(x)=\C v_j(x)=P_j(x)\cdot \C^n$ vary holomorphically in projective space $\P(\C^n)$ and their sum converges to a decomposition of the $\lambda$-eigenspace of $A(0)$ into $1$-dimensional summands.

 \item More generally, suppose $A(0)$ has a generalized sub-eigenspace $GE_i(0)\subset \ker (A(0)-\lambda)^n$, corresponding to a $k\times k$ Jordan block of $A(0)$ for $\lambda$. Then we want to show that the eigenvalues $\lambda_j(x)$ arise in families indexed by $j\in J_i=\{j_1,\ldots,j_k\}$ such that $E_{j_1}(x)\oplus \cdots \oplus E_{j_k}(x)$ converges continuously to $GE_i(0)$ in the Grassmannian $\mathrm{Gr}_{k}(\C^n)$ of $k$-dimensional vector subspaces of $\C^n$ as $x\to 0$, where $E_j(x)= P_j(x)\cdot \C^n$. Moreover, each $E_j(x)$, for $j\in J_i$, converges in $\P(\C^n)$ to the $1$-dimensional sub-eigenspace $E_i(0)\subset GE_i(0)$.
\end{enumerate}

In particular, this implies that if a map
$$
\mathcal{OC}_j(x) \in \mathrm{End}(\C^n)
$$
depends continuously on $x$ and is non-zero only in the summand $E_j(x)= P_j(x)\cdot \C^n$ of $\C^n=\oplus E_i(x)$ for $x\neq 0$, then this also holds at $x=0$: $\mathrm{Image}(\mathcal{OC}_j(0))\subset E_j(0)$.

Moreover, if $(A(x)-\lambda_j(x))\cdot \mathcal{OC}_j(x)=0$ for $x\neq 0$, then the same holds for $x=0$. That is, if the image of $\mathcal{OC}_j(x)$ consists only of eigenvectors (rather than \emph{generalized} eigenvectors) for $x\neq 0$ then the same holds for $x=0$.

\subsection{The case when $\lambda$ is a semisimple eigenvalue of $A(0)$}
We recall from Kato \cite[Ch.2]{Kato} some basic properties of matrices $A(x)$ depending holomorphically on $x\in \C$. We will assume that $x=0$ is an exceptional point (that is, $\chi_{A(0)}$ has a repeated root $\lambda$), since otherwise $\lambda_j(x),P_j(x)$ extend holomorphically over $x=0$ and there is nothing to prove. The eigenvalues $\lambda_1(x),\ldots,\lambda_n(x)$ of $A(x)$ are branches of analytic functions on a punctured disc around $x=0$, but they are continuous even at $x=0$. 

The $\lambda_j(x)$ are holomorphic on simply connected domains avoiding exceptional points, therefore they define multi-valued analytic functions near $x=0$. Analytic continuation, as $x$ travels around $0$, will permute the functions $\lambda_j(x)$ and so we can reindex them so that the permutation is described by the following disjoint cycle decomposition
$$
(\lambda_1(x),\ldots,\lambda_p(x))\cdot (\lambda_{p+1}(x),\ldots,\lambda_{p+p'}(x))\cdot \cdots
$$
where the periods of the cycles are $p,p',\ldots$. The $\lambda_j(x)$ can then be expanded in Puiseaux series around $x=0$ (which is a branch point of the same period as the relevant cycle),
namely for $j=1,\ldots,p$:
$$
\lambda_j(x) = \lambda + c_1 \xi_p^j x^{1/p} + c_2 \xi_p^{2j} x^{2/p} + \cdots
$$
where $\xi_p=e^{2\pi i/p}$ and $c_1,c_2,\ldots\in \C$ are constants. We recall that the $\lambda_j(x)$ are continuous everywhere in $x$, and their value at $0$ is some eigenvalue $\lambda$ of $A(0)$.

The eigenprojections $P_j(x)=\tfrac{1}{2\pi i} \int_{\Gamma_j} R(z,x)\, dz$ are defined by integrating the resultant $R(z,x)=(A(x)-z)^{-1}$ around a circle $\Gamma_j$ surrounding $\lambda_j(x)$ but no other $\lambda_i(x)$. The $P_j(x)$ have branch point $x=0$ of period $p$ just like the $\lambda_j(x)$, so $P_1(x)+\cdots+P_p(x)$ will be single-valued on a punctured disc around $0$. However, $P_1(x)+\cdots+P_p(x)$ may have a pole at $x=0$, as was the case in example \ref{Example bad matrix}.

However, suppose $\lambda_1(x),\ldots,\lambda_r(x)$ are all the eigenvalues which converge to $\lambda$ at $x=0$, then the \emph{total projection}
$$
P^{\mathrm{tot}}_{\lambda}(x) = P_1(x)+ \cdots + P_r(x)
$$
is single-valued and holomorphic even at $x=0$. In particular, $P^{\mathrm{tot}}_{\lambda}(0)$ is the eigenprojection onto the generalized eigenspace of $A(0)$ for $\lambda$. In example \ref{Example bad matrix}, $P^{\mathrm{tot}}_0(x)=P_1(x)+P_2(x)=\mathrm{id}$.

Suppose now that $\lambda$ is a semisimple eigenvalue of $A(0)$, that is the Jordan blocks of $A(0)$ for $\lambda$ all have size $1\times 1$. We work with indices $j$ such that $\lambda_j(0)=\lambda$. Semi-simplicity of $\lambda$ ensures that the holomorphic function $(A(x)-\lambda)P^{\mathrm{tot}}_{\lambda}(x)$ vanishes at $x=0$, so
$$\widetilde{A}(x)=\tfrac{1}{x}(A(x)-\lambda)P^{\mathrm{tot}}_{\lambda}(x)$$
is well-defined at $x=0$ and is holomorphic in $x$. Let 
$\widetilde{\lambda}_j(x)$ denote the eigenvalues of $\widetilde{A}(x)$ restricted to $\mathrm{Image}(P^{\mathrm{tot}}_{\lambda}(x))$. The reduction process \cite[Ch.2 Sec.2.3]{Kato} shows that the $\widetilde{\lambda}_j(x)$ are related to the $\lambda_j(x)$ for $A(x)$ by:
$$
\lambda_j(x) = \lambda + \widetilde{\lambda}_j(0) \cdot x + c_j x^{1+\tfrac{1}{p_j}} + \cdots
$$
where $p_j$ is the period of the cycle for $\widetilde{\lambda}_j(x)$. In particular, it follows that $\lambda_j(x)$ is not just continuous at $x=0$, but also differentiable at $x=0$ with derivative $\lambda_j'(0) = \widetilde{\lambda}_j(0)$. 

This breaks up the collection $\lambda_1(x),\ldots,\lambda_r(x)$ into smaller collections depending on the value of the derivative $\lambda_j'(0)=\widetilde{\lambda}$, which is an eigenvalue of $\widetilde{A}(0)$. The projection $\widetilde{P}_j(x)$ for $\widetilde{A}(x)$ in fact coincides with the projection $P_j(x)$, therefore if $\lambda_1(x),\ldots,\lambda_s(x)$ are the eigenvalues of $A(x)$ with derivative $\widetilde{\lambda}$ at $x=0$, then
$$
P_1(x) + \cdots + P_s(x) = \widetilde{P}_1(x) + \cdots + \widetilde{P}_s(x) = \widetilde{P}^{\mathrm{tot}}_{\widetilde{\lambda}}(x)
$$
is single-valued and holomorphic even at $x=0$ since it is the total projection for $\widetilde{A}(x)$ for $\widetilde{\lambda}$. If all $\lambda_j'(0)$ are distinct (for those $j$ with $\lambda_j(0)=\lambda$), then $s=1$ above, and so all $P_j(x)$ are holomorphic even at $x=0$. So by \cite[Ch.2 Sec.4.2]{Kato}, one can then construct holomorphic eigenvectors $v_j(x)$ for each $P_j(x)$ on a disc around $x=0$. Thus we obtained:

\begin{corollary}\label{Corollary Cgce of espaces for semisimple evalue}
If $A(x)\in \mathrm{End}(\C^n)$ depends holomorphically on $x\in \C$ for small $x$, and $\lambda$ is a semisimple eigenvalue of $A(0)$, then the eigenvalues $\lambda_j(x)$ with $\lambda_j(0)=\lambda$ are differentiable at $x=0$. If $\lambda_j'(0)$ are all distinct, then the eigenprojections $P_j(x)$ are single-valued near $x=0$ and holomorphic even at $x=0$, and there is a linearly independent collection of holomorphic eigenvectors $v_j(x)$ on a disc around $0$ converging to a basis of the $\lambda$-eigenspace of $A(0)$. In particular, $E_j(x)=\C v_j(x)=P_j(x)\cdot \C^n$ vary holomorphically in projective space $\P(\C^n)$ and their sum converges to a decomposition of the $\lambda$-eigenspace of $A(0)$ into $1$-dimensional summands.
\end{corollary}

\subsection{When $\lambda$ is a non-semisimple eigenvalue of $A(0)$}
\label{Subsection non-semisimple eigenvalue}
(Used in Sections \ref{Subsection The motivation for requiring the superpotential to be Morse} and \ref{Subsection OC00 lands in espace}) 

When $\lambda$ is not semisimple, the classical literature on matrix perturbation theory does not appear to address the issue of convergence of eigenspaces and generalized eigenspaces at exceptional points.

\begin{example}
Let  $A(x)=\left(\begin{smallmatrix} 0 & 0 & 0 \\ 0 & x & 1 \\ 0 & 0 & 0\end{smallmatrix} \right)$. The eigenvalues are $\lambda_1(x)=0$, $\lambda_2(x)=x$, $\lambda_3(x)=0$, and $x=0$ is an exeptional value (the lower-right $2\times 2$ block is the matrix from Example \ref{Example bad matrix}). The eigenvectors $v_1(x)=e_1$, $v_2(x)=e_2$, $v_3(x)=e_2-xe_3$ converge to $v_1=e_1$, $v_2= v_3=e_2$. But in fact the spans of the bases for each Jordan block converge in the relevant Grassmannian: $\C v_1(x) = \C e_1 \to \C e_1$ in $\P(\C^n)$, and $\C v_1(x) + \C v_2(x) = \C e_2+ \C e_3 \to \C e_2+ \C e_3$ in $\mathrm{Gr}_2(\C^n)$. It is a ``singular'' linear combination $\tfrac{1}{x}(v_2(x)-v_3(x))=e_3$ of $v_2(x),v_3(x)$ that will converge to the generalized eigenvector $e_3$ of $A(0)$. 
\end{example}

 Recall that for a linear map $A\in \mathrm{End}(\C^n)$, 
$$\C^n = \oplus \, GE_{\lambda}(A)$$
 is a direct sum of the generalized eigenspaces $GE_{\lambda}(A)=\ker (A-\lambda)^n$ as $\lambda$ varies over the eigenvalues of $A$. The $GE_{\lambda}(A)$ can be further decomposed into vector subspaces
$$
GE_{\lambda}(A) = J_{\lambda,1}(A) \oplus \cdots \oplus J_{\lambda,r}(A)
$$
corresponding to the Jordan blocks of the Jordan canonical form of $A$, namely each $J_{\lambda,j}(A)$ intersects the eigenspace $E_{\lambda}(A)=\ker (A-\lambda)$ in a one-dimensional subspace $E_j(A)=\C\cdot v_j$, and there is a basis $v_j,v_{j,2},\ldots,v_{j,k}$ for the $k$-dimensional subspace $J_{\lambda,j}(A)$ such that 
$$
(A-\lambda)^{m}v_{j,m}=0 \quad \textrm{ and } \quad
(A-\lambda)^{m-1}v_{j,m}\neq 0 \in E_j(A)=\C\cdot v_j.
$$
We call such a generalized eigenvector $v_{j,m}$ an \emph{$m$-gevec over $v_j$}.

Let $A(x)\in \mathrm{End}(\C^n)$ depend holomorphically on $x\in \C$ for small $x$. 
Recall the eigenvalues $\lambda_j(x)$ are continuous everywhere, and are branches of multivalued analytic functions on $\D^{\times}$. The eigenprojections $P_j(x)$ are branches of analytic matrix-valued functions on $\D^{\times}$.
 
 We will henceforth assume that $A(x)$ is \emph{not permanently degenerate} on $\D^{\times}$, meaning the characteristic polynomial $\chi_{A(x)}$ has distinct roots for at least some $x\in \D^{\times}$. This implies that the eigenvalues $\lambda_j(x)$ are not pairwise identical analytic functions on any simply connected subset of $\D^{\times}$. Although this assumption is presumably not essential for our argument, it will always hold in our applications, and it does simplify the discussion. 
 
 Since the analytic functions $\lambda_j(x)$ are pairwise distinct, they are pointwise distinct except at finitely many points $x$ (the exceptional points). So, making $\varepsilon$ smaller if necessary, we can assume that the values $\lambda_j(x)$ are all distinct at each $x\in \D^{\times}$. Hence $A(x)$ is diagonalizable for $x\in \D^{\times}$. In particular, each $P_j(x)$ is then the projection onto a $1$-dimensional eigenspace $E_j(x)=P_j(x)\cdot \C^n$ of $A(x)$, for $x\in \D^{\times}$. By \cite[Ch.2 Sec.4.2]{Kato}, the $P_j(x)$ therefore give rise to eigenvectors $v_j(x)$ which are branches of analytic vector-valued functions on $\D^{\times}$. Therefore 
$$\D^{\times}\setminus \R_{< 0} \to \P(\C^n), \; x\mapsto \C v_j(x)$$
is a holomorphic map into projective space, where we made a cut along the negative real axis to deal with the branching issue (as $x$ travels around $0$, the $P_j(x)$ will cycle according to how the $\lambda_j(x)$ cycle, and so $\C v_j(x)=E_j(x)=P_j(x)\cdot \C^n$ will cycle through different eigenspaces). We claim the above map extends continuously over $x=0$:

\begin{lemma}\label{Lemma Evecs converge}
Let $v_1,v_2,\ldots,v_r$ be a basis of eigenvectors of $A(0)$. Then each $v_j(x)$ converges continuously in $\P(\C^n)$, $P_j(x)\cdot \C^n=\C v_j(x) \to \C v_i$ as $x\to 0$ in $\D^{\times}\setminus \R_{< 0}$, for some unique $i\in \{ 1,\ldots,r \}$. Conversely, each $v_i$ arises as such a limit: $\C  v_j(x) \to \C v_i$ for some $j$.

Thus, for each $v_i$ there is a non-empty collection of $j\in J_i$ for which 
$$E_j(x)=P_j(x)\cdot\C^n=\C v_j(x)\to \C v_i=E_i(0),$$ 
in $\P(\C^n)$ as $x\to 0$. If $j\in J_i$ then also $k\in J_i$ for any $\lambda_k(x)$ arising in the cycle for $\lambda_j(x)$. That is, the various branches of $v_j(x)$ all converge in $\P(\C^n)$ to the same $v_i$.
\end{lemma}
\begin{proof}
Suppose $\lambda_j(x)\to \lambda$ as $x\to 0$. 
 Observe that the operator $A(0)-\lambda$ is bounded away from zero on $S\setminus \cup C_i$, where $S$ is the unit sphere and $C_i$ are disjoint neighbourhoods of $S^1\cdot \tfrac{v_i}{\| v_i \|}$ (here $S^1$ is multiplication by phases $e^{i\theta}$). Indeed, $\cup C_i$ is a neighbourhood of the zero set $\cup S^1\cdot \tfrac{v_i}{\| v_i \|}$ of $A(0)-\lambda$ (the eigenvectors). By continuity, also $A(x)-\lambda_j(x)$ on $S\setminus \cup C_i$ is bounded away from zero for small $x$. Hence $v_j(x)$ is trapped inside some $\C C_i$ for all small enough $x$, since $A(x)-\lambda_j(x)$ vanishes on $v_j(x)$. It follows that $\C v_j(x)$ converges to $\C v_i$ in $\P(\C^n)$ as $x\to 0$. 

To show that all $v_i$ arise as limits, we first make a small perturbation of $A(x)$. Call $A_{\varepsilon}(x)=A(x) + \varepsilon_1 \mathrm{id}_1 + \cdots + \varepsilon_r \mathrm{id}_r$, where $\mathrm{id}_i$ is the projection operator for the summand $J_{\lambda,i}(A(0))$ in the above generalized eigenspace decomposition of $A(0)$. Thus $A_{\varepsilon}(0)$ differs from $A(0)$ by having added $\varepsilon_i$ to each diagonal entry of the $i$-th Jordan block of $A(0)$. For small $\varepsilon_i>0$ this has the effect of separating the repeated eigenvalue $\lambda$ according to Jordan blocks. Repeating the argument above, for $A_{\varepsilon}(x)$ in place of $A(x)$, shows that each eigenvector $v_{j,\varepsilon}(x)$ of $A_{\varepsilon}(x)$ is trapped inside some $\C C_i$ for all small $x,\varepsilon$. Running the argument above, but with $A(x),A_{\varepsilon}(x)$ in place of $A(0),A(x)$, and keeping $x$ fixed but letting $\varepsilon\to 0$, shows that $v_{j,\varepsilon}(x)$ is trapped in some $\C C_{k,x}$ for 
small $\varepsilon$, where the $C_{k,x}$ are neighbourhoods of $S^1\cdot \tfrac{v_k(x)}{\| v_k(x) \|}$. Combining these two facts, shows that for those indices $i,k$, we have $\C v_k(x) \to \C v_i$ in $\P(\C^n)$ as $x\to 0$.
\end{proof}

\begin{theorem}\label{Theorem Gevecs converge}
 Let $v_1$ be an eigenvector of $A(0)$, and let $v_{1,m}$ be $m$-gevecs over $v_1$, so $v_1,v_{1,2},\ldots,v_{1,k}$ is a basis of a Jordan summand $J_1(A(0))$ as described above. By Gram-Schmidt, we may assume $v_1,v_{1,2},\ldots,v_{1,k}$ are orthonormal.
 Suppose $v_1(x),v_2(x),\ldots$ are the eigenvectors of $A(x)$ which converge to $v_1$ in $\P(\C^n)$ (see Lemma \ref{Lemma Evecs converge}). Then the Gram-Schmidt process applied to $v_1(x),v_2(x),\ldots$ will produce orthonormal vectors $w_1(x),w_2(x),\ldots$ such that $\C w_j(x) \to \C v_{1,j}$ in $\P(\C^n)$, and $\C w_1(x)+\cdots+\C w_k(x)\to J_1(A(0))$ in $\mathrm{Gr}_k(\C^n)$.
\end{theorem}
\begin{proof}
 Let $w_1(x)=\tfrac{v_1(x)}{\| v_1(x)\|}=v_1(x)$, and recall $w_2(x)=\tfrac{W_2(x)}{\| W_2(x)\|}$ where $$W_2(x)=v_2(x) - \langle v_2(x),w_1(x) \rangle w_1(x).$$ Then $(A(x)-\lambda_2(x))w_2(x)=\tfrac{1}{\| W_2(x)\|} (- \langle v_2(x),w_1(x) \rangle (\lambda_1(x)-\lambda_2(x)))\, w_1(x)$. Observe that $\langle v_2(x),w_1(x) \rangle$ is non-zero for small $x$ since $\C v_2(x) \to \C v_1$, $\C w_1(x)=\C v_1(x)\to \C v_1$ and $\langle v_1,v_1 \rangle \neq 0$.
Since the unit sphere $S\subset \C^n$ is compact, a subsequence $w_2(\delta_n)$ will converge as $\delta_n\to 0$. The limit $w_2$ will be a unit vector orthogonal to $v_1$ (since $w_2(\delta_n)$ is orthogonal to $w_1(\delta_n)$ and $\C w_1(\delta_n)\to \C v_1$), satisfying $(A(0)-\lambda)w_2 \in \C v_1$ (since $(A(x)-\lambda_2(x)) w_2(x)\in \C w_1(x)\to \C v_1$).
So $w_2$ is either an eigenvector of $A(0)$ orthogonal to $v_1$ or a $2$-gevec over $v_1$ for $A(0)$ orthogonal to $v_1$. The first case can be ruled out if the eigenvectors of $A(0)$ are never orthogonal to each other (which is a generic condition). Even if this is not the case for the standard inner product on $\C^n$, we can always perturb the inner product so that this condition holds. It follows a posteriori that also for the standard inner product the limit $w_2$ was not an eigenvector.

Therefore $w_2$ is a $2$-gevec over $v_1$ orthogonal to $v_1$. But there is only a $1$-dimensional space of $2$-gevec over $v_1$ orthogonal to $v_1$, namely $\C v_{1,2}$, so $\C w_2 = \C v_{1,2}$ (using that the $v_{1,j}$ are orthonormal). By uniqueness, it follows that $\C w_2(x) \to \C v_{1,2}$ in $\P(\C^n)$ (indeed, if by contradiction $\C w_2(\varepsilon_n)\to \C v_{1,2}$ fails for a sequence $\varepsilon_n\to 0$, then no subsequence of $w_2(\varepsilon_n)$ would be allowed to converge, contradicting compactness of $S$).

The argument now continues by induction. We show one more step of the induction, since the general inductive step is then obvious. The next step of Gram-Schmidt is: $W_3(x)=v_3(x)-\langle v_3(x) , w_2(x)\rangle w_2(x) -\langle v_3(x) , w_1(x)\rangle w_1(x)$ and $w_3(x)$ is the normalization of $W_3(x)$. Then $(A(x)-\lambda_3(x))(A(x)-\lambda_2(x))w_3(x)$ will again be a non-zero multiple of $w_1(x)$, for small $x$. A subsequence $w_3(\delta_n)$ will converge to some $w_3$ which is orthogonal to both $\C v_1,\C v_{1,2}$, and which satisfies $(A(0)-\lambda)^2w_3\in \C v_1$. Now $w_3$ cannot be an eigenvector since it is orthogonal to $v_1$ (arguing as before). So $w_3$ is either a $2$-gevec or a $3$-gevec over $v_1$. We need to exclude the first case. But in the first case, $\C w_3 = \C v_{1,2}$ (since $w_3$ is orthogonal to $v_1$) contradicts that $w_3$ is orthogonal to $v_{1,2}$. Therefore $w_3$ is a $3$-gevec over $v_1$ orthogonal to $\C v_1 + \C v_{1,2}$. But the space of such $3$-gevecs is $1$-dimensional and equal to $\C v_{1,3}$. Thus $\C w_3=\C v_{1,3}$, and so $\C w_3(x) \to \C v_{1,3}$ in $\P(\C^n)$ (again, arguing by contradiction: if a subsequence $w_3(\varepsilon_n)$ makes $\C w_3(\varepsilon_n)\to \C v_{1,3}$ fail, then $w_3(\varepsilon_n)$ would have no convergent subsequence in $S$).

It follows that in $\mathrm{Gr}_k(\C^n)$,
$$\C v_1(x) + \cdots + \C v_k(x) = \C w_1(x) + \cdots + \C w_k(x) \to \C v_1 + \C v_{1,2}+ \cdots + \C v_{1,k} = J_1(A(0)).
$$
In particular, a posteriori, it follows that $v_1(x),v_2(x),\ldots$ involved exactly $k$ vectors.
\end{proof}

\section{Appendix C: The extended maximum principle}
\label{Section The Maximum principle revisited}

\subsection{Symplectic manifolds conical at infinity}
\label{Subsection Symplectic manifolds conical at infinity}

Let $(M,\omega)$ be a symplectic manifold \emph{conical at infinity}. As in \cite{Ritter4}, this means $\omega$ is allowed to be non-exact, but outside of a bounded domain $M_0\subset M$ there is a symplectomorphism
$$
(M\setminus M_0,\omega|_{M\setminus M_0}) \cong (\Sigma\times [1,\infty),d(R\alpha)),
$$
where $(\Sigma,\alpha)$ is a contact manifold, and $R$ is the coordinate on $[1,\infty)$.

We call $\Sigma\times [1,\infty)$ the \emph{collar} of $M$. On the collar, $\omega = d\theta$ is exact with $\theta=R\alpha$. We denote by $Y$ the Reeb vector field on $\Sigma$ (so $\alpha(Y)=1$ and $d\alpha(Y,\cdot)=0$), and by $Z=R\partial_R$ the Liouville vector field on the collar (so $\omega(Z,\cdot)=\theta$).

Recall that an $\omega$-compatible almost complex structure $J$ is called of \emph{contact type} if, for large $R$, it satisfies $JZ = Y$. This is equivalent to the condition 
$$\theta \circ J = dR.$$
\begin{remark}\label{Remark tweaking contact condition Max}
The contact type condition $J Z = Y$ can be slightly generalized to $JZ = c(R) Y$, equivalently $\theta\circ J = c(R) dR$, where $c: \R \to \R$ satisfies $c> 0$ and $c'\geq 0$.
\end{remark}

\subsection{The class of Hamiltonians}
\label{Subsection Symplectic manifolds conical at infinity and The class of Hamiltonians}

Usually to make Floer theory work in this context one would require the Hamiltonians $H: M \to \R$ to depend only on $R$ for large $R$: $H=h(R)$. Then the coordinate $\rho=R\circ u: S \to \R$ will satisfy a maximum principle for Floer solutions $u: S \to M$. The maximum principle is originally due to Viterbo \cite{Viterbo1} for Floer trajectories (so when $S$ is a cylinder), and it holds also when $S$ is a punctured Riemann surface (due to Seidel \cite{Seidel2}, see the final appendix of \cite{Ritter3} for a detailed description). By \emph{Floer solution} $u(s,t)$ we mean a solution of
$$
(du-X_H\otimes \beta)^{0,1}=0
$$
where $X_H$ is the Hamiltonian vector field (so $\omega(\cdot,X_H)=dH$); $\beta=\beta_s\, ds + \beta_t \, dt$ is a $1$-form on $S$ satisfying $d\beta\leq 0$; and the $(0,1)$ part is taken with respect to $J$ (see \cite[Section D.1]{Ritter3} for details). The Hamiltonian orbits that $u$ converges to at the punctures of $S$ (near which $\beta$ is a constant positive multiple of $dt$) are called the \emph{asymptotic conditions}.

The condition that the Hamiltonian depends only on $R$ on the collar is quite restrictive: it implies that $X_H$ is a multiple of the Reeb vector field, $X_H=h'(R)Y$. This is too restrictive in the case of non-compact toric varieties, as one would like to allow the Hamiltonians arising from the natural $S^1$-actions around the various toric divisors, mentioned in the Introduction.

The aim of the Appendix is to prove the following maximum principle.

\begin{theorem}[Extended Maximum Principle]\label{Theorem maximum principle}
 Let $H: M \to \R$ have the form
$$
 H(y,R) = f(y) R
$$
for large $R$, where $(y,R)\in \Sigma\times (1,\infty)$ are the collar coordinates, and $f:\Sigma\to \R$ satisfies
\begin{itemize}
 \item $f$ is invariant under the Reeb flow (that is $df(Y)=0$ for the Reeb vector field $Y$);
 \item $f \geq 0$.\;\;  \emph{[this condition can be omitted if $d\beta=0$]}
\end{itemize}

Let $J$ be an $\omega$-compatible almost complex structure of contact type at infinity.

Then the $R$-coordinate of any Floer solution $u$ is bounded a priori in terms of the $R$-coordinates of the asymptotic conditions of $u$.
\end{theorem}

\subsection{Observations about functions invariant under the Reeb flow}
\label{Subsection Observations about functions invariant under the Reeb flow}
%
The condition that $f: \Sigma \to \R$ is invariant under the Reeb flow is equivalent to 
$$
\boxed{df(Y)=0}
$$
where $Y$ is the Reeb vector field. This is equivalent to the condition
\begin{equation}\label{EqndRXfVanish}
 \boxed{dR(X_f)=0}
\end{equation}
since $Y=X_R$ and $dR(X_f)=\omega(X_f,X_R) = -df(Y)=0$.

For $h=h(R)$ depending only on the radial coordinate, $X_h = h'(R)Y$, therefore
\begin{equation}\label{EqndfXhVanish}
 \boxed{df(X_{h})=0}
\end{equation}
The condition that $f$ only depends on $y\in \Sigma$ and not on $R$ on the collar $\Sigma \times (1,\infty)$ implies
\begin{equation}\label{EqnThetaXfVanishes}
\boxed{\theta(X_f) = 0}
\end{equation}
since $\theta(X_f)= \omega(Z,X_f)=df(Z)=R\partial_Rf =0$ where $Z=R\partial_R$ is the Liouville vector field.

\begin{lemma}\label{Lemma gt preserves alpha,R,f}
 For $H = f(y)R$ on $(\Sigma \times (1,\infty),d(R\alpha))$, with $f:\Sigma \to \R$ invariant under the Reeb flow,
 the flow $g_t$ of $H$ satisfies:
 \begin{enumerate}
  \item $g_t$ preserves $\alpha$ (the contact form on $\Sigma$);
  \item $g_t$ preserves the $R$-coordinate;
  \item $g_t$ preserves $f$.
 \end{enumerate}
This also holds if $f$ is allowed to be time-dependent.
 \end{lemma}
\begin{proof}
 By \eqref{EqnThetaXfVanishes}, $\alpha(X_f)=0$. Since $\alpha(Y)=1$ and $X=X_H = fY + RX_f$, we deduce 
 $$\alpha(X)=f.$$
 By comparing $f\, dR + R\, df = dH = d(R\alpha)(\cdot,X)$ with $(dR\wedge \alpha + R\,d\alpha)(\cdot,X) = f\, dR + R\, d\alpha(\cdot,X)$ (using $dR(X_f)=0$ from \eqref{EqndRXfVanish}), we obtain:
 $$
 d\alpha(\cdot,X)=df.
 $$

 The Lie derivative of $\alpha$ along $X$ can be computed by Cartan's formula:
 $$
 \mathcal{L}_{X}\alpha = di_X \alpha + i_X d\alpha = df - df = 0.
 $$
 Since $g_t^*\mathcal{L}_{X}\alpha=\partial_s|_{s=t} g_s^*\alpha$, it follows that $g_t^*\alpha$ is constantly $\alpha$ (using $g_0=\mathrm{id}$). This proves (1). Claim (2) follows from $dR(X)=0$. Claim (3) follows from $df(X)=0$ (using $df(Y)=0$ and $df(X_f)=\omega(X_f,X_f)=0$).
 
 For time-dependent $f$, $\mathcal{L}_X^t\alpha$ depends on the time-parameter $t$: it is defined by $\mathcal{L}_{X}^t\alpha=\partial_s|_{s=t} g_{s,t}^*\alpha$ where $\partial_s g_{s,t}=X_{f_s R}$ with initial point $g_{t,t}=\mathrm{id}$. Then Cartan's formula, and hence the above argument, still holds. 
\end{proof}

\begin{lemma}
If a diffeomorphism $g_t$ preserves $\alpha$ and $R$ for large $R$, then it also preserves $\theta,d\theta,Y, Z=R\partial_R$, and it preserves the splitting $TM=\ker \alpha \oplus \R Y \oplus \R Z$ (for large $R$).
\end{lemma}
\begin{proof}
It preserves $\theta=R\alpha$, since $g_t$ preserves $R,\alpha$. So $g_t^*d\theta=dg_t^*\theta=d\theta$.

Since $\ker \alpha = \ker g_t^*\alpha$, we get $dg_t(\ker \alpha) \subset \ker \alpha$ and this has to be an equality by dimensions ($dg_t$ is invertible). Since $Y\in T\Sigma$ is the unique vector field determined by $\alpha(Y)=1$, $d\alpha(Y,T\Sigma)=0$, to prove that $dg_t\cdot Y = Y$ we just need to check these conditions for $dg_t\cdot Y$. Now $1=\alpha(Y)=(g_t^*\alpha)(Y)=\alpha(dg_t Y)$ gives the first condition. The second follows from
$$
0=d\alpha(Y,T\Sigma) = d(g_t^*\alpha)(Y,T\Sigma)=g_t^*d\alpha(Y,T\Sigma)= d\alpha(dg_t Y,T\Sigma),
$$
using that $dg_t(T\Sigma)=T\Sigma$ since $g_t$ preserves $R$.

Since $d\theta(Z,\cdot)=\theta$, we obtain
$$
d\theta(Z,\cdot)=\theta=g_t^*\theta = \theta(dg_t \cdot) = 
d\theta(Z,dg_t \cdot) = (g_t^*d\theta)(dg_t^{-1}Z,\cdot) = d\theta(dg_t^{-1}Z,\cdot).
$$
Therefore $dg_t Z = Z$. The claim follows.
\end{proof}

\begin{corollary}\label{Corollary Ham flow preserves Reeb vector field}\label{Corollary contact type at infinity under flow}
 The Hamiltonian flow $g_t$ for $H$ as in Theorem \ref{Theorem maximum principle} satisfies $dg_t = \mathrm{id}$ on $\textrm{span}(Z,Y) \subset T(\Sigma\times (1,\infty))$ for large $R$.
 
 In particular, if $J$ is of contact type at infinity, then so is $g^*J = dg_t^{-1} \circ J \circ dg_t$.
\end{corollary}
\begin{proof}
The first part follows by the previous two Lemmas. For the second part, recall that the contact condition is $JZ = Y$ for large $R$. So $g_t^*J$ is of contact type by the first part.
\end{proof}

\subsection{The no escape Lemma}
\label{Subsection No escape Lemma}

\begin{lemma}\label{Lemma Class of Hams satisfy dR(X) is 0 and H is theta(X)}
The Hamiltonians $H: M \to \R$ of the form $H(y,R) = f(y) R$ for large $R$, where $(y,R)\in \Sigma\times (1,\infty)$, with $f:\Sigma\to \R$ invariant under the Reeb flow,
are precisely the class of Hamiltonians whose Hamiltonian vector field $X$ satisfies:
\begin{enumerate}
 \addtocounter{enumi}{1}
 \item $dR(X)=0$;
 \item $H = \theta(X)$.
\end{enumerate}
\end{lemma}
\begin{proof}
 Recall $\theta(X) = \omega(Z,X)=dH(Z)=R\partial_R H$, where $Z=R\partial_R$ is the Liouville vector field. Thus $\theta(X)=H$ corresponds to $\partial_R \log H = \partial_R \log R$. Integrating yields $H = f(y)R$, for some $f: \Sigma \to \R$. Since 
$$
X = f(y) X_R + R X_f = f(y) Y + R X_f,
$$
the condition $dR(X)=0$ is equivalent to $dR(X_f)=0$, equivalently $df(Y)=0$ by \eqref{EqndRXfVanish}.
\end{proof}

Theorem \ref{Theorem maximum principle} now follows by the following ``no escape lemma'' applied to $u$ restricted to $u^{-1}(V)$, taking $R_0$ greater than or equal to the maximal value that $R$ attains on the asymptotics of $u$ (and we pick $R_0$ generically so that $\partial V$ intersects $u$ transversely).

\begin{lemma}[No escape lemma]\label{Lemma No escape lemma}
Let $H: M \to \R$ be any Hamiltonian. Let $(V,d\theta)$ be the region $R\geq R_0$ in $\Sigma\times (1,\infty)$, so $\partial V= \{ R=R_0 \}$. Assume:
\begin{enumerate}
 \item $J$ is of contact type along $\partial V$;
 \item $dR(X)=0$ on $\partial V$;
 \item $H = \theta(X)$ on $\partial V$;
 \item \label{Item H geq 0 dropped} $H\geq 0$ on $V$ \emph{[this condition can be omitted if $d\beta=0$]}.
\end{enumerate}
where $X=X_H$. Let $S$ be a compact Riemann surface with boundary and let $\beta$ be a $1$-form on $S$ with $d\beta \leq 0$. 
Then any solution $u: S \to V$ of $(du-X\otimes \beta)^{0,1}=0$ with $u(\partial S)\subset \partial V$, is either constant or maps entirely into $\partial V$.
\end{lemma}
\begin{proof}
We run the same argument as in \cite[Section D.5]{Ritter3}. The energy of $u$ is
$$
\begin{array}{lllll}

E(u) &=& \tfrac{1}{2} \int_S \| du - X\otimes \beta \|^2 \textrm{vol}_S
\\
 &=& 
\int_S u^*d\theta - u^*(dH)\wedge \beta
\\
 &=& 
\int_S d(u^*\theta - u^*H\beta) + u^*H\, d\beta
& \textrm{expanding }d(u^*H\beta)
\\
& \leq & 
\int_{S} d( u^*\theta - u^*H \beta)
& \textrm{when } d\beta\neq 0, \textrm{we use } H\geq 0,\, d\beta \leq 0 
\\
& = & 
\int_{\partial S} u^*\theta - u^*H \beta
& \textrm{by Stokes's theorem}
\\
& = &
\int_{\partial S} u^*\theta - u^*\theta(X) \beta
& \textrm{using }H=\theta(X)\textrm{ on }\partial V
\\
& = & 
\int_{\partial S} \theta(du - X\otimes \beta)
\\
& = & 
\int_{\partial S} -\theta J(du - X\otimes \beta)j
& \textrm{since }(du-X\otimes \beta)^{0,1}=0
\\
& = & 
\int_{\partial S} -dR(du- X\otimes \beta)j
& \textrm{since }J \textrm{ is of contact type on }\partial V
\\
& = & 
\int_{\partial S} -dR(du)j
& \textrm{since }dR(X)=0\textrm{ on }\partial V
\\
& \leq & 0 & \textrm{since }\partial V\textrm{ minimizes }R \textrm{ on }V.
\end{array}
$$
For the last line: if $\hat{n}$ is the outward normal along $\partial S\subset S$, then $j\hat{n}$ is the oriented tangent direction along $\partial S$, so $-dR(du)j(j\hat{n})=d(R\circ u)\cdot \hat{n}\leq 0$. Indeed, if we assumed that $J$ was of contact type on all of $V$, then the above argument is just Green's formula \cite[Appendix C.2]{Evans} for $S$: $\int_{\partial S} u^* \theta - u^*\theta(X)\beta = \int_S \Delta (R\circ u) \, ds\wedge dt = \int_{\partial S} \tfrac{\partial (R\circ u)}{\partial\hat{n}}\, dS \leq 0$.

 Since $E(u)\geq 0$ by definition, the above inequality $E(u)\leq 0$ forces $E(u)=0$, and so $du=X\otimes \beta.$ But $X\in \ker dR = T\partial V$, so the image of $du$ lies in $T\partial V$, and so $u$ lies in $\partial V$.
\end{proof}

\begin{remark}[When $d\beta=0$]
 Condition \eqref{Item H geq 0 dropped} can be dropped when $d\beta=0$. This applies to cylinders $(\R\times S^1, \beta=dt)$; continuation cylinders (see Theorem \ref{Theorem Continuation maps max}), and to TQFT operations \cite{Ritter3} where the sum of the weights at the inputs equals the sum at the outputs, e.g. for products: $HF^*(H)\otimes HF^*(H)\to HF^*(2H)$.
\end{remark}

\begin{lemma}[Generalizing contact type]
 In Lemma \ref{Lemma No escape lemma} we can replace the contact type condition by $\theta\circ J = c(R)dR$ for any positive function $c: \R \to (0,\infty)$. Indeed, since we only need the condition along $\partial V$, $\theta\circ J = c\cdot dR$ for a positive constant $c=c(R_0)$ suffices.
\end{lemma}
\begin{proof}
This only affects the integrand in $E(u)$ at the end: instead of $ -dR(du- X\otimes \beta)j$ we get $ -c(R_0)\cdot dR(du- X\otimes \beta)j$, but $c(R_0)>0$ so it does not affect the sign.
\end{proof}

We remark that $J=J_z$ can also be domain-dependent, this does not affect the above proofs as long as $J_z$ is of contact type for each $z$.

\begin{remark}[Time-dependent Hamiltonians]\label{Remark time dependent Hams}
 To achieve non-degeneracy of the $1$-periodic Hamiltonian orbits (a requirement in the proof of transversality for Floer moduli spaces), one must in fact always make a time-dependent perturbation of the Hamiltonian $H=H_t$ near the ends of the surface $S$ \cite[Appendix A.5 and A.6]{Ritter3}. Therefore $H=H_{z}$ will also depend on $s$ near the ends, where $z=s+it$. 
When we expand $d(u^*H\beta)$ in the proof of Lemma \ref{Lemma No escape lemma}, we obtain new terms: $u^*(\partial_t H)\, dt\wedge \beta$ and $u^*(\partial_s H)\, ds \wedge \beta$.

In the case of Floer trajectories, one wants $H=H_t$ and $\beta=dt$ on the whole cylinder $\R \times S^1$, so both those new terms vanish.

For more general Floer solutions on a Riemann surface $S$, we only need the $z$-dependence near the ends \cite[Appendix A.5]{Ritter3} where we can ensure that $\beta = C\, dt$ for a constant $C>0$. Thus, to ensure that we can drop the new term $u^*\partial_s H\, ds \wedge \beta$, we require
\begin{enumerate}
 \addtocounter{enumi}{4}
 \item $\partial_s H \leq 0$.
\end{enumerate}
This condition can always be achieved. Indeed, on a negative end, $\R \times (-\infty,0)$, take $H_t$ to be a very small time-dependent perturbation of $H + \varepsilon R$ for small $\varepsilon>0$, then we can find a homotopy $H_{s,t}$ from $H_t$ to $H$ with $\partial_s H_{s,t}\leq 0$. On a positive end, one considers $H - \varepsilon R$.
\end{remark}

\begin{theorem}[Continuation maps]\label{Theorem Continuation maps max}
Suppose the data $(H,J)=(H_z,J_z)$ now depends on the domain coordinates $z=s+it$ of the cylinder $(\R \times S^1,\beta=dt)$, such that $(H_z,J_z)$ becomes independent of $s$ for large $|s|$. We assume that for large $R$:
\begin{enumerate}
 \item $J$ is of contact type;
 \item $dR(X)=0$;
 \item $H = \theta(X)$;
 \addtocounter{enumi}{1}
 \item \label{Item ds H leq 0} $\partial_s H_z \leq 0$.
\end{enumerate}
Then Theorem \ref{Theorem maximum principle} still applies (via Lemma \ref{Lemma No escape lemma}).
\end{theorem}
\begin{proof}
Since $\beta=dt$, \eqref{Item H geq 0 dropped} is not necessary but we need \eqref{Item ds H leq 0} by Remark \ref{Remark time dependent Hams}.
\end{proof}
%
%
%
%
%
\subsection{Symplectic cohomology using Hamiltonians $\mathbf{H=f(y)R}$}
\label{Subsection Remark about the definition of symplectic cohomology}

Recall that
$SH^*(M)=\varinjlim HF^*(H,J),$
taking the direct limit over continuation maps $HF^*(H,J)\to HF^*(\widetilde{H},\widetilde{J})$ for Hamiltonians linear in $R$ at infinity, as the slope of $H$ is increased, where $J,\widetilde{J}$ are of contact type at infinity. These continuation maps are defined since for large $R$ there is a homotopy $H_s$ from $\widetilde{H}$ to $H$ with $\partial_s H_s \leq 0$ (condition \eqref{Item ds H leq 0} in Theorem \ref{Theorem Continuation maps max}). We now show we can enlarge the class to $H$ as in Theorem \ref{Theorem maximum principle}.
 
\begin{theorem}\label{Theorem SH is limit of f(y)R Hams}
 $SH^*(M)=\varinjlim HF^*(H_k,J_k)$ for any sequence of almost complex structures $J_k$ of contact type at infinity, and any sequence of Hamiltonians $H_k$ on $M$ with
 \begin{enumerate}
  \item  $H_k = f_k(y)R+\textrm{constant}$ for large $R$;
  \item the $f_k:\Sigma \to \R$ are invariant under the Reeb flow;
  \item $\max f_k \leq \min f_{k+1}$ and $\min f_k \to \infty$.
 \end{enumerate}
\end{theorem}
\begin{proof} Adding a constant to $H$ plays no role in Floer theory, so we can ignore the constant in (1).
Condition \eqref{Item ds H leq 0} in Theorem \ref{Theorem Continuation maps max} ensures that continuation maps can be defined for the larger class of Hamiltonians of the form described in Theorem \ref{Theorem maximum principle}. The slope of $f(y)R$ is bounded by $\max_{y\in \Sigma} f(y)$, and a homotopy $H_s$ from $\widetilde{H}=\widetilde{f}(y)R$ to $H=f(y)R$ with $\partial_s H_s\leq 0$ exists provided that
$$\min_{y\in \Sigma} \widetilde{f}(y) \geq \max_{y\in \Sigma} f(y).$$
Any sequence $H_k = f_k(y)R$ with $\max f_k \leq \min f_{k+1}$ and $\min f_k \to \infty$ will be cofinal, and so enlarging the class of Hamiltonians as above does not affect the resulting group $SH^*(M)$.
\end{proof}

\begin{corollary}\label{Corollary SH for pairs gH gJ}
Let $H_j:M\to \R$ be Hamiltonians of the form $H_j=m_j R$ for large $R$, whose slopes $m_j\to \infty$, and let $J$ be of contact type at infinity. Suppose $g_t$ is the flow for a Hamiltonian of the form $K=f(y)R$ for large $R$, with $f:\Sigma \to \R$ invariant under the Reeb flow. Then the pairs $(g^*H_j,g^*J)$ compute symplectic cohomology: $SH^*(M)={\displaystyle \lim_{j\to \infty}} HF^*(g^*H_j,g^*J)$.
\end{corollary}
\begin{proof}
 For large $R$,
 $$g^*H = H\circ g_t - K \circ g_t = m_j R - f(g_t y) R = (m_j-f(y))R,$$
using that $g_t$ preserves $R$ and $f$ (Lemma \ref{Lemma gt preserves alpha,R,f}). Moreover, $g^*J$ is of contact type by Corollary \ref{Corollary contact type at infinity under flow}. The claim follows by Theorem \ref{Theorem SH is limit of f(y)R Hams}.
\end{proof}

As explained in Theorem \ref{Theorem Representation into SH for new Hamiltonians}, when there is a loop of Hamiltonian diffeomorphisms whose Hamiltonian has strictly positive slope at infinity, then $SH^*(M)$ is a quotient of $QH^*(M)$.

\begin{theorem}[{By Ritter \cite{Ritter4} and Theorem \ref{Theorem Representation into SH for new Hamiltonians}}]\label{Theorem When SH is quotient of QH} 
Given $g: S^1 \to \mathrm{Ham}(M,\omega)$ generated by a Hamiltonian of the form $K(y,R) = f(y)R$ for large $R$, with $f:\Sigma\to \R$ invariant under the Reeb flow, and satisfying $\min f > 0$, then the canonical map $c^*:QH^*(M)\to SH^*(M)$ induces an isomorphism of $\Lambda$-algebras
$$
SH^*(M) \cong QH^*(M)/(\textrm{generalized }0\textrm{-eigenspace of }r_{\widetilde{g}}) \cong QH^*(M)_{r_{\widetilde{g}}(1)},
$$
where the last expression is localization at $r_{\widetilde{g}}(1)$ (Section \ref{Subsection Some remarks about localizations of rings}), and $c^*$ is the localization map.\\
\emph{(Here $SH^*(M)$ is restricted to contractible loops, see Section \ref{Subsection Invertibles in the symplectic cohomology}. When $M$ is simply connected, such as toric manifolds $M$, this is all of $SH^*(M)$)}
\end{theorem}


\end{document}